\documentclass[12pt, reqno]{amsart}

\makeatletter
\g@addto@macro{\endabstract}{\@setabstract}
\makeatother

\usepackage{graphics, stackrel}
\usepackage{amsmath, amssymb, amsthm}
\usepackage{graphicx}
\usepackage{verbatim}
\usepackage{amsfonts}
\usepackage{filecontents}
\usepackage{natbib}
\usepackage{mathpazo}
\usepackage{fancyvrb}
\usepackage{color}
\usepackage{mdwlist}
\usepackage{enumitem}

\usepackage[pdftex, citecolor=blue, colorlinks=true, linkcolor=blue]{hyperref}

\usepackage{mathrsfs}
\usepackage{bbm}

\usepackage[left=1.25in, right=1.25in, top=1.0in, bottom=1.15in, includehead, includefoot]{geometry}

\usepackage{multirow}

\DeclareMathOperator{\interior}{int}

\newcommand{\iidsim}{\stackrel {\textrm{ {\sc iid }}} {\sim} }
\newcommand{\1}{\mathbbm 1}

\newcommand*\diff{\mathop{}\!\mathrm{d}}
\renewcommand{\epsilon}{\varepsilon}

\newcommand{\mM}{\mathscr M}

\newcommand{\fF}{\mathscr F}

\newcommand{\zZ}{\mathscr Z}

\newcommand{\XX}{\mathsf X}
\newcommand{\YY}{\mathsf Y}
\newcommand{\ZZ}{\mathsf Z}

\newcommand{\RR}{\mathbbm R}

\newcommand{\NN}{\mathbbm N}
\newcommand{\PP}{\mathbbm P}

\newcommand{\EE}{\mathbbm E \,}

\theoremstyle{plain}

\newtheorem{theorem}{Theorem}[section]
\newtheorem{corollary}{Corollary}[section]
\newtheorem{lemma}{Lemma}[section]
\newtheorem{proposition}{Proposition}[section]

\theoremstyle{definition}

\newtheorem{example}{Example}[section]
\newtheorem{remark}{Remark}[section]

\newtheorem{assumption}{Assumption}[section]

\usepackage{caption}

\usepackage{booktabs,caption,fixltx2e}
\usepackage[flushleft]{threeparttable}

\begin{document}

\title{}

\date{\today}

\begin{center}
  \LARGE 
  Optimal Timing of Decisions: A General Theory Based on Continuation 
  Values\footnote{ Financial support from Australian Research
      Council Discovery Grant DP120100321 is gratefully acknowledged. \\ 
  \emph{Email addresses:} \texttt{qingyin.ma@anu.edu.au}, \texttt{john.stachurski@anu.edu.au} }

  \bigskip
  \normalsize
  Qingyin Ma\textsuperscript{a} and John Stachurski\textsuperscript{b} \par \bigskip

  \textsuperscript{a, b}Research School of Economics, Australian National
  University \bigskip

  \today
\end{center}

\begin{abstract} 

	Building on insights of \cite{jovanovic1982selection} and subsequent authors, we 
	develop a comprehensive theory of optimal timing of decisions based around 
	continuation value functions and operators that act on them. Optimality results 
	are provided under 
	general settings, with bounded or unbounded reward functions. This approach 
	has several intrinsic advantages that we exploit in developing the 
	theory. One is that continuation value functions are smoother than value 
	functions, allowing for sharper analysis of optimal policies and more efficient 
	computation. Another is that, for a range of problems, the continuation value 
	function exists in a lower dimensional space than the value function, mitigating 
	the curse of dimensionality. In one typical experiment, this reduces the 
	computation time from over a week to less than three minutes.

    \vspace{1em}

    \noindent
    \textit{Keywords:} Continuation values, dynamic programming, optimal timing
\end{abstract}


\section{Introduction}

A large variety of decision making problems involve choosing when to act in
the face of risk and uncertainty.  Examples include deciding if or when to accept a
job offer, exit or enter a market, default on a loan, bring a new product to market,
exploit some new technology or business opportunity, or
exercise a real or financial option.  See, for example, \cite{mccall1970},
\cite{jovanovic1982selection}, \cite{hopenhayn1992entry}, 
\cite{dixit1994investment}, \cite{ericson1995markov}, \cite{peskir2006}, 
\cite{arellano2008default}, \cite{perla2014equilibrium}, and \cite{fajgelbaum2015uncertainty}.

The most general and robust techniques for solving these kinds of problems
revolve around the theory of dynamic programming.  The standard machinery
centers on the Bellman equation, which identifies current value in terms of a
trade off between current rewards and the discounted value of future states.
The Bellman equation is traditionally solved by framing the solution as a
fixed point of the Bellman operator.  Standard references include \cite{bellman1969new} and \cite{stokey1989}. Applications of these methods to optimal timing include \cite{dixit1994investment}, \cite{albuquerque2004optimal},
\cite{crawford2005uncertainty}, \cite{ljungqvist2012recursive}, and \cite{fajgelbaum2015uncertainty}.

Interestingly, over the past few decades, economists have initiated
development of an alternative method, based around continuation values, that
is both essentially parallel to the traditional method described above and yet
significantly different in certain asymmetric ways (described in detail below).

Perhaps the earliest technically sophisticated analysis based around
operations in continuation value function space is \cite{jovanovic1982selection}. In 
an incumbent firm's exit decision context, Jovanovic proposes an operator 
that is a contraction mapping on the space of bounded continuous functions, and 
shows that the unique fixed point of the operator coincides with the value of staying 
in the industry for the current period and then behave optimally. Intuitively, this 
value can be understood as the continuation value of the firm, since the firm gives 
up the choice to terminate the sequential decision process (exit the industry) in the 
current period. 

Other papers in a similar vein include \cite{burdett1988declining}, \cite{gomes2001equilibrium},
\cite{ljungqvist2008two}, \cite{lise2012job}, \cite{dunne2013entry}, \cite{moscarini2013stochastic}, and \cite{menzio2015equilibrium}.  All of the results found in these papers are
tied to particular applications, and many are applied rather than technical in
nature.

It is not difficult to understand why economists often focus on continuation
values as a function of the state rather than traditional value functions.
One is economic intuition. In a given context it might be more natural or
intuitive to frame a decision problem in terms of the continuation values
faced by an agent. For example, in a job search context, one of the key
questions is how the reservation wage, the wage at which the agent is indifferent
between accepting and rejecting an offer, changes with economic environments. Obviously, the continuation value, the value of rejecting the current offer,
has closer connection to the reservation wage than the value function, the maximum value of accepting and rejecting the offer.

There are, however, deeper reasons why a focus on continuation values can be
highly fruitful.  To illustrate, recall that, for a given problem, the value
function provides the value of optimally choosing to either act today or wait,
given the current environment.  The continuation value is the value associated
with choosing to wait today and then reoptimize next period, again taking into
account the current environment.  One key asymmetry arising here is that, if
one chooses to wait, then certain aspects of the current environment become
irrelevant, and hence need not be considered as arguments to the continuation
value function.

To give one example, consider a potential entrant to a market who must
consider fixed costs of entry, the evolution of prices, their own productivity
and so on.  In some settings, certain aspects of the environment will be
transitory, while others are persistent.  (For example, in \cite{fajgelbaum2015uncertainty}, prices and beliefs are persistent while fixed costs 
are transitory.) All relevant state components must be included in the value function, 
whether persistent or transitory, since all affect the choice of whether to enter or
wait today.  On the other hand, purely transitory components do not affect
continuation values, since, in that scenario, the decision to wait has already
been made.

Such asymmetries place the continuation value function in a lower dimensional
space than the value function whenever they exist, thereby mitigating the
curse of dimensionality.  This matters from both an analytical and a
computational perspective.  On the analytical side, lower dimensionality can
simplify challenging problems associated with, say, unbounded reward
functions, continuity and differentiability arguments, parametric monotonicity
results, etc.  On the computational side, reduction of the state space by even
one dimension can radically increase computational speed.  For example, while solving a well known version of the job search model in section \ref{ss:js_ls}, the continuation value based approach takes only 171 seconds to compute the optimal policy to a given level of precision, as opposed to more than 7 days for the traditional value function based approach.

One might imagine that this difference in dimensionality between the two
approaches could, in some circumstances, work in the other direction, with the
value function existing in a strictly lower dimensional space than the
continuation value function.  In fact this is not possible.  As will be clear
from the discussion below, for any decision problem in the broad class that we
consider, the dimensionality of the value function is always at least as
large.

Another asymmetry between value functions and continuation value functions is
that the latter are typically smoother.  For example, in a job search problem,
the value function is usually kinked at the reservation wage. However,
the continuation value function can be smooth.  More generally, continuation
value functions are lent smoothness by stochastic transitions, since integration is
a smoothing operation.  Like lower dimensionality, increased smoothness 
helps on both the analytical and the computational side.  On the computational
side, smoother functions are easier to approximate.  On the analytical side,
greater smoothness lends itself to sharper results based on derivatives, as
elaborated on below.

To summarize the discussion above, economists have pioneered the continuation
value function based approach to optimal timing of decisions.  This has been
driven by researchers correctly surmising that such an approach will yield
tighter intuition and sharper analysis than the traditional approach in many
modeling problems.  However, all of the analysis to date has been in the
context of specific, individual applications.  This fosters
unnecessary replication, inhibits applied researchers seeking off-the-shelf
results, and also hides deeper advantages.

In this paper we undertake a systematic study of optimal timing of decisions
based around continuation value functions and the operators that act on them.
The theory we develop accommodates both bounded rewards and the kinds of
unbounded rewards routinely encountered in modeling economic
decisions.\footnote{
		For example, many applications include Markov state processes (possibly with 
		unit roots), driving the state space and various common reward functions (e.g., 
		CRRA, CARA and log returns) unbounded (see, e.g, \cite{low2010wage}, 
		\cite{bagger2014tenure}, \cite{kellogg2014effect}). 
		Moreover, many search-theoretic studies model agent's learning behavior 
		(see, e.g., \cite{burdett1988declining}, \cite{mitchell2000scope}, 
		\cite{crawford2005uncertainty}, \cite{nagypal2007learning}, 
		\cite{timoshenko2015product}).
		To have favorable prior-posterior structure (e.g., both follow normal 
		distributions), unbounded state spaces and rewards are usually required. We 
		show that most of these problems can be handled without difficulty. } 
In fact, within the context of optimal timing, the assumptions placed on the
primitives in the theory we develop are weaker than those found in existing
work framed in terms of the traditional approach to dynamic programming, as 
discussed below.

We also exploit the asymmetries between traditional and continuation value
function based approaches to provide a detailed set of continuity,
monotonicity and differentiability results.  For example, we use the relative
smoothness of the continuation value function to state conditions under which
so-called ``threshold policies'' (i.e., policies where action occurs whenever
a reservation threshold is crossed) are continuously differentiable with respect to
features of the economic environment, as well as to derive expressions for the
derivative.

Since we explicitly treat unbounded problems, our work also contributes to ongoing research on dynamic programming with unbounded rewards. One  
general approach tackles unbounded rewards via the weighted supremum norm. The underlying idea is to introduce a weighted norm in a certain space of candidate functions, and then establish the contraction property for the relevant operator. This theory was pioneered by \cite{boud1990recursive} and has been used in numerous 
other studies of unbounded dynamic programming. Examples include 
\cite{becker1997capital},  \cite{alvarez1998dynamic}, \cite{duran2000dynamic, 
duran2003discounting} and \cite{le2005recursive}.

Another line of research treats unboundedness via the local contraction
approach, which constructs a local contraction based on a suitable sequence of
increasing compact subsets.  See, e.g., \cite{rincon2003existence},
\cite{rincon2009corrigendum}, \cite{martins2010existence} and
\cite{matkowski2011discounted}.  One motivation of this line of work is to deal with 
dynamic programming problems that are unbounded both above and 
below. 

So far, existing theories of unbounded dynamic programming have been confined to   
optimal growth problems. Rather less attention, however, has been paid to the study of optimal timing of decisions. Indeed, applied studies of unbounded problems in 
this field still rely on theorem 9.12 of \cite{stokey1989} (see, e.g., 
\cite{poschke2010regulation}, \cite{chatterjee2012spinoffs}).  Since the assumptions 
of this theorem are not based on model primitives, it is hard to verify in 
applications. Even if they are applicable to 
some specialized setups, the contraction mapping structure is unavailable. A recent 
study of unbounded problem via contraction mapping is \cite{kellogg2014effect}. 
However, he focuses on a highly specialized decision problem with linear 
rewards. Since there is no general unbounded dynamic programming theory in this 
field, we attempt to fill this gap.

Notably, the local contraction approach exploits the underlying structure of the technological correspondence related to the state process, which, in optimal 
growth models, provides natural bounds on the growth rate of the state process, 
thus a suitable selection of a sequence of compact subsets to construct local 
contractions. However, such structures are missing in most sequential decision settings we study, making the local contraction approach inapplicable. 

In response to that, we come back to the idea of weighted supremum norm, which 
turns out to interact well with the sequential decision structure we explore. To obtain an appropriate weight function, we introduce an innovative idea centered on dominating the future transitions of the reward functions, which renders the classical weighted supremum norm theory of \cite{boud1990recursive} as a special case, and leads to simple sufficient conditions that are straightforward to check in applications. 

The intuitions of our theory are twofold. First, when the underlying state process is mean-reverting, the effect of initial conditions tends to die out as time iterates forward, making the conditional expectations of the reward functions flatter than the original rewards. Second, in a wide 
range of applications, a subset of states are conditionally independent of the future 
states, so the conditional expectation of the payoff functions is actually 
defined on a space that is lower dimensional than the state space.\footnote{
		Technically, this also accounts for the lower dimensionality of the continuation 
		value function than the value function, as documented above. Section 
		\ref{s:opt_pol} provides a detailed discussion.}  
In each scenario, finding an appropriate weight function becomes an easier job.

The paper is structured as follows. Section \ref{s:opt_results} outlines the
method and provides the basic optimality results.
Section \ref{s:properties_cv} discusses the properties of the
continuation value function, such as monotonicity and
differentiability. Section \ref{s:opt_pol} explores the
connections between the continuation value and the optimal policy. Section
\ref{s:application} provides a list of economic applications and compares the 
computational efficiency of the continuation value approach and traditional 
approach. Section \ref{s:extension} provides extensions and section \ref{s:conclude} 
concludes. Proofs are deferred to the appendix.

\section{Optimality Results}
\label{s:opt_results}

This section studies the optimality results. Prior to this task, we introduce some 
mathematical techniques used in this paper.

\subsection{Preliminaries}
\label{ss:prel}

For real numbers $a$ and $b$ let $a \vee b := \max\{a, b\}$.  If $f$ and $g$
are functions, then $(f \vee g)(x) := f(x) \vee g(x)$.
If $(\ZZ, \zZ)$ is a measurable space, then
$b\ZZ$ is the set of $\zZ$-measurable bounded functions from $\ZZ$ to
$\RR$, with norm
    $\| f \| := \sup_{z \in \ZZ} |f(z)|$.
Given a function $\kappa \colon \ZZ \to [1, \infty)$,
the \emph{$\kappa$-weighted supremum norm} of $f \colon \ZZ \to \RR$ is
\begin{equation*}
    \| f \| _\kappa 
    := \left\| f/\kappa \right\|
    = \sup_{z \in \ZZ} \frac{|f(z)|}{\kappa(z)}.
\end{equation*}
If $\| f\|_\kappa < \infty$, we say that $f$ is \emph{$\kappa$-bounded}.
The symbol $b_\kappa \ZZ$ will denote the set of all functions from $\ZZ$ to $\RR$ that are both $\zZ$-measurable and $\kappa$-bounded. 
The pair $(b_\kappa \ZZ, \| \cdot \|_\kappa)$ forms a Banach space (see e.g., 
\cite{boud1990recursive}, page 331).

A \emph{stochastic kernel} $P$ on $(\ZZ, \zZ)$ is a map $P \colon
\ZZ \times \zZ \to [0, 1]$ such that $z \mapsto P(z, B)$ is $\zZ$-measurable
for each $B \in \zZ$ and $B \mapsto P(z, B)$ is a probability measure for each
$z \in \ZZ$.  We understand $P(z, B)$ as the probability of a state transition from 
$z \in \ZZ$ to $B \in \zZ$ in one step. 
Throughout, we let $\NN := \{1,2, \ldots \}$ and $\NN_0 := \{ 0\} \cup \NN$. 
For all $n \in \NN$,
$P^n (z, B) := \int P(z', B) P^{n-1} (z, \diff z')$ is the probability of a state 
transition from $z$ to $B \in \zZ$ in $n$ steps.
Given a $\zZ$-measurable function $h: \ZZ \rightarrow \RR$, let
\begin{equation*}
	(P^n h)(z) :=: \EE_z h(Z_n) := \int h(z') P^n(z, \diff z') \mbox{ for all } n \in \NN_0,
\end{equation*}
where $(P^0 h) (z) :=: \EE_z h (Z_0) := h(z)$. When $\ZZ$ is a Borel subset of 
$\RR^m$, a \emph{stochastic density kernel} (or \emph{density kernel}) on $\ZZ$ is 
a measurable map $f:\ZZ \times \ZZ \rightarrow \RR_+$ such that 
$\int_{\ZZ} f(z'|z) dz' = 1$ for all $ z \in \ZZ$. We say that the stochastic kernel $P$ 
\emph{has a density representation} if there exists a density kernel $f$ such that
	\begin{equation*}
		P(z, B) = 
					 \int_B f(z'|z) \diff z' 
		\mbox{ for all } z \in \ZZ \mbox{ and } B \in \mathscr{Z}.
	\end{equation*}

\subsection{Set Up}

Let $(Z_n)_{n \geq 0}$ be a time-homogeneous Markov process defined on
probability space $(\Omega, \fF, \PP)$ and taking values in 
measurable space $(\ZZ, \zZ)$.  Let $P$ denote the corresponding
stochastic kernel.  Let $\{ \fF_n\}_{n \geq 0}$ be a filtration contained in
$\fF$ such that $(Z_n)_{n \geq 0}$ is adapted to $\{ \fF_n\}_{n\geq0}$.
Let $\PP_z$ indicate probability conditioned on $Z_0 = z$, while
 $\EE_z$ is expectation conditioned on the same event.
In proofs we take $(\Omega, \fF)$ to be the canonical sequence space, so
that $\Omega = \times_{n = 0}^\infty \ZZ$ and $\fF$ is the product
$\sigma$-algebra generated by $\zZ$.\footnote{
		For the formal construction of $\PP_z$ on $(\Omega, \fF)$ given $P$ and 
		$z \in \ZZ$ see theorem~3.4.1 of \cite{meyn2012markov} or section~8.2 of 
		\cite{stokey1989}.}

A random variable $\tau$ taking values in $\NN_0$ is
called a (finite) \emph{stopping time} with respect to the filtration $\{ \fF_n\}_{n\geq0}$ if
$\PP\{\tau < \infty\} = 1$ and $\{\tau \leq n\} \in \fF_n$ for all $n \geq 0$.  Below, $\tau = n$ has the
interpretation of choosing to act at time $n$.  
Let $\mM$ denote the set of all stopping times on $\Omega$ with
respect to the filtration $\{ \fF_n\}_{n\geq0}$.  

Let $r\colon \ZZ \to \RR$ and $c\colon \ZZ \to \RR$ be measurable functions, referred to
below as the \emph{exit payoff} and \emph{flow continuation payoff},
respectively.  Consider a decision problem where, at each time $t \geq 0$, an agent
observes $Z_t$ and chooses between stopping (e.g., accepting a job, exiting a market, exercising an option) and continuing to the next stage. Stopping
generates final payoff $r(Z_t)$.  Continuing involves a
continuation payoff $c(Z_t)$ and transition to the next
period, where the agent observes $Z_{t+1}$ and the process repeats.
Future payoffs are discounted at rate $\beta \in (0, 1)$.

Let $v^*$ be the value function, which is defined at $z \in \ZZ$ by
\begin{equation}
    \label{eq:defv}
    v^*(z)
    :=
    \sup_{\tau \in \mM} 
    \EE_z
    \left\{
        \sum_{t=0}^{\tau-1} \beta^t c(Z_t) + \beta^{\tau} r(Z_{\tau}) 
    \right\}.
\end{equation}
A stopping time $\tau \in \mM$ is called an \emph{optimal stopping time}
if it attains the supremum in \eqref{eq:defv}.
A \emph{policy} is a map $\sigma$ from $\ZZ$ to $\{0, 1\}$,
with $0$ indicating the decision to continue and $1$ indicating the decision
to stop. A policy $\sigma$ is called an \emph{optimal policy} if 
$\tau^*$ defined by $\tau^* := \inf\{t \geq 0 \,|\, \sigma(Z_t) =  1\}$ is an
optimal stopping time. 

To guarantee existence of the value function and related properties without insisting that payoff functions are bounded, 
we adopt the next assumption:

\begin{assumption}
	\label{a:ubdd_drift_gel}
	There exist a  $\zZ$-measurable function $g\colon \ZZ \rightarrow \RR_+$ 
	and constants $n \in \NN_0$, $m, d \in \RR_+$ such that $\beta m < 1$, and,
	for all $z \in \ZZ$, 
    \begin{equation}
    	\label{eq:bd}
        \max \left\{ \int |r(z')| P^n (z, \diff z'), 
        					\int |c(z')| P^n (z, \diff z') \right\} 
        \leq g(z)
    \end{equation}
    and
    \begin{equation}
        \label{eq:drift}
        \int g(z') P(z, \diff z') \leq m g(z) + d.
    \end{equation}
\end{assumption} 

Note that by definition, condition \eqref{eq:bd} reduces to $|r| \vee |c| \leq g$ 
when $n=0$. The interpretation of assumption \ref{a:ubdd_drift_gel} is that both 
$\EE_z |r(Z_n)|$ and $\EE_z |c(Z_n)|$ are small relative to some function $g$ such 
that $\EE_z g(Z_t)$ does not grow too quickly.  
Slow growth in $\EE_z g(Z_t)$ is imposed by \eqref{eq:drift}, which can be understood as a geometric drift condition (see, e.g., \cite{meyn2012markov}, chapter~15).

\begin{remark}
\label{rm:suff_key_assu}
	To verify assumption~\ref{a:ubdd_drift_gel}, it suffices to obtain a 
	$\zZ$-measurable function 
	$g\colon \ZZ \rightarrow \RR_+$, 
	and constants $n \in \NN_0$, $m, d \in \RR_+$ with $\beta m < 1$, 
	and $a_1, a_2, a_3, a_4 \in \RR_+$ such that 
	$\int |r(z')| P^n (z, \diff z') \leq a_1 g(z) + a_2$, 
	$\int |c(z')| P^n (z, \diff z') \leq a_3 g(z) + a_4$ and \eqref{eq:drift} holds.
	We use this fact in the applications below. 
\end{remark}

\begin{remark}
	One can show that if assumption \ref{a:ubdd_drift_gel} holds for 
	some $n$, it must hold for all $n' \in \NN_0$ such that 
	$n' > n$. Hence, to satisfy assumption \ref{a:ubdd_drift_gel}, it suffices to find 
	a measurable map $g$ and constants $n_1, n_2 \in \NN_0$, 
	$m,d \in \RR_+$ with $\beta m<1$ such that 
	$\int |r(z')| P^{n_1} (z, \diff z') \leq g(z)$, 
	$\int |c(z')| P^{n_2} (z, \diff z') \leq g(z)$ and \eqref{eq:drift} holds.
	One can combine this result with remark \ref{rm:suff_key_assu}	to obtain more 
	general sufficient conditions.	
\end{remark}

\begin{example}
    \label{eg:js_1}
    Consider first an example with bounded rewards. Suppose,
    as in \cite{mccall1970}, that a worker can either accept a current
    wage offer $w_t$ and work permanently at that wage, or reject the offer,
    receive unemployment compensation $c_0>0$ and reconsider next period.
    Let the current wage offer be a function $w_t = w(Z_t)$ of some
    idiosyncratic or aggregate state process $(Z_t)_{t \geq 0}$.
    The exit payoff is $r(z) = u(w(z)) / (1 - \beta)$, where $u$ is a utility
    function and $\beta < 1$ is the discount factor.  The flow continuation
    payoff is $c \equiv c_0$. If $u$ is bounded, then we can set 
    $g(z) \equiv \| r \| \vee c_0$, and assumption~\ref{a:ubdd_drift_gel}
    is satisfied with $n:=0$, $m := 1$ and $d := 0$.
\end{example}

\begin{example}
    \label{eg:jsll}
    Consider now Markov state dynamics in a job search framework
    (see, e.g., \cite{lucas1974equilibrium}, \cite{jovanovic1987work}, 
    \cite{bull1988mismatch}, \cite{gomes2001equilibrium}, 
    \cite{cooper2007search}, \cite{ljungqvist2008two},
    \cite{kambourov2009occupational}, \cite{robin2011dynamics}, 
    \cite{moscarini2013stochastic}, \cite{bagger2014tenure}).
    Consider the same setting as example~\ref{eg:js_1}, with state process 
    \begin{equation}
        \label{eq:state_proc}
        Z_{t+1} = \rho Z_t + b + \epsilon_{t+1}, 
        \quad (\epsilon_t) \iidsim N(0, \sigma^2).
    \end{equation}
    The state space is $\ZZ := \RR$. We consider a typical unbounded problem and 
    provide its proof in appendix B. Let $w_t = \exp(Z_t)$ and the utility of 
    the agent be defined by the CRRA form
    \begin{equation}
    \label{eq:crra_utils}
    	u(w) = \left\{
               	 \begin{array}{ll}
                 	 \frac{w^{1-\delta}}{1- \delta}, \; \mbox{ if }
                 	 															 \delta \geq 0 
                 	 															 \mbox{ and }	
                 	 															 \delta \neq 1	 \\
                 	 \ln w, \; \mbox{ if } \delta = 1	\\
                \end{array}
            \right.
    \end{equation}
    \textit{Case I:} $\delta \geq 0$ and $\delta \neq 1$. If $\rho \in (-1,1)$, then 
    we can select an $n \in \NN_0$ that satisfies $\beta e^{|\rho^n| \xi}<1$, where 
    $\xi := |(1- \delta)b| + (1 - \delta)^2 \sigma^2 / 2$. In this case, assumption 
    \ref{a:ubdd_drift_gel} holds for 
    $g(z) := e^{\rho^n (1 - \delta) z} + 
    			  e^{\rho^n (\delta - 1) z}$ 
    and $m := d:= e^{|\rho^n| \xi}$. 
    Indeed, if $\beta e^{\xi} < 1$, 
    then assumption \ref{a:ubdd_drift_gel} holds (with $n=0$) 
    for all $\rho \in [-1,1]$.
    
    \textit{Case II:} $\delta = 1$. If $\beta |\rho|<1$, then 
    assumption~\ref{a:ubdd_drift_gel} holds with $n:=0$,  
    $g(z) := |z|$, $m := |\rho|$ and $d := \sigma + |b|$.
	Notably, since $|\rho| \geq 1$ is not excluded, wages can be nonstationary 
	provided that they do not grow too fast. 
\end{example}

\begin{remark}
Assumption \ref{a:ubdd_drift_gel} is weaker than the assumptions of existing theory. Consider the local contraction method of 
\cite{rincon2003existence}. The essence is to find a countable increasing sequence 
of compact subsets, denoted by $\{K_j \}$, such that 
$\ZZ = \cup_{j=1}^{\infty} K_j$. 
Let $\Gamma: \ZZ \rightarrow 2^{\ZZ}$ be the technological correspondence of the 
state process $( Z_t)_{t \geq 0}$ giving the set of feasible actions. To construct local contractions, one need  
$\Gamma(K_j) \subset K_j$ or $\Gamma(K_j) \subset K_{j+1}$ with probability one 
for all $j \in \NN$ (see, e.g., theorems 3--4 of \cite{rincon2003existence}, or
assumptions D1--D2 of \cite{matkowski2011discounted}). This assumption is 
often violated when $(Z_t)_{t \geq 0}$ has unbounded supports. In example 
\ref{eg:jsll}, since the AR$(1)$ state process \eqref{eq:state_proc} travels intertemporally through $\RR$ with positive probability, the local contraction method breaks down.
\end{remark}

\begin{remark}
The use of $n$-step transitions in assumption \ref{a:ubdd_drift_gel}-\eqref{eq:bd} has certain advantages. For example, if $(Z_t)_{t \geq 0}$ is mean-reverting, as time iterates forward, the initial effect tends to die out, making the conditional expectations $\EE_z |r(Z_n)|$ and $\EE_z |c(Z_n)|$ flatter than the original payoffs. As a result, finding an appropriate $g$-function with geometric drift property is much easier. Typically, in \textit{Case I} of example \ref{eg:jsll}, if $\rho \in (-1,1)$, without using future transitions (i.e., $n=0$ is imposed),\footnote{
	Indeed, our assumption in this case reduces to the standard weighted 
	supnorm assumption. See, e.g., section 4 of \cite{boud1990recursive}, or 
	assumptions 1-4 of 	\cite{duran2003discounting}.}  
one need further assumptions such as $\beta e^{\xi} < 1$ (see appendix B),
which puts nontrivial restrictions on the key parameters $\beta$ and $\delta$. 
Using $n$-step transitions, however, such restrictions are completely removed.
\end{remark}

\begin{example}
\label{eg:js_adap}
	Consider now agent's learning in a job search framework (see, 
	e.g., \cite{mccall1970}, \cite{chalkley1984adaptive}, \cite{burdett1988declining},
	\cite{pries2005hiring}, \cite{nagypal2007learning}, 
	\cite{ljungqvist2012recursive}). 
	We follow \cite{mccall1970} (section IV) and explore how the reservation 
	wage changes in response to the agent's expectation of the mean and variance of 
	the (unknown) wage offer distribution. Each period, the agent observes an 
	offer $w_t$ and decides whether to accept it or remain unemployed. The wage 
	process $(w_t)_{t \geq 0}$ follows
	\begin{equation}
	\label{eq:lm_w}
		\ln w_t =\xi + \epsilon_{t}, 
				\quad 
		(\epsilon_{t})_{t \geq 0} \iidsim N(0,\gamma_{\epsilon}),
	\end{equation}
	where $\xi$ is the mean of the wage process, which is not observed by the 
	worker, who has prior belief 
	$\xi \sim N(\mu,\gamma)$.\footnote{
		In general, $\xi$ can be a stochastic process, e.g., 
		$\xi_{t+1} = \rho \xi_{t} + \epsilon_{t+1}^{\xi}$, 
		$(\epsilon_t^{\xi} ) \iidsim N(0, \gamma_\xi)$. 
		We consider such an extension in a firm entry framework in
		section \ref{ss:fe}. } 
	The worker's current estimate of the next period wage distribution is 
	$f(w'|\mu,\gamma)=LN(\mu,\gamma+\gamma_{\epsilon})$. 
	After observing $w'$, the belief is updated, with posterior  
	$\xi|w' \sim N(\mu',\gamma')$, where
	\begin{equation}
	\label{eq:pos_js_adap}
		\gamma' 
			= \left( 
						1 / \gamma + 1 / \gamma_{\epsilon} 
				\right)^{-1} 
			\quad \mbox{and} \quad
		\mu' 
			= \gamma' \left( 
									\mu / \gamma + 
									\ln w' / \gamma_{\epsilon} 
							   \right).
	\end{equation}
	Let the utility of the worker be defined by \eqref{eq:crra_utils}. If he 
	accepts the offer, the search process terminates and a utility $u(w)$ is 
	obtained in each future period. Otherwise, the worker gets compensation
	$\tilde{c}_0>0$, updates his belief next period, and reconsiders. The state
	vector is $z = (w, \mu, \gamma) 
					   \in \RR_{++} \times \RR \times \RR_{++}
					   =: \ZZ$. 
	For any integrable function $h$, the stochastic kernel $P$ satisfies
	\begin{equation}
	\label{eq:Ph}
		\int h(z') P(z, \diff z')
			= \int 
					  h (w',\mu', \gamma') f(w'|\mu, \gamma) 
			    \diff w',
	\end{equation}
	where $\mu'$ and $\gamma'$ are defined by \eqref{eq:pos_js_adap}.
	The exit payoff is $r(w) = u(w) / (1-\beta)$, and the flow continuation payoff is
	$c \equiv c_0 := u(\tilde{c}_0)$. If $\delta \geq 0$ and $\delta \neq 1$, 
	 assumption \ref{a:ubdd_drift_gel} holds by letting $n:=1$, 
	$g(\mu, \gamma) 
		:= e^{(1 - \delta) \mu + (1 - \delta)^2 \gamma /2}$,
	$m := 1$ and $d := 0$.
	If $\delta = 1$, assumption \ref{a:ubdd_drift_gel} holds with
	$n :=1$, 
	$g(\mu, \gamma) 
		:= e^{-\mu + \gamma/2} + e^{\mu + \gamma / 2}$, 
	$m := 1$ and $d:=0$. See appendix B for a detailed proof.
\end{example}

\begin{remark}
	Since in example \ref{eg:js_adap}, the wage process $(w_t)_{t \geq 0}$ is 
	independent and has unbounded support $\RR_+$, the local contraction method 
	cannot be applied.
\end{remark}

\begin{remark}
	From \eqref{eq:Ph} we know that the conditional expectation of the reward
	functions in example \ref{eg:js_adap} is defined on a space of lower 
	dimension than the state	space. Although there are 3 states, $\EE_z |r(Z_1)|$ is a 
	function of only 2 arguments: $\mu$ and  $\gamma$. Hence, taking conditional 
	expectation makes it easier to find an appropriate $g$ function. 
	Indeed, if the standard weighted supnorm method were applied, one need to find
	a $\tilde{g}(w, \mu, \gamma)$ with geometric drift property that dominates 
	$|r|$ (see, e.g., section 4 of \cite{boud1990recursive}, or, assumptions 1--4 of 
	\cite{duran2003discounting}), which is more challenging due to the higher 
	state dimension. This type of problem is pervasive in economics. Sections 
	\ref{s:opt_pol}--\ref{s:application} provide a systematic study, along with a list of 
	applications.
\end{remark}

\subsection{The Continuation Value Operator}

The \textit{continuation value function} associated with the sequential decision problem \eqref{eq:defv} is defined at $z \in \ZZ$ by 
\begin{equation}
\label{eq:cvf}
	\psi^*(z) := c(z) + \beta \int v^*(z') P(z,dz').
\end{equation}

Under assumption \ref{a:ubdd_drift_gel}, the value function is a solution to the Bellman equation, i.e., $v^* = r \vee \psi^*$. To see this, by theorem 1.11 of \cite{peskir2006}, it suffices to show that
\begin{equation*}
	\EE_z \left( 
					\sup_{k\geq 0}
						\left| 
							\sum_{t=0}^{k-1} \beta^t c(Z_t) + \beta^k r(Z_k) 
						\right|
			  \right)
	< \infty
\end{equation*}
for all $z \in \ZZ$. This obviously holds since
\begin{equation*}
	\sup_{k\geq 0}
		\left| 
			\sum_{t=0}^{k-1} \beta^t c(Z_t) + \beta^k r(Z_k) 
		\right|
	\leq 	
		\sum_{t \geq 0} \beta^t [|r(Z_t)| + |c(Z_t)|]	
\end{equation*}
with probability one, and by lemma \ref{lm:bd_vcv} (see \eqref{eq:bdsum} 
in appendix A), the right hand side is $\PP_z$-integrable for all $z \in \ZZ$. 

To obtain some fundamental optimality results concerning the continuation value 
function, define an operator $Q$ by
\begin{equation}
    \label{eq:defq}
    Q \psi (z) = c(z) + \beta \int \max\{ r(z'), \psi(z') \} P(z, \diff z').
\end{equation}
We call $Q$ the \textit{Jovanovic operator} or the \textit{continuation value operator}. As shown below, fixed points of $Q$ are continuation value functions. From them 
we can derive value functions, optimal policies and so on. To begin with, recall
$n$, $m$ and $d$ defined in assumption \ref{a:ubdd_drift_gel}. Let $m', d' > 0$ such that $m+2m' > 1$, $\beta(m+ 2m')<1$ and $d' \geq d / (m + 2m' -1)$.
Let the weight function $\ell \colon \ZZ \to \RR$ be
\begin{equation}
\label{eq:ell_func}
    \ell(z) := 
        m' \left( \sum_{t=1}^{n-1} 
        					\EE_z |r(Z_t)| + 
        			  \sum_{t=0}^{n-1}
        			  		\EE_z |c(Z_t)| \right)
        + g(z) + d'.
\end{equation}
We have the following optimality result.
\begin{theorem}
    \label{t:bk}
    Under assumption~\ref{a:ubdd_drift_gel},
    the following statements are true:
    \begin{enumerate}
        \item[1.] $Q$ is a contraction mapping on 
       		$\left(b_\ell \ZZ, \| \cdot \|_\ell \right)$ of modulus $\beta (m + 2m')$.
        \item[2.] The unique fixed point of $Q$ in $b_\ell \ZZ$ is $\psi^*$.
        \item[3.] The policy defined by 
        	$\sigma^*(z) = \1\{r(z) \geq \psi^*(z) \}$ is an optimal policy. 
    \end{enumerate}
\end{theorem}

\begin{remark}
\label{rm:bdd_n01}
	If both $r$ and $c$ are bounded, then $\ell$ can be chosen as a constant, and 
	$Q$ is a contraction mapping of modulus $\beta$ on 
	$\left( b \ZZ, \| \cdot \| \right)$. If assumption \ref{a:ubdd_drift_gel}
	is satisfied for $n=0$, then the weight function $\ell(z) = g(z) + d'$. 
	If assumption \ref{a:ubdd_drift_gel} holds for	$n=1$, then 
	$\ell(z) = m' |c(z)| + g(z) + d'$. 
\end{remark}

\begin{example}
\label{eg:jsll_continue1}
    Recall the job search problem of example \ref{eg:jsll}.
    Let $g, n, m$ and $d$ be defined as in that example. Define $\ell$ 
    as in \eqref{eq:ell_func}. The Jovanovic operator is
    \begin{equation*}
        Q \psi(z) 
        	= c_0 + \beta \int 
        								\max \left\{ 
        													\frac{u(w(z'))}{1-\beta}, 
        											      	\psi (z') 
        										  \right\} 
          				 				 f(z'|z) 
          				 		  \diff z'.
    \end{equation*}
    Since assumption~\ref{a:ubdd_drift_gel} holds, theorem~\ref{t:bk} implies
    that $Q$ has a unique fixed point in $b_\ell \ZZ$ that coincides with 
    the continuation value function, which, in this case, can be interpreted as the 
    expected value of unemployment.
\end{example}

\begin{example}
\label{eg:js_adap_continue1}
	Recall the adaptive search model of example \ref{eg:js_adap}. Let $\ell$ be 
	defined by \eqref{eq:ell_func}. The Jovanovic operator is
	\begin{equation}
	\label{eq:cvo_jsadap}
		Q \psi (\mu, \gamma)
			= c_0 + 
				\beta \int 
								\max \left\{ 
												\frac{u(w')}{1 - \beta},
												\psi (\mu', \gamma')
										 \right\}
								f(w' | \mu, \gamma) \diff w',
	\end{equation}
	where $\mu'$ and $\gamma'$ are defined by \eqref{eq:pos_js_adap}. As shown 
	in example \ref{eg:js_adap}, assumption \ref{a:ubdd_drift_gel} holds. By 
	theorem \ref{t:bk}, $Q$ is a contraction mapping on 
	$( b_{\ell} \ZZ, \| \cdot \|_{\ell} )$
	with unique fixed point $\psi^*$, the expected value of unemployment.
\end{example}

\begin{example}
    \label{eg:perpetual_option}
    Consider an infinite-horizon American option (see, e.g.,
    \cite{shiryaev1999essentials} or \cite{duffie2010dynamic}).
    Let the state process be as in \eqref{eq:state_proc} so that the state space
    $\ZZ := \RR$. Let $p_t = p(Z_t) = \exp(Z_t)$ be the current price of the 
    underlying asset, and $\gamma >0$ be the riskless rate of return (i.e., 
    $\beta = e^{-\gamma}$).
    The exit payoff for a call option with a strike price $K$ is $r(z) = (p(z) - K)^+$, 
    while the flow continuation payoff is $c \equiv 0$. The Jovanovic 
    operator for the option satisfies
	\begin{equation*}
		Q \psi(z) 
			= e^{-\gamma} 
				\int 
					  \max \{ (p(z')-K)^+, \psi(z') \} 
					  f(z'|z) 
				\diff z'.
	\end{equation*}
    If $\rho \in (-1,1)$, we can let $\xi := |b| + \sigma^2 / 2$ and  
    $n \in \NN_0$ such that $e^{-\gamma + |\rho^n| \xi} < 1$, then assumption 
	\ref{a:ubdd_drift_gel} holds with 
	$g(z) := e^{\rho^n z} + e^{-\rho^n z}$ and $m := d:= e^{|\rho^n| \xi}$.
	Moreover, if $e^{-\gamma + \xi} <1$, then assumption  
	\ref{a:ubdd_drift_gel} holds (with $n=0$) for all $\rho \in [-1,1]$. 
	For $\ell$ as defined by \eqref{eq:ell_func},
	theorem \ref{t:bk} implies that $Q$ admits a unique fixed point in 
	$b_{\ell} \ZZ$ that coincides with $\psi^*$, the expected value of retaining the 
	option and exercising at a later stage.
	The proof is similar to that of example \ref{eg:jsll} and thus omitted.
\end{example}

\begin{example}
\label{eg:r&d}
	Firm's R$\&$D decisions are often modeled as a sequential search 
	process for better technologies (see, e.g., \cite{jovanovic1989growth}, 
	\cite{bental1996accumulation},  \cite{perla2014equilibrium}). 
	In each period, an idea with value $Z_t \in \ZZ := \RR_+$ is observed, and 
	the firm decides whether to put this idea into productive use, or develop 
	it further by investing in R$\&$D. The former choice gives a payoff 
	$r(Z_t) = Z_t$. The latter incurs a fixed cost $c_0 >0$ so as to create a new 
	technology. Let the R$\&$D process be governed by the exponential law of 
	motion (with rate $\theta>0$),
	\begin{equation}
	\label{eq:law_expo}
		F(z'|z) 
			:= \PP (Z_{t+1} \leq z' | Z_t = z) 
			= 1 - e^{ - \theta (z' - z)}
		\quad
		(z' \geq z),
	\end{equation}
	While the payoff functions are unbounded, assumption \ref{a:ubdd_drift_gel} is 
	satisfied with $n:=0$, $g(z) := z$, $m:=1$ and $d := 1/ \theta$.
	The Jovanovic operator satisfies
	\begin{equation*}
		Q \psi(z) 
		 = - c_0 + 
		 	\beta \int \max \{ z', \; \psi(z') \} \diff F(z'|z).
	\end{equation*}
	With $\ell$ as in \eqref{eq:ell_func}, $Q$ is a contraction mapping on 
	$b_{\ell} \ZZ$ with unique fixed point $\psi^*$, the expected value of investing 
	in R$\&$D.
	The proof is straightforward and omitted.
\end{example}

\begin{example}
\label{eg:firm_exit}
    Consider a firm exit problem (see, e.g., \cite{hopenhayn1992entry}, 
    \cite{ericson1995markov}, \cite{albuquerque2004optimal}, 
    \cite{asplund2006firm}, \cite{poschke2010regulation},
    \cite{dinlersoz2012information}, \cite{cocsar2016firm}). 
    Each period, a productivity shock $a_t$ is observed by an incumbent firm, where
    $a_t = a(Z_t) = e^{Z_t}$, and the state process $Z_t \in \ZZ := \RR$ is defined by 
    \eqref{eq:state_proc}. The firm then decides whether to exit the market 
    next period or not (before observing $a'$). A fixed cost $c_f >0$ is paid each 
    period by the incumbent firm. The firm's output is $q (a, l) = a l^{\alpha}$, where
    $\alpha \in (0,1)$ and $l$ is labor demand. Given output and input prices 
    $p$ and $w$, the payoff functions are
    $r(z)=c(z) = G a(z)^{\frac{1}{1 - \alpha}} - c_f$, where 
    $G = \left(\alpha p / w \right)^{\frac{1}{1-\alpha}} 
    		(1-\alpha) w / \alpha  $. 
    The Jovanovic operator satisfies
    \begin{equation*}
        Q \psi(z) 
            = \left( 
            			G a(z) ^{\frac{1}{1 - \alpha}} - c_f 
            	\right) 
            	+ \beta 
            		\int \max \left\{
               								 G a(z') ^{\frac{1}{1 - \alpha}} - c_f, \psi(z') 
          						    \right\} 
          					f(z'|z) 
          			\diff z'.
    \end{equation*}
    For $\rho \in [0,1)$, choose $n \in \NN_0$ such that 
    $\beta e^{ \rho^n \xi}<1$, where 
    $\xi := \frac{b}{1-\alpha} +
    			\frac{\sigma^2}{2 (1-\alpha)^2}$.
    Then assumption \ref{a:ubdd_drift_gel} holds with   
    $g(z) := e^{\rho^n z / (1-\alpha)}$
    and $m := d := e^{\rho^n \xi}$. Moreover, if $\beta e^{\xi}<1$, then  
    assumption \ref{a:ubdd_drift_gel} holds (with $n=0$) for all $\rho \in [0,1]$. 
    The case $\rho \in [-1, 0]$ is similar.
    By theorem \ref{t:bk}, $Q$ admits a unique 
    fixed point in $b_{\ell} \ZZ$ that corresponds to $\psi^*$, the expected value of 
    staying in the industry next period.\footnote{
     The proof is similar to that of example \ref{eg:jsll}. Here we are considering the 
     case $\rho \in [-1,0]$ and $\rho \in [0,1]$ separately. Alternatively, we can treat 
     $\rho \in [-1,1]$ directly as in examples \ref{eg:jsll} and 
     \ref{eg:perpetual_option}. As shown in the proof of example \ref{eg:jsll}, the 
     former provides a simpler $g$ function when $\rho \geq 0$. }
\end{example}

\begin{example}
\label{eg:firm_exit_j}
	Consider agent's learning in a firm exit framework (see, e.g.,
	\cite{jovanovic1982selection}, \cite{pakes1998empirical}, 
	\cite{mitchell2000scope}, \cite{timoshenko2015product}). 
	Let $q$ be firm's output, $C(q)$ a cost function, and $C(q) x$ be the total cost, 
	where the state process $(x_t)_{t \geq 0}$ satisfies
	$\ln x_t 
			= \xi + \epsilon_t, 
			(\epsilon_t) \iidsim N(0, \gamma_{\epsilon})$
	with $\xi$ denoting the firm type. 
	Beginning each period, the firm observes $x$ and decides whether to exit 
	the industry or not. The prior belief is $\xi \sim N(\mu,\gamma)$, so the 
	posterior after observing $x'$ is 	$\xi |x' \sim N(\mu', \gamma')$, where 
	$\gamma' 
			= \left( 
						1 / \gamma + 
						1 / \gamma_{\epsilon}
				\right)^{-1}$ 
	and 
	$\mu' = \gamma' \left( 
		   								\mu / \gamma + 
		   								(\ln x') / \gamma_{\epsilon}
		   					 	   \right)$.
	Let $\pi(p,x) = \underset{q}{\max} [pq - C(q)x]$ be the maximal profit, and
	$r(p,x)$ be the profit of other industries, where $p$ is price. 
	Consider, for example, $C(q) := q^2$, and 
	$(p_t)_{t \geq 0}$ satisfies 
	$\ln p_{t+1} = \rho \ln p_t + b + \epsilon^p_{t+1}$, 
	$(\epsilon^p_t)_{t \geq 0} \iidsim N(0, \gamma_p)$. Let
	$z := (p,x, \mu, \gamma) 
		  \in \RR_+^2 \times \RR \times \RR_+
		  =: \ZZ$. 
	Then the Jovanovic operator satisfies
	\begin{equation*}
		Q \psi(z) 
			= \pi(p, x) + 
			\beta \int 
						  \max \{r(p',x'), \psi(z') \} 
						  l(p', x' |p, \mu,\gamma)  
					  \diff (p', x'),
	\end{equation*}
	where $ l(p', x' |p, \mu,\gamma) 
					:= h(p'|p) f(x' | \mu,\gamma)$ with
	$h(p'|p) := LN(\rho \ln p + b, \gamma_p)$ and
	$f(x'| \mu, \gamma) 
			:= LN(\mu, \gamma+ \gamma_{\epsilon})$.   
	If $\rho \in (-1,1)$ and $|r(p,x)| \leq h_1 p^2 / x + h_2$ for some
	constants $h_1, h_2 \in \RR_+$, let $\xi := 2(|b| + \gamma_p)$ and 
	choose $n \in \NN_0$ such that $\beta e^{|\rho^n| \xi} < 1$. Define $\delta$
	such that
	$\delta \geq 
				 e^{|\rho^n| \xi} / \left( 
				 									e^{|\rho^n| \xi} - 1 
				 							  \right)$.\footnote{
	Implicitly, we are considering $\rho \neq 0$. The case $\rho = 0$ is trivial. } 
	Then assumption \ref{a:ubdd_drift_gel} holds by letting 
	$g(p,\mu,\gamma) 
		:= \left( p^{2\rho^{n}} + p^{-2\rho^{n}} + \delta \right)
			 e^{ -\mu + \gamma / 2}$,
	$m := e^{|\rho^n| \xi}$ and $d:=0$. 
	Hence, $Q$ admits a unique fixed point in $b_{\ell} \ZZ$ that equals $\psi^*$, 
	the value of staying in the industry.\footnote{
	In fact, the same result holds for more general settings, e.g., 
	$|r(p,x)| 
		\leq h_1 p^2 / x + h_2 p^2 + h_3 x^{-1} + h_4 x + h_5$ 
	for some $h_1, ..., h_5 \in \RR_+$.} 
\end{example}

\section{Properties of Continuation Values}
\label{s:properties_cv}

In this section we explore some further properties of the continuation value 
function. As one of the most significant results, $\psi^*$ is shown to be continuously differentiable under 
certain assumptions.

\subsection{Continuity}

We first develop a theory for the continuity of the fixed point. 

\begin{assumption}
\label{a:feller}
	The stochastic kernel $P$ satisfies the Feller property, i.e., $P$ maps bounded 
	continuous functions into bounded continuous functions.
\end{assumption}

\begin{assumption}
\label{a:payoff_cont}
	The functions $c$, $r$ and $z \mapsto \int |r(z')| P(z, \diff z')$ are continuous.
\end{assumption}

\begin{assumption}
\label{a:l_cont}
	The functions $\ell$ and $z \mapsto \int \ell(z') P(z, \diff z')$ are continuous.
\end{assumption}

\begin{proposition}
\label{pr:cont}
	Under assumptions \ref{a:ubdd_drift_gel} and \ref{a:feller}--\ref{a:l_cont}, 
	$\psi^*$ and $v^*$ are continuous.
\end{proposition}

The next result treats the special case when $P$ admits a density representation. The proof is similar to that of proposition \ref{pr:cont}, except that we use lemma \ref{lm:cont} instead of the generalized Fatou's lemma of \cite{feinberg2014fatou}  to establish continuity in \eqref{eq:fatou_eq}. In this way, notably, the continuity of $r$ is not necessary for the continuity of $\psi^*$. The proof is 
omitted.

\begin{corollary}
\label{cr:cont_dst}
	Suppose assumptions \ref{a:ubdd_drift_gel} and \ref{a:l_cont}
	hold, $P$ admits a density representation $f(z'|z)$ that is continuous in $z$, 
	and that $z \mapsto \int |r(z')| f(z'|z) \diff z'$ and $c$ are 
	continuous, then $\psi^*$ is continuous. 
	If in addition $r$ is continuous, then $v^*$ is continuous.
\end{corollary}

\begin{remark}
\label{rm:bdd_cont}
	By proposition \ref{pr:cont}, if the payoffs $r$ and $c$ are bounded, 
	assumption \ref{a:feller} and the continuity of $r$ and $c$ are sufficient for the 
	continuity of $\psi^*$ and $v^*$.
	If in addition $P$ has a density representation $f$, by corollary \ref{cr:cont_dst},  
	the continuity of the flow payoff $c$ and $z \mapsto f(z'|z)$ (for all $z' \in \ZZ$) 
	is sufficient for $\psi^*$ to be continuous.\footnote{
		Notice that in these cases, $\ell$ can be chosen as a constant, so assumption 
		\ref{a:l_cont} holds naturally.} 
	Based on these, the continuity of $\psi^*$ and 
	$v^*$ of example \ref{eg:js_1} can be established. 
\end{remark}

\begin{remark}
	If assumption \ref{a:ubdd_drift_gel} satisfies for $n=0$ and assumption 
	\ref{a:feller} holds, then assumptions \ref{a:payoff_cont}--\ref{a:l_cont} are 
	equivalent to: $r$, $c$, $g$ and $z \mapsto \EE_z g(Z_1)$ are 
	continuous.\footnote{
		When $n=0$, $\ell(z) = g(z) + d'$, so $|r| \leq G(g + d')$ for some constant
		$G$. Since $r,g$ and $z \mapsto \EE_z g(Z_1)$ are continuous, 
		\cite{feinberg2014fatou} (theorem 1.1) implies that 
		$z \mapsto \EE_z |r(Z_1)|$ is continuous. The next claim in this
		remark can be proved similarly. } 
	If assumption \ref{a:ubdd_drift_gel} holds for $n = 1$ and assumptions 
	\ref{a:feller}--\ref{a:payoff_cont} are satisfied, then assumption \ref{a:l_cont} 
	holds if and only if $g$ and $z \mapsto \EE_z |c(Z_1)|, \EE_z g(Z_1)$ are 
	continuous. 
\end{remark}

\begin{example}
\label{eg:jsll_continue2}
Recall the job search model of examples \ref{eg:jsll} and \ref{eg:jsll_continue1}. 
By corollary \ref{cr:cont_dst}, $\psi^*$ and $v^*$ are continuous.
Here is the proof.
Assumption \ref{a:ubdd_drift_gel} holds, as was shown. $P$ has a density representation $f(z'|z) = N(\rho z + b, \sigma^2)$ that is continuous in $z$. 
Moreover, $r,c$ and $g$ are continuous. It remains to verify assumption \ref{a:l_cont}.

\textit{Case I:} $\delta \geq 0$ and $\delta \neq 1$.  
The proof of example \ref{eg:jsll} shows that 
$z \mapsto \EE_z |r(Z_t)|$ is continuous for all $t \in \NN$, and that
$z \mapsto \EE_z g(Z_1)$ is continuous (recall \eqref{eq:e_ntimes}--\eqref{eq:e_g} 
in appendix B). By the definition of $\ell$ in \eqref{eq:ell_func}, assumption 
\ref{a:l_cont} holds.

\textit{Case II:} $\delta = 1$. 
Recall that assumption \ref{a:ubdd_drift_gel} holds for $n=0$ and $g(z) = |z|$. 
Since $z \mapsto \int |z'| f(z'|z) \diff z'$ is continuous by properties of the normal 
distribution, $z \mapsto \EE_z g(Z_1), \EE_z |r(Z_1)|$ are continuous.\footnote{
		Indeed, 
		$\int |z'| f(z'|z) \diff z' 
		= \sqrt{2 \sigma^2 / \pi} \;
			e^{ -(\rho z + b)^2 / 2 \sigma^2 }
			+ (\rho z + b) \left[ 
											1 - 2 \Phi \left( -(\rho z + b) / \sigma \right) 
									\right]$, 
		where $\Phi$ is the cdf of the standard normal distribution. The continuity
		can also be proved by lemma \ref{lm:cont}. } 
Hence, assumption \ref{a:l_cont} holds.
\end{example}

\begin{example}
\label{eg:js_adap_continue2}
Recall the adaptive search model of examples \ref{eg:js_adap} and 
\ref{eg:js_adap_continue1}. Assumption \ref{a:ubdd_drift_gel} holds for $n=1$, 
as already shown. Assumption \ref{a:feller} follows from \eqref{eq:Ph} and lemma \ref{lm:cont}. Moreover, $r,c$ and $g$ are continuous. In the proof of example 
\ref{eg:js_adap}, we have shown that 
$(\mu, \gamma) 
	\mapsto \EE_{\mu, \gamma} |r(w')|,
				   \EE_{\mu, \gamma} g(\mu', \gamma')$ 
are continuous (see \eqref{eq:js_adap_er}--\eqref{eq:js_adap_eu} in appendix B),
where $\EE_{\mu, \gamma} g(\mu', \gamma')$ is defined by \eqref{eq:js_adap_eg} and
$\EE_{\mu, \gamma} |r(w')|
	:= \int |r(w')| f(w'|\mu, \gamma) \diff w'$.
Since $\ell = m'|c| + g + d'$ when $n = 1$, assumptions 
\ref{a:payoff_cont}--\ref{a:l_cont} hold. By proposition \ref{pr:cont}, $\psi^*$ and
$v^*$ are continuous.
\end{example}

\begin{example}
\label{eg:perp_option_continue1}
Recall the option pricing model of example \ref{eg:perpetual_option}.
By corollary \ref{cr:cont_dst}, we can show that
$\psi^*$ and $v^*$ are continuous. The proof is similar to example \ref{eg:jsll_continue2}, except that we use $|r(z)| \leq e^z + K$, the continuity of 
$z \mapsto \int (e^{z'} + K) f(z'|z) \diff z'$, and lemma \ref{lm:cont} to show that $z \mapsto \EE_z |r(Z_1)|$ is continuous. The continuity of 
$z \mapsto \EE_z |r(Z_t)|$ (for all $t \in \NN$) then follows from induction. 
\end{example}

\begin{example}
\label{eg:r&d_continue1}
Recall the R$\&$D decision problem of example \ref{eg:r&d}. Assumption \ref{a:ubdd_drift_gel} holds for $n=0$. For all bounded continuous function $h:\ZZ \rightarrow \RR$, lemma \ref{lm:cont} shows that 
$z \mapsto \int h(z') P(z, \diff z')$ is continuous, so assumption \ref{a:feller} holds.
Moreover, $r$, $c$ and $g$ are continuous, and
\begin{equation*}
	\int |z'| P(z, \diff z')
	= \int_{[z, \infty)} z' \theta e^{- \theta (z' - z)} \diff z'
	= z + 1 / \theta
\end{equation*}
implies that $z \mapsto \EE_z |r(Z_1)|, \EE_z|g(Z_1)|$ are continuous. Since 
$\ell(z) = g(z) + d'$ when $n=0$, assumptions \ref{a:payoff_cont}--\ref{a:l_cont}
hold. By proposition \ref{pr:cont}, $\psi^*$ and $v^*$ are continuous.
\end{example}

\begin{example}
\label{eg:firm_exit_continue1}
Recall the firm exit model of example \ref{eg:firm_exit}. Through similar analysis as in examples \ref{eg:jsll_continue2} and \ref{eg:perp_option_continue1}, we can show that $\psi^*$ and $v^*$ are continuous.
\end{example}

\begin{example}
\label{eg:fej_continue1}
Recall the firm exit model of example \ref{eg:firm_exit_j}. Assumption \ref{a:ubdd_drift_gel} holds, as was shown. The flow continuation payoff
$\pi(p,x) = p^2 / (4x)$ since $C(q) = q^2$. Recall that $z=(p,x,\mu, \gamma)$, and for all integrable $h$, we have
\begin{equation*}
	\int h(z') P(z, \diff z') 
		= \int 
				h(p',x',\mu',\gamma') 
				l(p',x'|p,\mu,\gamma) 
			\diff (p', x').
\end{equation*}
Since by definition $\gamma'$ is continuous in $\gamma$ and $\mu'$ is 
continuous in $\mu, \gamma$ and $x'$, assumption \ref{a:feller} holds by lemma 
\ref{lm:cont}. Further, induction shows that for some constant $a_t$ and all 
$t \in \NN$,
\begin{equation*}
	\int |\pi (p', x')| P^t (z, \diff z')
		= a_t p^{2 \rho^t} e^{-\mu + \gamma / 2},
\end{equation*} 
which is continuous in 
$(p, \mu, \gamma)$. 
If $r(p,x)$ is continuous, then, by lemma \ref{lm:cont} and induction (similarly as in 
example \ref{eg:perp_option_continue1}), 
$(p, \mu, \gamma) 
		\mapsto \int |r(p',x')| P^t (z, \diff z')$
is continuous for all $t \in \NN$. Moreover, $g$ is continuous and
\begin{equation*}
	\int g(p',\mu', \gamma') P(z, \diff z')
	= \left( p^{2\rho^{n+1}}  
					   e^{2 \rho^n b + 2 \rho^{2n} \gamma_p}
				+ p^{-2 \rho^{n+1}} 
					e^{-2 \rho^n b + 2 \rho^{2n} \gamma_p}	    
			+ \delta
			  \right) 
			  e^{ -\mu + \gamma / 2 }	,
\end{equation*}
which is continuous in $(p, \mu,\gamma)$. Hence, assumptions 
\ref{a:payoff_cont}--\ref{a:l_cont} hold.
Proposition \ref{pr:cont} then implies that $\psi^*$ and $v^*$ are continuous.
\end{example}

\subsection{Monotonicity} We now study monotonicity under the following assumptions.

\begin{assumption}
\label{a:c_incre}
	The flow continuation payoff $c$ is increasing (resp. decreasing).
\end{assumption}

\begin{assumption}
\label{a:mono_map}
	The function $z \mapsto \int \max \{ r(z'), \psi(z')\} P(z, \diff z')$ is increasing 
	(resp. decreasing) for all increasing (resp. decreasing) function $\psi \in b_{\ell} 
	\ZZ$.
\end{assumption}

\begin{assumption}
	\label{a:r_incre}
	The exit payoff $r$ is increasing (resp. decreasing). 
\end{assumption}

\begin{remark}
If assumption \ref{a:r_incre} holds and $P$ is stochastically increasing in the sense that $P(z, \cdot)$ (first order) stochastically dominates $P(\tilde{z}, \cdot)$ for all $\tilde{z} \leq z$, then assumption \ref{a:mono_map} holds.
\end{remark}

\begin{proposition}
\label{pr:mono}
	Under assumptions \ref{a:ubdd_drift_gel} and \ref{a:c_incre}--\ref{a:mono_map}, 
	$\psi^*$ is increasing (resp. decreasing). If in addition assumption \ref{a:r_incre} 
	holds, then $v^*$ is increasing (resp. decreasing).
\end{proposition}

\begin{proof}[Proof of proposition~ \ref{pr:mono}]
	Let $b_{\ell} i \ZZ$ (resp. $b_{\ell} d \ZZ$) be the set of increasing (resp. 
	decreasing) functions in $b_{\ell} \ZZ$. Then $b_{\ell} i \ZZ$ (resp. $b_{\ell} d 
	\ZZ$) is a closed subset of $b_{\ell} \ZZ$.\footnote{
		Let $(\phi_n) \subset b_{\ell} i \ZZ$ such that 
		$\rho_{\ell} (\phi_n, \phi) \rightarrow 0$, then $\phi_n \rightarrow \phi$ 
		pointwise.
		Since $(b_{\ell} \ZZ, \rho_{\ell})$ is 
		complete, $\phi \in b_{\ell} \ZZ$. For all $x, y \in \ZZ$ with $x < y$,
		$\phi(x) - \phi (y) 
			= [\phi(x) - \phi_n (x)] + 
			    [\phi_n (x) - \phi_n (y)] +
				[\phi_n (y) - \phi (y)]$. 
		The second term on the right side is nonpositive, $\forall n$. Taking limit 
		supremum on both sides yields $\phi (x) \leq \phi(y)$. Hence, 
		$\phi \in b_{\ell} i \ZZ$ and $b_{\ell} i \ZZ$ is a closed subset. The case 
		$b_{\ell} d \ZZ$ is similar.} 
	To show that $\psi^*$ is increasing (resp. decreasing), it suffices to verify that 
	$Q(b_{\ell} i \ZZ) \subset b_{\ell} i \ZZ$ 
	(resp. $Q(b_{\ell} d \ZZ) \subset b_{\ell} d \ZZ$).\footnote{ 
	See, e.g., \cite{stokey1989}, corollary 1 of theorem 3.2.} 
	The assumptions of the proposition guarantee 
	that this is the case. Since, in addition, $r$ is increasing (resp. decreasing) by 
	assumption and $v^*=r \vee \psi^*$, $v^*$ is increasing (resp. decreasing).
\end{proof}

\begin{example}
\label{eg:jsll_continue3}
Recall the job search model of examples \ref{eg:jsll}, \ref{eg:jsll_continue1} and 
\ref{eg:jsll_continue2}. Assumption \ref{a:ubdd_drift_gel}, \ref{a:c_incre} and 
\ref{a:r_incre} hold. If $\rho \geq 0$, the stochastic kernel $P$ is stochastically 
increasing since the density kernel is $f(z'|z)=N(\rho z + b, \sigma^2)$, so 
assumption \ref{a:mono_map} holds. By proposition \ref{pr:mono}, $\psi^*$ and 
$v^*$ are increasing.
\end{example}

\begin{remark}
\label{rm:po_fe_rd}
Similarly, we can show that for the option pricing model of example \ref{eg:perpetual_option} and the firm exit model of example \ref{eg:firm_exit}, $\psi^*$ and $v^*$ are increasing if $\rho \geq 0$. Moreover, $\psi^*$ and $v^*$ are increasing in example \ref{eg:r&d}. The details are omitted.
\end{remark}

\begin{example}
\label{eg:js_adap_continue3}
Recall the job search model of examples \ref{eg:js_adap},
\ref{eg:js_adap_continue1} and \ref{eg:js_adap_continue2}. Note that $r(w)$ is increasing, $\mu'$ is increasing in $\mu$, and 
$f(w'|\mu, \gamma) 
	= N(\mu, \gamma + \gamma_{\epsilon})$ 
is stochastically increasing in $\mu$. So
$\EE_{\mu, \gamma} (r(w') \vee \psi(\mu', \gamma'))$ is increasing in $\mu$ for all
candidate $\psi$ that is increasing in $\mu$. Since $c \equiv c_0$, by proposition \ref{pr:mono}, $\psi^*$ and $v^*$ are increasing in $\mu$. Since $r$ is increasing 
in $w$, $v^* = r \vee \psi^*$ is increasing in $w$.
\end{example}

\begin{example}
\label{eg:fej_continue2}
Recall the firm exit model of examples \ref{eg:firm_exit_j} and 
\ref{eg:fej_continue1}. 
The flow continuation payoff $\pi(p,x) = p^2 / (4x)$ is increasing in $p$ and decreasing in $x$. 
Since $P(r \vee \psi^*)$ is not a function of $x$, $\psi^*$ is decreasing in $x$. 
If the exit payoff $r(p,x)$ is decreasing in $x$, then $v^*$ is decreasing in $x$. 
If $\rho \geq 0$ and $r(p,x)$ is increasing in $p$, since 
$h(p'|p)= LN(\rho \ln p + b, \gamma_p)$ is stochastically increasing, 
$P(r \vee \psi)$ is increasing in $p$ for all candidate $\psi$ that is increasing in $p$. By proposition \ref{pr:mono}, $\psi^*$ and $v^*$ are increasing in $p$. 
Recall that $\mu'$ is increasing in $\mu$. Since
$f(x'|\mu, \gamma) := LN(\mu, \gamma + \gamma_{\epsilon})$ is stochastically increasing in $\mu$, $P(r \vee \psi)$ is decreasing in $\mu$ for all candidate
$\psi$ that is decreasing in $\mu$. By proposition \ref{pr:mono}, $\psi^*$ and $v^*$ are decreasing in $\mu$.
\end{example}

\subsection{Differentiability}
\label{ss:diff}

Suppose $\ZZ \subset \RR^m$. For $i = 1, ..., m$, let $\ZZ^{(i)}$ be 
the $i$-th dimension and $\ZZ^{(-i)}$ the remaining $m-1$ dimensions of $\ZZ$. 
A typical element $z \in \ZZ$ takes form of $z = (z^1, ..., z^m)$. Let $z^{-i} := (z^1, ..., z^{i-1}, z^{i+1}, ..., z^m)$. Given $z_0 \in \ZZ$ and $\delta >0$, let 
$B_{\delta}(z_0^i) 
	:= \{ z^i \in \ZZ^{(i)}: 
			|z^i-z_0^i| < \delta 
		\}$ 
and $\bar{B}_{\delta}(z_0^i)$ be its closure.

Given $h: \ZZ \rightarrow \RR$, let
$D_i h(z) := \partial h(z) / \partial z^i$
and $D_i^2 h(z) := \partial^2 h(z) / \partial {(z^i)}^2$.
For a density kernel $f$, let
$D_i f(z'|z) := \partial f(z'|z) / \partial z^i$ and 
 $D_i^2 f(z'|z) := \partial^2 f(z'|z) / \partial (z^i)^2$.
Let $\mu(z) 
		:= \int 
					\max \{r(z'), \psi^*(z') \}  
					f(z'|z)
			 dz'$,
$\mu_i(z) 
		:= \int 
					\max \{r(z'), \psi^*(z')\} 
					D_i f(z'|z)
			 dz'$, and
denote $k_1 (z) := r(z)$ and $k_2 (z) := \ell(z)$.

\begin{assumption}
\label{a:c_diff}
	$D_i c(z)$ exists for all $z \in \interior (\ZZ)$ and $i=1,...,m$.
\end{assumption}

\begin{assumption}
\label{a:2nd_diff}
	$P$ has a density representation $f$, and, for $i = 1, ..., m$:
	\begin{enumerate}
		\item $D_i^2 f(z'|z)$ exits for all $(z,z')\in \interior(\ZZ) \times \ZZ$;
		\item $(z,z') \mapsto D_i f(z'|z)$ is continuous;
		\item There are finite solutions of $z^i$ to $D_i^2 f(z'|z)=0$ (denoted by 
			$z_i^* (z', z^{-i})$), and, for all $z_0 \in \interior (\ZZ)$, there exist 
			$\delta>0$ and a compact subset $A \subset \ZZ$ such that 
			$z' \notin A$ implies
			$z_i^* (z', z_0^{-i}) \notin B_{\delta}(z_0^i)$.
	\end{enumerate}
\end{assumption}

\begin{remark}
\label{rm:diff_ubdd_ss}
When the state space is unbounded above and below, for example, a sufficient condition for assumption \ref{a:2nd_diff}-(3) is: there are finite solutions 
of $z^i$ to $D_i^2 f(z'|z)=0$, and, for all $z_0 \in \interior(\ZZ)$, 
$\|z'\| \rightarrow \infty$ implies 
$|z_i^* (z', z_0^{-i})| \rightarrow \infty$.
\end{remark}

\begin{assumption}
\label{a:diff}
	 $k_j$ is continuous, and, 
		$\interior(\ZZ) \ni z \mapsto 
				\int |k_j (z') D_i f(z'|z)| \diff z' 
				\in \RR_+$	
	for $i=1,...,m$ and $j = 1, 2$.
\end{assumption}

The following provides a general result for the differentiability of $\psi^*$.

\begin{proposition}
\label{pr:diff}
	Under assumptions \ref{a:ubdd_drift_gel} and \ref{a:c_diff}--\ref{a:diff}, 
	$\psi^*$ is differentiable at interior points, with 
	$D_i \psi^* (z) = D_i c(z) + \mu_i (z)$ 
	for all $z \in \interior (\ZZ)$ and $i=1,...,m$.
\end{proposition}

\begin{proof}[Proof of proposition~ \ref{pr:diff}]
Fix $z_0 \in \interior (\ZZ)$. By assumption \ref{a:2nd_diff}-(3), there exist 
$\delta >0$ and a compact subset $A \subset \ZZ$ such that $z' \in A^c$ implies $z_i^*(z',z_0^{-i}) 
	\notin B_{\delta} (z_0^i)$, hence 
$\underset{z^i \in \bar{B}_{\delta}(z_0^i)}{\sup} |D_i f(z'|z)| 
= |D_i f(z'|z)|_{z^i = z_0^i + \delta}  \vee
	|D_i f(z'|z)|_{z^i = z_0^i - \delta}$ 
(given $z^{-i} = z_0^{-i}$). By assumption \ref{a:2nd_diff}-(2), 
given $z^{-i} = z_0^{-i}$, there exists  $G \in \RR_+$, such that
\begin{align*}
	\underset{
					z^i \in \bar{B}_{\delta}(z_0^i)
					}
					{\sup} |D_i f(z'|z)| 
	& \leq 
		\underset{
						z' \in A, 
						z^i \in \bar{B}_{\delta}(z_0^i)
						}
						{\sup} |D_i f(z'|z)| 
			   \cdot \1 (z' \in A)    \\
		& + \left( 
					  |D_i f(z'|z)|_{z^i = z_0^i + \delta} 
					  \vee |D_i f(z'|z)|_{z^i = z_0^i - \delta}
			   \right) 
			   \cdot \1 (z' \in A^c)    \\
	& \leq 
		G \cdot \1 (z' \in A)		\\
		& + \left( 
					  |D_i f(z'|z)|_{z^i = z_0^i + \delta} 
					  + |D_i f(z'|z)|_{z^i = z_0^i - \delta} 
			   \right) 
			   \cdot \1 (z' \in A^c). 
\end{align*}
Assumption \ref{a:diff} then shows that condition (2) of lemma \ref{lm:diff_gel} holds. By assumption \ref{a:c_diff} and lemma \ref{lm:diff_gel}, 
$D_i \psi^* (z) = D_i c(z) + \mu_i (z)$, $\forall z \in \interior(\ZZ)$, as was to be shown.
\end{proof}

\subsection{Smoothness}
Now we are ready to study smoothness (i.e., continuous differentiability), an 
essential property for numerical computation and characterizing optimal policies.

\begin{assumption}
\label{a:cont_diff_gel}
	For $i=1,...,m$ and $j=1, 2$, the following conditions hold:
	\begin{enumerate}
	\item $z \mapsto D_i c(z)$ is continuous on $\interior (\ZZ)$;
	\item $k_j$ is continuous, and, $z \mapsto \int |k_j (z') D_i f(z'|z)| \diff z'$
	 		is continuous on $\interior (\ZZ)$.
	\end{enumerate}
\end{assumption}

The next result provides sufficient conditions for smoothness.

\begin{proposition}
\label{pr:cont_diff_gel}
	Under assumptions \ref{a:ubdd_drift_gel}, \ref{a:2nd_diff} and 
	\ref{a:cont_diff_gel}, $z \mapsto D_i \psi^*(z)$ is continuous on 
	$\interior(\ZZ)$ for $i = 1,..., m$.
\end{proposition}

\begin{proof}[Proof of proposition~ \ref{pr:cont_diff_gel}]
	Since assumption \ref{a:cont_diff_gel} implies assumptions \ref{a:c_diff} and 
	\ref{a:diff}, by proposition \ref{pr:diff}, 
	$D_i \psi^*(z) = D_i c(z) + \mu_i (z)$ on $\interior(\ZZ)$.
	Since $D_i c (z)$ is continuous by assumption 
	\ref{a:cont_diff_gel}-(1), to show that $\psi^*$ is continuously differentiable, it 
	remains to verify: $z \mapsto \mu_i (z)$ is continuous on $\interior(\ZZ)$.
	Since $|\psi^*| \leq G \ell$ for some $G \in \RR_+$,
	\begin{equation}
	\label{eq:contdiff_bd}
	\left|
			\max \{
						r(z'), \psi^*(z') 
					 \} 
	  		D_i f(z'|z)
	  \right| 
	   	 \leq 
	   	 		(|r(z')| + G \ell(z')) |D_i f(z'|z)|,
	  \;
	  \forall z',z \in \ZZ.
	\end{equation} 
	By assumptions \ref{a:2nd_diff} and \ref{a:cont_diff_gel}-(2), the right side of 
	\eqref{eq:contdiff_bd} is continuous in $z$, and
	$z \mapsto 
			\int 
				[|r(z')| + G \ell(z')]|D_i f(z'|z)| 
			\diff z'$ 
	is continuous. 
	Lemma \ref{lm:cont} then implies that 
	$z \mapsto 
			\mu_i (z)
		= \int \max \{ r(z'), \psi^*(z') \} 
				   D_i f(z'|z) 
			 dz' $ 
	is 	continuous, as was to be shown.
\end{proof}

\begin{example}
\label{eg:jsll_continue4}
Recall the job search model of example \ref{eg:jsll} (subsequently studied by examples 	\ref{eg:jsll_continue1}, \ref{eg:jsll_continue2} and \ref{eg:jsll_continue3}).
For all $a \in \RR$, let 
$h(z) := e^{ a (\rho z + b)
					+ a^2 \sigma^2 / 2} 
						/ 
			 \sqrt{2 \pi \sigma^2}$.
We can show that the following statements hold:
\begin{enumerate}
	\item[(a)] There are two solutions to $\frac{\partial^2 f(z'|z)}{\partial z^2} = 0
				\colon$ 
					$z^*(z' ) = \frac{z' - b \pm \sigma}{\rho}$;
	
	\item[(b)] $\int \left| 
									\frac{\partial f(z'|z)}{\partial z}
							  \right|
					   \diff z'
					   = 
					   		\frac{|\rho|}{\sigma} \sqrt{\frac{2}{\pi}}$;
	
	\item[(c)] $\left| 
							z' \frac{\partial f(z'|z)}{\partial z}
					   \right|
					   \leq 
					   		\frac{1}{\sqrt{2 \pi \sigma^2}} 
		   							 \exp \left\{
											 	- \frac{(z' - \rho z - b)^2}{2 \sigma^2}
											  \right\}
							\left(
								\frac{|\rho|}{\sigma^2}   {z'^2} +
								\frac{|\rho(\rho z + b)|}{\sigma^2} |z'|
							\right)$;
	\item[(d)] $e^{a z'} 
					  \left|
							\frac{\partial f(z'|z)}{\partial z}
					  \right|
							\leq   h (z) 	
				 					 \exp \left\{ 
												- \frac{
															[z' - (\rho z + b + a \sigma^2)]^2
										  				   }{2 \sigma^2}
						  					  \right\}
				 					\frac{|\rho z'| + |\rho (\rho z + b)|}{\sigma^2}$,
				 	 $\forall a \in \RR$;
	\item[(e)] The four terms on both sides of (c) and (d) are continuous in $z$;
	
	\item[(f)] The integrations (w.r.t. $z'$) of the two terms on the right side of (c) 
					and (d) are continuous in $z$.				   
\end{enumerate}
Remark \ref{rm:diff_ubdd_ss} and (a) imply that assumption \ref{a:2nd_diff}-(3) 
holds. If $\delta=1$, assumption \ref{a:cont_diff_gel}-(2) holds by conditions (b), (c), (e), (f) and lemma \ref{lm:cont}. If $\delta \geq 0$ and $\delta \neq 1$, based on \eqref{eq:e_ntimes} (appendix B), conditions (b) and (d)--(f), and lemma \ref{lm:cont}, we can show that assumption 
\ref{a:cont_diff_gel}-(2) holds.
The other assumptions of proposition \ref{pr:cont_diff_gel} are easy to verify. Hence, $\psi^*$ is continuously differentiable.
\end{example}

\begin{example}
\label{eg:perp_option_continue2}
	Recall the option pricing problem of example \ref{eg:perpetual_option} 
	(subsequently studied by example \ref{eg:perp_option_continue1} and remark 
	\ref{rm:po_fe_rd}).
	Through similar analysis as in example \ref{eg:jsll_continue4}, we can show that 
	$\psi^*$ is continuously differentiable.\footnote{
		This holds even if the exit payoff $r(z) = (p(z) - K)^+$ has a kink at 
		$z = p^{-1} (K)$. Hence, the differentiability of the exit payoff is not 
		necessary for the smoothness of the continuation value. }
\end{example}

\begin{example}
Recall the firm exit model of example \ref{eg:firm_exit} (subsequently studied by example	\ref{eg:firm_exit_continue1} and remark \ref{rm:po_fe_rd}). 
Through similar analysis to examples \ref{eg:jsll_continue4}--\ref{eg:perp_option_continue2}, we can show that $\psi^*$ is continuously differentiable. Figure \ref{fig:vf_cvf} illustrates. We set  
$\beta = 0.95$, $\sigma = 1$, $b = 0$, $c_f = 5$, $\alpha=0.5$, $p=0.15$, $w=0.15$, and consider respectively $\rho = 0.7$ and $\rho = -0.7$. While $\psi^*$ is smooth, $v^*$ is kinked at around 
$z = 1.5$ when $\rho = 0.7$, and has two kinks when $\rho = -0.7$. 
\begin{figure}[h]
\centering
\begin{minipage}{.55\textwidth}
  \centering
  \includegraphics[width=1\linewidth]{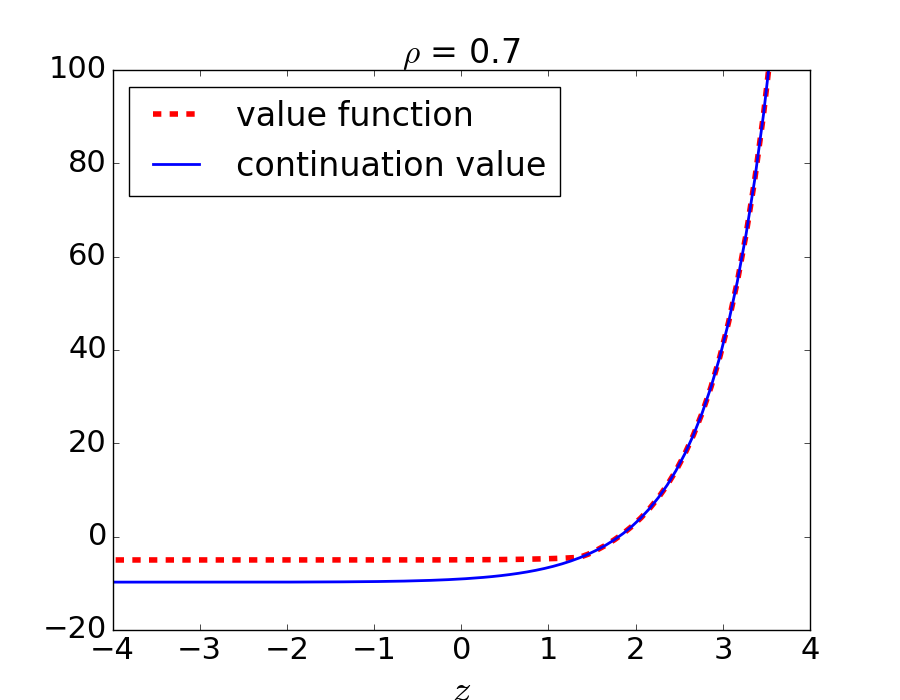}
\end{minipage}%
\begin{minipage}{.55\textwidth}
  \centering
  \includegraphics[width=1\linewidth]{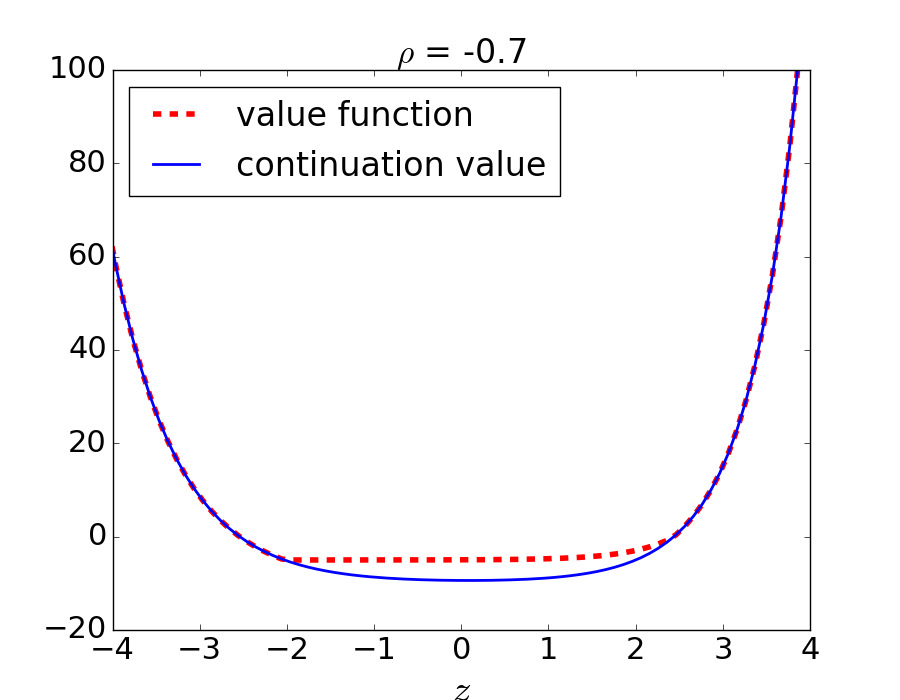}
\end{minipage}
\caption{Comparison of $\psi^*$ and $v^*$}
\label{fig:vf_cvf}
\end{figure}
\end{example}

\subsection{Parametric Continuity}


Consider the parameter space $\Theta \subset \RR^k$. Let $P_{\theta}, r_{\theta}$, $c_{\theta}$, $v^*_{\theta}$ and $\psi_{\theta}^*$ denote the stochastic kernel, exit and flow continuation payoffs, value and continuation value functions with respect to the parameter $\theta \in \Theta$, respectively. Similarly, let $n_\theta$, $m_\theta$, $d_\theta$ and $g_\theta$ denote the key elements in assumption \ref{a:ubdd_drift_gel} with respect to $\theta$. 
Define $n := \underset{\theta \in \Theta}{\sup} \;  n_{\theta}$,
$m:=\underset{\theta \in \Theta}{\sup} \; m_{\theta}$ and 
$d := \underset{\theta \in \Theta}{\sup} \; d_{\theta}$.
\begin{assumption}
\label{a:para_ubdd}
	Assumption \ref{a:ubdd_drift_gel} holds at all $\theta \in \Theta$, with
	$\beta m < 1$ and $n, d < \infty$.
\end{assumption}
Under this assumption, let $m'>0$ and $d' > 0$ such that $m + 2m'>1$, $\beta (m + 2m')<1$ and 
$d' \geq d / (m + 2m' - 1)$.
Consider $\ell: \ZZ \times \Theta \rightarrow \RR$ defined by 
\begin{equation*}
	\ell(z, \theta) := 
			m' \left( \sum_{t=1}^{n-1} \EE_z^{\theta} |r_{\theta} (Z_t)| +
						  \sum_{t=0}^{n-1} \EE_z^{\theta} |c_{\theta} (Z_t)|
				\right)
			+ g_{\theta}(z) + d',
\end{equation*}
where $\EE_z^{\theta}$ denotes the conditional expectation with respect to 
$P_{\theta} (z, \cdot)$.


\begin{remark}
We implicitly assume that $\Theta$ does not include $\beta$. However, by letting $\beta \in [0, a]$ and $a\in [0,1)$, we can incorporate $\beta$ into $\Theta$.
$\beta m < 1$ in assumption \ref{a:para_ubdd} is then replaced by $am<1$. All the
parametric continuity results of this paper remain true after this change.
\end{remark}

\begin{assumption}
\label{a:para_feller}
	$P_{\theta}(z, \cdot)$ satisfies the Feller property, i.e., 
	$(z, \theta) 
		\mapsto \int h(z') P_{\theta}(z, \diff z')$ is continuous for all bounded 
	continuous function $h: \ZZ \rightarrow \RR$.
\end{assumption}

\begin{assumption}
\label{a:para_cont} 
	$(z, \theta) \mapsto c_{\theta}(z), 
									 r_{\theta}(z), 
									 \ell(z, \theta), 
									 \int |r_{\theta} (z')| P_{\theta} (z, \diff z'),					   
									 \int \ell (z', \theta) P_{\theta} (z, \diff z')$
	are continuous.
\end{assumption}


The following result is a simple extension of proposition \ref{pr:cont}. We omit 
its proof.

\begin{proposition}
\label{pr:para_cont_gel}
	Under assumptions \ref{a:para_ubdd}--\ref{a:para_cont},
	$(z, \theta) \mapsto \psi^*_{\theta}(z), v^*_{\theta}(z)$ are continuous.
\end{proposition}

\begin{example}
Recall the job search model of example \ref{eg:jsll} (subsequently studied by examples \ref{eg:jsll_continue1}, \ref{eg:jsll_continue2}, \ref{eg:jsll_continue3} and
\ref{eg:jsll_continue4}). Let the parameter space 
$\Theta := [-1, 1] \times A \times B \times C$, 
where $A, B$ are bounded subsets of $\RR_{++}, \RR$, respectively, and $C \subset \RR$. A typical element $\theta \in \Theta$ is $\theta = (\rho, \sigma, b, c_0)$. Proposition \ref{pr:para_cont_gel} implies that $(\theta, z) \mapsto \psi^*_{\theta}(z)$ and $(\theta, z) \mapsto v^*_{\theta}(z)$ are continuous. The proof is similar to
example \ref{eg:jsll_continue2} and omitted.
\end{example}

\begin{remark}
The parametric continuity of all the other examples discussed above can be 
established in a similar manner. 
\end{remark}

\section{Optimal Policies}
\label{s:opt_pol}

In this section, we provide a systematic study of optimal timing of decisions when there are threshold states, and explore the key properties of the optimal policies.

\subsection{Conditional Independence in Transitions}
\label{ss:ci}

For a broad range of problems, the continuation value function exists in a lower dimensional space than the value function. Moreover, the relationship is asymmetric. While each state variable that appears in the continuation value must appear in the value function, the converse is not true. The continuation value function can have strictly fewer arguments than the value function (recall example \ref{eg:js_adap}).

To verify, suppose that the state space $\ZZ \subset \RR^{m}$ and can be written as $\ZZ = \XX \times \YY$, where $\XX$ is a convex subset of $\RR^{m_0}$, $\YY$ is a convex subset of $\RR^{m-m_0}$, and $m_0 \in \NN$ such that $m_0 < m$.
The state process $(Z_t)_{t \geq 0}$ is then $ \{(X_t, Y_t)\}_{t \geq 0}$, where $(X_t)_{t \geq 0}$ and $(Y_t)_{t \geq 0}$ are two stochastic processes taking values in $\XX$ and $\YY$, respectively. In particular, for each $t \geq 0$, $X_t$ represents the first $m_0$ dimensions and $Y_t$ the rest $m-m_0$ dimensions of the period-$t$ state $Z_t$. 

Assume that the stochastic processes $(X_t)_{t \geq 0}$ and $(Y_t)_{t \geq 0}$ are 
\textit{conditionally independent}, in the sense that conditional on each $Y_t$, the 
next period states $(X_{t+1},Y_{t+1})$ and $X_t$ are independent. 
Let $z := (x,y)$ and $z' := (x',y')$ be the current and next period states, 
respectively. With conditional independence, the stochastic kernel 
$P(z, \diff z')$ can be represented by the conditional distribution of $(x',y')$ on $y$, 
denoted as $\mathbb{F}_y (x',y')$, i.e., 
$P(z,\diff z') = P((x,y),\diff (x',y')) = \diff \mathbb{F}_y (x',y')$. 

Assume further that the flow continuation payoff  $c$ is defined 
on $\YY$, i.e., $c: \YY \rightarrow \RR$.\footnote{
		Indeed, in many applications, the flow payoff $c$ is a constant, as seen in 
		previous examples.} 
Under this setup, $\psi^*$ has strictly fewer arguments than $v^*$. While $v^*$ is a function of both $x$ and $y$, $\psi^*$ is a function of $y$ only. Hence, the continuation value based method allows us to mitigate one of the primary stumbling 
blocks for numerical dynamic programming: the so-called curse of dimensionality 
(see, e.g., \cite{bellman1969new}, \cite{rust1997using}).

\subsection{The Threshold State Problem}
\label{ss:tsp}

Among problems where conditional independence exists, the optimal policy is 
usually determined by a reservation rule, in the sense that the decision process 
terminates whenever a specific state variable hits a threshold level. In such cases, 
the continuation value based method allows for a sharp analysis of the optimal 
policy. This type of problem is pervasive in quantitative and theoretical economic 
modeling, as we now formulate.

For simplicity, we assume that $m_0 = 1$, in which case $\XX$ is a convex subset of $\RR$ and $\YY$ is a convex subset of $\RR^{m-1}$. For each $t \geq 0$, $X_t$ represents the first dimension and $Y_t$ the rest $m-1$ dimensions of the period-$t$ state $Z_t$. If, in addition, $r$ is monotone on $\XX$, we call $X_t$ the 
\textit{threshold state} and $Y_t$ the \textit{environment state} (or 
\textit{environment}) of period $t$, moreover, we call $\XX$ the \textit{threshold 
state space} and $\YY$ the \textit{environment space}. 

\begin{assumption}
\label{a:opt_pol}
	$r$ is strictly monotone on $\XX$. Moreover, for all $y\in \YY$, there exists 
	$x\in \XX$ such that $r(x,y) = c(y) + \beta \int v^*(x',y') \diff 
	\mathbb{F}_y(x',y')$.
\end{assumption}

Under assumption \ref{a:opt_pol}, the \textit{reservation rule property} holds.  When the exit payoff $r$ is strictly increasing in $x$, for instance, this property states that
if the agent terminates at $x \in \XX$ at a given point of time, then he would have terminated at any higher state at that moment. Specifically, there is a \textit{decision threshold} $\bar{x}:\YY \rightarrow \XX$ such that when $x$ attains this threshold level, i.e., $x = \bar{x}(y)$, the agent is indifferent between stopping and continuing, i.e., $r(\bar{x}(y), y) = \psi^*(y)$ for all $y \in \YY$. 

As shown in theorem \ref{t:bk}, the optimal policy satisfies 
$\sigma^*(z) = \1 \{r(z) \geq \psi^*(z)\}$. For a sequential decision 
problem with threshold state, this policy is fully specified by the decision threshold $\bar{x}$. In particular, under assumption \ref{a:opt_pol}, we have
\begin{equation}
\label{eq:res_rule_pol}
    \sigma^*(x,y)
    	 = \left\{
               	 \begin{array}{ll}
                 	 \1 \{ x \geq \bar{x}(y) \}, \; 
                 	 \mbox{ if $r$ is strictly increasing in $x$}
					\\
                 	 \1 \{ x \leq \bar{x}(y)\}, \; 
                 	 \mbox{ if $r$ is strictly decreasing in $x$} 	\\
                \end{array}
            \right.
\end{equation}

Further, based on properties of the continuation value, properties of the decision 
threshold can be established. The next result provides sufficient conditions for 
continuity. The proof is similar to proposition \ref{pr:pol_para_cont} below and thus 
omitted.

\begin{proposition}
\label{pr:pol_cont}
	Suppose either assumptions of proposition \ref{pr:cont} or of
	corollary \ref{cr:cont_dst} hold, and that assumption \ref{a:opt_pol} holds. 
	Then $\bar{x}$ is continuous.
\end{proposition}

The next result discusses monotonicity. The proof is obvious and we omit it.  

\begin{proposition}
\label{pr:pol_mon}
	Suppose assumptions of proposition \ref{pr:mono} and assumption 
	\ref{a:opt_pol} hold, and that $r$ is defined on $\XX$. If $\psi^*$ is increasing 
	and $r$ is strictly increasing (resp. decreasing), then $\bar{x}$ is increasing (resp. 
	decreasing). If $\psi^*$ is decreasing and $r$ is strictly increasing (resp. 
	decreasing), then $\bar{x}$ is decreasing (resp. increasing).
\end{proposition}

A typical element $y \in \YY$ is $y = \left( y^1, ..., y^{m-1} \right)$. For given 
functions $h: \YY \rightarrow \RR$ and $l: \XX \times \YY \rightarrow \RR$, 
define $D_i h(y) := \partial h(y) / \partial y^i$, 
$D_i l(x,y) := \partial l(x,y) / \partial y^i$, and 
$D_x l(x,y) := \partial l(x,y) / \partial x$.
The next result follows immediately from proposition \ref{pr:cont_diff_gel} and the implicit function theorem.

\begin{proposition}
\label{pr:pol_diff}
	Suppose assumptions of proposition \ref{pr:cont_diff_gel} and 
	assumption \ref{a:opt_pol} hold, and that $r$ is continuously differentiable on 
	$\interior (\ZZ)$. Then $\bar{x}$ is continuously differentiable on 
	$\interior (\YY)$. In particular, 
	$D_i \bar{x}(y) 
			= - \frac{
							D_i r(\bar{x}(y),y) - D_i \psi^*(y)
						  }{
						    D_x r(\bar{x}(y),y)
						   }$ 
	for all $y \in \interior (\YY)$.
\end{proposition}

Intuitively, $(x,y) \mapsto r(x ,y) - \psi^*(y)$ denotes the premium of terminating 
the decision process. Hence, 
$(x,y) \mapsto D_i r(x,y) - D_i \psi^*(y), D_x r(x,y)$ are the instantaneous rates of change of the terminating premium in response to changes in $y^i$ and $x$, respectively. Holding aggregate premium null, the premium changes due to changes in $x$ and $y$ cancel out. As a result, the rate of change of $\bar{x}(y)$ with respect to changes in $y^i$ is equivalent to the ratio of the instantaneous rates of change in the premium. The negativity is due to zero terminating premium at the decision threshold.

Let $\bar{x}_\theta$ be the decision threshold with respect to $\theta \in \Theta$. We have the following result for parametric continuity.

\begin{proposition}
\label{pr:pol_para_cont}
	Suppose assumptions of proposition \ref{pr:para_cont_gel} and 
	assumption \ref{a:opt_pol} hold. Then $(y,\theta) \mapsto \bar{x}_{\theta}(y)$ 
	is continuous.
\end{proposition}

\begin{proof}[Proof of proposition \ref{pr:pol_para_cont}]
	Define $F: \XX \times \YY \times \Theta \rightarrow \RR$ by 
	$F(x,y, \theta) := r_{\theta}(x,y) - \psi_{\theta}^*(y)$. Without loss of generality, 
	assume that $(x,y,\theta) \mapsto r_{\theta}(x,y)$ is strictly increasing in $x$, 
	then $F$ is strictly increasing in $x$ and continuous. 
	For all fixed $(y_0, \theta_0) \in \YY \times \Theta$ and $\epsilon>0$, 
	since $F$ is strictly increasing in $x$ and 
	$F(\bar{x}_{\theta_0}(y_0), y_0, \theta_0)=0$, we have
	\begin{equation*}
		F(\bar{x}_{\theta_0}(y_0) + \epsilon, y_0, \theta_0) > 0 
		\quad \mbox{and} \quad
		F(\bar{x}_{\theta_0}(y_0) - \epsilon, y_0, \theta_0) < 0.
	\end{equation*}
	Since $F$ is continuous with respect to $(y,\theta)$, there exists $\delta>0$
	such that for all 
	$(y,\theta) \in B_{\delta}((y_0,\theta_0)) 
	 		 := \left\{
	  					(y, \theta) \in \YY \times \Theta: 
						\| (y,\theta) - (y_0,\theta_0) \| < \delta 
			  	  \right\}$,
 	we have
	\begin{equation*}
		F(\bar{x}_{\theta_0}(y_0) + \epsilon, y, \theta) > 0 
		\quad \mbox{and} \quad 
		F(\bar{x}_{\theta_0}(y_0) - \epsilon, y, \theta) < 0.
	\end{equation*}
	Since $F(\bar{x}_{\theta}(y), y, \theta)=0$ and $F$ is strictly increasing in $x$, 
	we have
	\begin{equation*}
		\bar{x}_{\theta}(y) \in
			 \left(
			 		\bar{x}_{\theta_0}(y_0) - \epsilon, 
			  		\bar{x}_{\theta_0}(y_0) + \epsilon
			 \right),
		\mbox{ i.e., }
		|\bar{x}_{\theta}(y)-\bar{x}_{\theta_0}(y_0)|<\epsilon.
	\end{equation*}
	Hence, 
	$(y, \theta) \mapsto \bar{x}_{\theta}(y)$ is continuous, as was to be shown.
\end{proof}

\section{Applications}
\label{s:application}

In this section we consider several typical applications in economics, and compare the computational efficiency of continuation value and value function based methods. Numerical experiments show that the partial impact of lower dimensionality of the continuation value can be huge, even when the difference between the arguments of this function and the value function is only a single variable.

\subsection{Job Search II}  
\label{ss:js_ls}

Consider the adaptive search model of \cite{ljungqvist2012recursive} (section 
6.6). The model is as example \ref{eg:js_1}, apart from 
the fact that the distribution of the wage process $h$ is unknown. The worker 
knows that there are two possible densities $f$ and $g$, and puts 
prior probability $\pi_t$ on $f$ being chosen. If the current offer $w_t$ is rejected, a new offer $w_{t+1}$ is observed at the beginning of next period, and, by the Bayes' rule, $\pi_t$ updates via 
\begin{equation}
\label{eq:pi'}
	\pi_{t+1} 
		= \pi_t f(w_{t+1})
				    /
					 [ \pi_t f(w_{t+1}) + (1 - \pi_t) g(w_{t+1}) ]
		=: q (w_{t+1}, \pi_t).
\end{equation}
The state space is $\ZZ := \XX \times [0,1]$, where $\XX$ is a compact interval
of $\RR_+$.  Let $u(w) := w$. The value function of the unemployed worker satisfies
\begin{equation*}
	v^*(w, \pi) 
		= \max \left\{ \frac{w}{1 - \beta}, 
								c_0 + \beta \int 
														   v^*(w', q(w', \pi) ) 
														   h_\pi (w') 
													\diff w' 
					 \right\}, 
\end{equation*}
where $h_\pi (w') := \pi  f(w') + (1 - \pi) g(w')$. This is a typical threshold state 
problem, with threshold state $x := w \in \XX$ and environment 
$y := \pi \in [0,1] =: \YY$.
As to be shown, the optimal policy is determined by a reservation wage  
$\bar{w}:[0,1] \rightarrow \RR$ such that when $w = \bar{w}(\pi)$, the worker is 
indifferent between accepting and rejecting the offer. Consider the candidate space $(b[0,1], \| \cdot \|)$. The Jovanovic operator is
\begin{equation}
	\label{eq:rr_job}
	Q \psi (\pi) 
		= c_0 + \beta \int 
									 \max \left\{ 
													\frac{w'}{1-\beta}, 
													\psi \circ q(w',\pi) 
											 \right\} 
									 h_\pi (w') 
							   \diff w'.
\end{equation}

\begin{proposition}
\label{pr:js_ls}
	Let $c_0 \in \XX$. The following statements are true:		
	\begin{enumerate}
		\item[1.] $Q$ is a contraction on $(b[0,1], \| \cdot \|)$ of 
				 modulus 	$\beta$, with unique fixed point $\psi^*$.	
		\item[2.] The value function 
				 $v^*(w,\pi) 
				 		= \frac{w}{1-\beta} 
				 				\vee 
							\psi^*(\pi)$, 
				 reservation wage 
				 $\bar{w}(\pi) 
				 		= (1 - \beta) \psi^*(\pi)$, 
				 and optimal policy 
				 $\sigma^*(w, \pi) 
				 		= \1 \{ w \geq \bar{w}(\pi) \}$ 
				 for all $(w,\pi) \in \ZZ$.
		\item[3.] $\psi^*$, $\bar{w}$ and $v^*$ are continuous.
	\end{enumerate}
\end{proposition}
%
%
%

Since the computation is 2-dimensional via value function iteration (VFI), and is only 
1-dimensional via continuation value function iteration (CVI), we expect the 
computation via CVI to be much faster. We run several groups of tests and compare 
the time taken by the two methods. All tests are processed in a standard Python 
environment on a laptop with a 2.5 GHz Intel Core i5 and 8GB RAM.

\subsubsection{Group-1 Experiments}
\label{sss:g1}

This group documents the time taken to compute the fixed point across different parameter values and at different precision levels. Table \ref{tb:exp_g1} provides 
the list of experiments performed and table \ref{tb:result_g1} shows the result. 

\begin{table}[h]
	\caption{Group-1 Experiments}
	\label{tb:exp_g1}
	\vspace*{-0.3cm}
	\begin{center}
	\begin{threeparttable}
	\begin{tabular}{|c|c|c|c|c|c|}
		\hline 
		Parameter & Test 1 & Test 2 & Test 3 & Test 4 & Test 5\tabularnewline
		\hline 
		\hline 
		$\beta$ & $0.9$ & $0.95$ & $0.98$ & $0.95$ & $0.95$\tabularnewline
		\hline 
		$c_0$ & $0.6$ & $0.6$ & $0.6$ & $0.001$ & $1$\tabularnewline
		\hline 
		\end{tabular}
	\begin{tablenotes}
      \fontsize{9pt}{9pt}\selectfont
      \item Note: Different parameter values in each experiment.
    \end{tablenotes}
 	\end{threeparttable}
 	\par\end{center}
\end{table}
\begin{table}[h]
\caption{Time Taken of Group-1 Experiments }
\label{tb:result_g1}
\vspace*{-0.3cm}
	\noindent \begin{center}
	\begin{threeparttable}
	\begin{tabular}{|c|c|c|c|c|c|c|c|}
	\hline 
	\multicolumn{2}{|c|}{Test/Method/Precision} & $10^{-3}$ & $10^{-4}$  & 				$10^{-5}$ & $10^{-6}$  & $10^{-7}$ & $10^{-8}$\tabularnewline
	\hline 
	\hline 
	\multirow{2}{*}{Test 1} & VFI & $114.17$ & $140.94$ & $174.91$ & $201.77$ & 	$228.59$ & $255.67$\tabularnewline
	\cline{2-8} 
 	& CVI & $0.67$ & $0.92$ & $1.16$ & $1.43$ & $1.71$ & $1.94$\tabularnewline
	\hline 
	\multirow{2}{*}{Test 2} & VFI & $181.78$ & $234.58$ & $271.89$ & $323.22$ & $339.87$ & $341.55$\tabularnewline
	\cline{2-8} 
 	& CVI & $0.95$ & $1.49$ & $1.80$ & $2.27$ & $2.69$ & $3.11$\tabularnewline
	\hline 
	\multirow{2}{*}{Test 3} & VFI & $335.78$ & $335.87$ & $335.28$ & $335.91$ & 	$338.70$ & $334.21$\tabularnewline
	\cline{2-8} 
 	& CVI & $1.77$ & $2.68$ & $3.08$ & $3.03$ & $3.03$ & $3.06$\tabularnewline
	\hline 
	\multirow{2}{*}{Test 4} & VFI & $154.18$ & $201.05$ & $247.72$ & $294.90$ & 	$335.32$ & $335.00$\tabularnewline
	\cline{2-8} 
 	& CVI & $0.79$ & $1.22$ & $1.65$ & $2.06$ & $2.50$ & $2.91$\tabularnewline
	\hline 
	\multirow{2}{*}{Test 5} & VFI & $275.41$ & $336.02$ & $326.33$ & $327.41$ & $327.11$ & $327.71$\tabularnewline
	\cline{2-8} 
 	& CVI & $1.33$ & $2.12$ & $2.79$ & $2.99$ & $2.97$ & $2.97$\tabularnewline
	\hline 
	\end{tabular}
	\begin{tablenotes}
      \fontsize{9pt}{9pt}\selectfont
      \item Note: We set $\XX = [0,2]$, $f = \mbox{Beta}(1,1)$ and 
      $g = \mbox{Beta}(3, 1.2)$. The grid points of $(w, \pi)$ lie in 
      $[0, 2] \times [ 10^{-4}, 1-10^{-4}]$ with $100$ points for $w$ 
      and $50$ for $\pi$. For each given test and level of precision, 
      we run the simulation $50$ times for CVI, 20 times for VFI, and 
      calculate the average time (in seconds). 
    \end{tablenotes}
	\end{threeparttable}
	\end{center}
\end{table}

As shown in table \ref{tb:result_g1}, CVI performs much better than VFI. On 
average, CVI is $141$ times faster than VFI. In the best case, CVI is $207$ times 
faster (in test 5, VFI takes $275.41$ seconds to achieve a level of accuracy 
$10^{-3}$, while CVI takes only $1.33$ seconds). In the worst case, CVI is 
$109$ times faster (in test 5, CVI takes $2.99$ seconds as opposed to $327.41$ 
seconds by VFI to attain a precision level $10^{-6}$).

\subsubsection{Group-2 Experiments}
\label{sss:g2}

In applications, increasing the number of grid points provides more accurate 
numerical approximations. This group of tests compares how the two 
approaches perform under different grid sizes. The setup and result are summarized 
in table \ref{tb:exp_g2} and table \ref{tb:result_g2}, respectively. 

\begin{table}[h]
\caption{Group-2 Experiments}
\label{tb:exp_g2}
\vspace*{-0.3cm}
	\noindent \begin{center}
	\begin{threeparttable}
	\begin{tabular}{|c|c|c|c|c|c|c|}
	\hline 
	Variable & Test 2 & Test 6  & Test 7  & Test 8  & Test 9  & Test 10\tabularnewline
	\hline 
	\hline 
	$\pi$ & $50$ & $50$ & $50$ & $100$ & $100$ & $100$\tabularnewline
	\hline 
	$w$ & $100$ & $150$ & $200$ & $100$ & $150$ & $200$\tabularnewline
	\hline 
	\end{tabular}
	\begin{tablenotes}
      \fontsize{9pt}{9pt}\selectfont
      \item Note: Different grid sizes of the state variables in each experiment.
    \end{tablenotes}
	\end{threeparttable}
\end{center}
\end{table}
\begin{table}[h]
\caption{Time Taken of Group-2 Experiments}
\label{tb:result_g2}
\vspace*{-0.3cm}
	\noindent \begin{center}		
	\begin{threeparttable}
	\begin{tabular}{|c|c|c|c|c|c|c|c|}
	\hline 
	\multicolumn{2}{|c|}{Test/Precision/Method} & $10^{-3}$ & $10^{-4}$  & 	$10^{-5}$ & $10^{-6}$  & $10^{-7}$ & $10^{-8}$\tabularnewline
	\hline 
	\hline 
	\multirow{2}{*}{Test 2} & VFI & $181.78$ & $234.58$ & $271.89$ & $323.22$ & $339.87$ & $341.55$\tabularnewline
	\cline{2-8} 
	 & CVI & $0.95$ & $1.49$ & $1.80$ & $2.27$ & $2.69$ & 	$3.11$\tabularnewline
	\hline 
	\multirow{2}{*}{Test 6} & VFI & $264.34$ & $336.20$ & $407.52$ & $476.01$ & $508.05$ & $509.05$\tabularnewline
	\cline{2-8} 
	 & CVI & $0.96$ & $1.39$ & $1.82$ & $2.30$ & $2.73$ & 	$3.14$\tabularnewline
	\hline 
	\multirow{2}{*}{Test 7} & VFI & $355.40$ & $449.55$ & $545.51$ & $641.05$ & $679.93$ & $678.28$\tabularnewline
	\cline{2-8} 
	 & CVI & $0.92$ & $1.37$ & $1.79$ & $2.22$ & $2.84$ & 	$3.07$\tabularnewline
	\hline 
	\multirow{2}{*}{Test 8} & VFI & $352.76$ & $447.36$ & $541.75$ & $639.73$ & $678.91$ & $677.52$\tabularnewline
	\cline{2-8} 
	 & CVI & $1.94$ & $2.74$ & $3.58$ & $4.42$ & $5.30$ & 	$6.14$\tabularnewline
	\hline 
	\multirow{2}{*}{Test 9} & VFI & $526.72$ & $670.19$ & $812.66$ & $951.78$ & $1017.29$ & $1015.15$\tabularnewline
	\cline{2-8} 
	 & CVI & $1.81$ & $2.68$ & $3.68$ & $4.33$ & $5.23$ & 	$6.08$\tabularnewline
	\hline 
	\multirow{2}{*}{Test 10} & VFI & $706.34$ & $897.07$ & $1086.15$ & $1278.27$ & $1354.37$ & $1360.07$\tabularnewline
	\cline{2-8} 
	 & CVI & $1.83$ & $2.72$ & $3.51$ & $4.40$ & $5.21$ & 	$6.10$\tabularnewline
	\hline 
	\end{tabular}
	\begin{tablenotes}
      \fontsize{9pt}{9pt}\selectfont
      \item Note: We set $\XX = [0,2]$, $\beta = 0.95$, $c_0 = 0.6$,
      $f = \mbox{Beta}(1,1)$ and $g = \mbox{Beta}(3, 1.2)$. The grid points of 
      $(w, \pi)$ lie in $[0, 2] \times [10^{-4}, 1 - 10^{-4}]$. For each given test and 
      precision level, we run the simulation $50$ times for CVI, 20 times for VFI, 
      and calculate the average time (in seconds). 
    \end{tablenotes}
	\end{threeparttable}
	\par\end{center}
\end{table}
CVI outperforms VFI more obviously as the grid size increases. 
In table \ref{tb:result_g2} we see that as we increase the number of grid 
points for $w$, the speed of CVI is not affected. However, the speed of VFI drops 
significantly. Amongst tests 2, 6 and 7, CVI is $219$ times faster than VFI on 
average. In the best case, CVI is 386 times faster (while it takes VFI $355.40$ 
seconds to achieve a precision level $10^{-3}$ in test 7, CVI takes only $0.92$ 
second). As we increase the grid size of $w$ from $100$ to $200$, CVI is not 
affected, but the time taken for VFI almost doubles. 

As we increase the grid size of both $w$ and $\pi$, there is a slight decrease in the  
speed of CVI. Nevertheless, the decrease in the speed of VFI is exponential. Among 
tests 2 and 8--10, CVI is $223.41$ times as fast as VFI on average. In test 10, VFI 
takes $706.34$ seconds to obtain a level of precision $10^{-3}$, instead, CVI takes 
only $1.83$ seconds, which is 386 times faster.

\subsubsection{Group-3 Experiments}
\label{sss:g3}

Since the total number of grid points increases exponentially with the number of 
states, the speed of computation will drop dramatically with an additional state. To 
illustrate, consider a parametric class problem with respect to $c_0$. We set $\XX = 
[0,2]$, $\beta = 0.95$, $f = \mbox{Beta}(1,1)$ and $g = \mbox{Beta}(3, 1.2)$. Let 
$(w, \pi, c_0)$ lie in $[0, 2] \times [10^{-4}, 1-10^{-4}] \times [0, 1.5]$ with $100$ 
grid points for each. In this case, VFI is $3$-dimensional and suffers the "curse of 
dimensionality": the computation takes more than 7 days. However, CVI is only 
$2$-dimensional and the computation finishes within 171 seconds (with precision 
$10^{-6}$). 

In figure \ref{fig:rw2}, we see that the reservation wage is increasing in $c_0$ and decreasing in $\pi$. Intuitively, a higher level of compensation hinders the agent's incentive of entering into the labor market. Moreover, since $f$ is a less attractive distribution than $g$ and larger $\pi$ means more weight on $f$ and less on $g$, a larger $\pi$ depresses the worker's assessment of future prospects, and relatively low current offers become more attractive.
\begin{figure}[h]
\begin{center}
\includegraphics[scale=0.45]{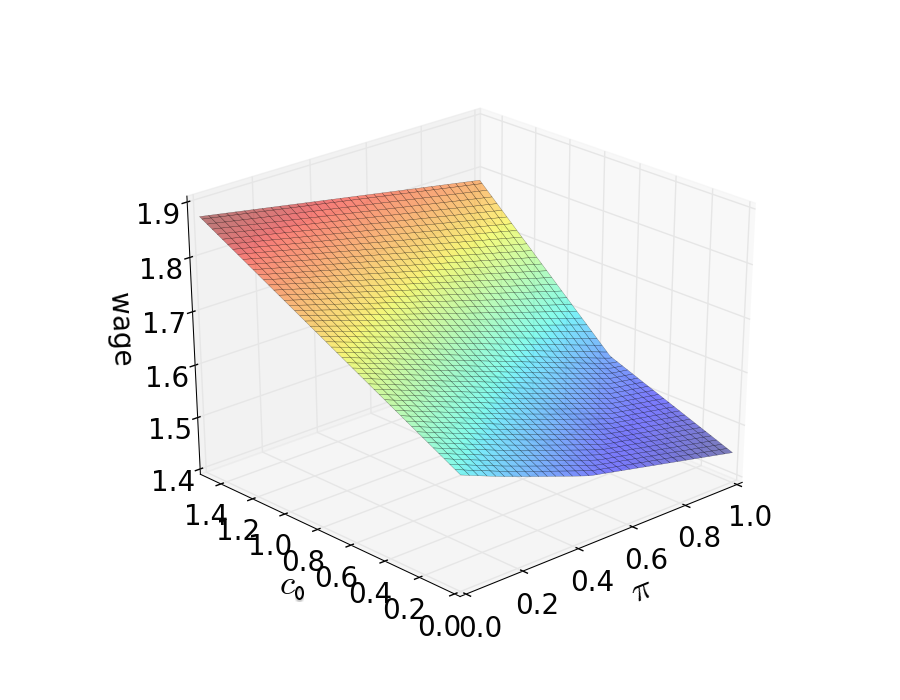}
\caption{The reservation wage}
\label{fig:rw2}
\end{center}
\end{figure}

\subsection{Job Search III}
\label{ss:jsg}

Recall the adaptive search model of example \ref{eg:js_adap} (subsequently studied by examples \ref{eg:js_adap_continue1}, \ref{eg:js_adap_continue2} and
\ref{eg:js_adap_continue3}). The value function satisfies
\begin{equation}
\label{eq:val_jsg}
	v^*(w,\mu,\gamma)
		=\max \left\{ 
							\frac{u(w)}{1-\beta}, 
							c_0 + \beta 	\int v^*(w',\mu',\gamma') 
												   	   f(w'|\mu, \gamma) 
										        \diff w' 
					\right\}.	
\end{equation}
Recall the Jovanovic operator defined by \eqref{eq:cvo_jsadap}.
This is a threshold state sequential decision problem, with threshold state 
$x:=w \in \RR_{++} =: \XX$ and environment
$y := (\mu,\gamma) \in \RR \times \RR_{++} =: \YY$. 
By the intermediate value theorem, assumption \ref{a:opt_pol} holds. Hence, the optimal policy is determined by a 
reservation wage $\bar{w}: \YY \rightarrow \RR$ such that when 
$w=\bar{w}(\mu, \gamma)$, the worker is indifferent between accepting and 
rejecting the offer. Since all the assumptions of proposition \ref{pr:cont} hold (see
example \ref{eg:js_adap_continue2}), by proposition 
\ref{pr:pol_cont}, $\bar{w}$ is continuous. Since $\psi^*$ is increasing in $\mu$ 
(see example \ref{eg:js_adap_continue3}), by proposition \ref{pr:pol_mon}, 
$\bar{w}$ is increasing in $\mu$.


In simulation, we set $\beta=0.95$, $\gamma_{\epsilon} = 1.0$,  
$\tilde{c}_0 = 0.6$, and consider different levels of risk aversion: $\sigma=3,4,5,6$. The grid points of $(\mu,\gamma)$ lie in $[-10,10] \times [10^{-4},10]$, with $200$ points for the $\mu$ grid and $100$ points for the $\gamma$ grid. 
We set the threshold function outside the grid to its value at the closest grid. The integration is computed via Monte Carlo with 1000 draws.\footnote{
	Changing the number of Monte Carlo samples, the grid range and grid density 
	produce almost the same results.} 
Figure \ref{fig:ru} provides the simulation results. There are several key characteristics, as can be seen.

\begin{figure}[h]
\begin{center}
\includegraphics[width=1\textwidth]{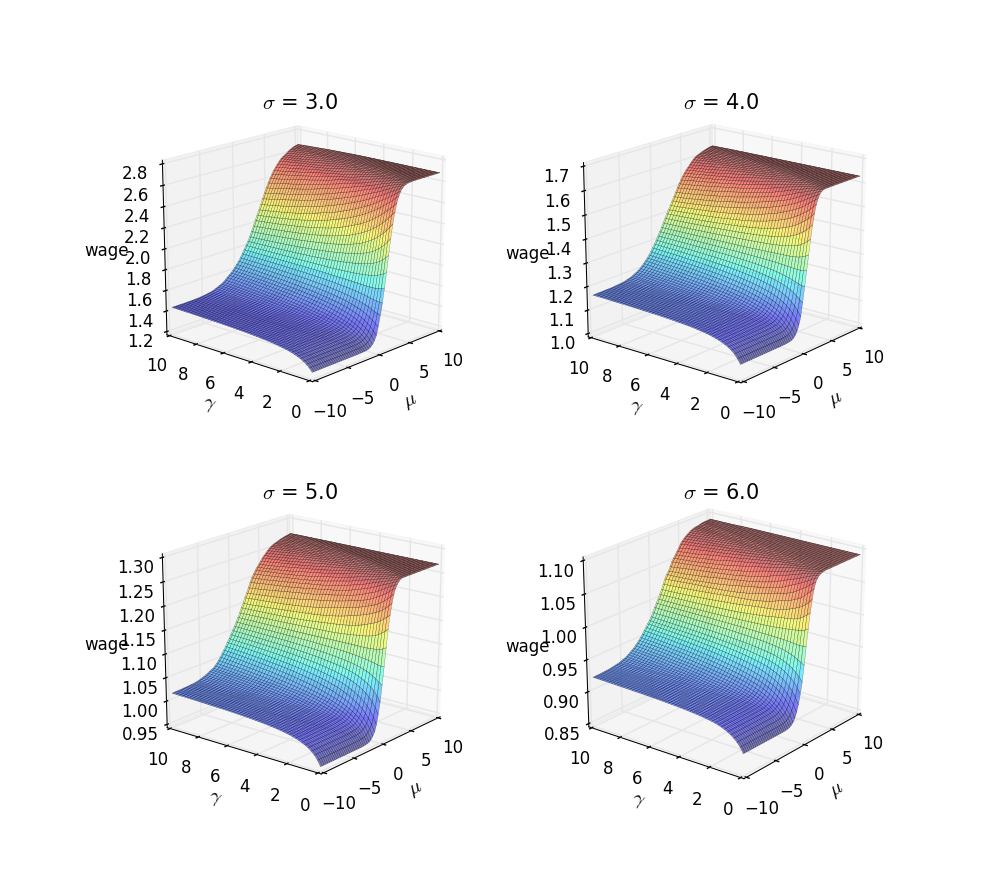}
\caption{The reservation wage}
\label{fig:ru}
\end{center}
\end{figure}

First, in each case, the reservation wage is an increasing function of $\mu$, which 
parallels the above analysis. Naturally, a more optimistic agent (higher $\mu$) 
would expect that higher offers can be obtained, and will not accept the offer until 
the wage is high enough. 

Second, the reservation wage is increasing in $\gamma$ for given $\mu$ of 
relatively small values, though it is decreasing in $\gamma$ for given $\mu$ of 
relatively large values. Intuitively, although a pessimistic worker (low $\mu$) expects to obtain low wage offers on average, part of the downside risks are chopped off since a compensation $\tilde{c}_0$ is obtained when the offer is turned down. In this case, a higher level of uncertainty (higher $\gamma$) provides a better chance to "try the fortune" for a good offer, boosting up the reservation wage. For an optimistic (high $\mu$) but risk-averse worker, the insurance out of compensation loses power. Facing a higher level of uncertainty, the worker has an incentive to enter the labor market at an earlier stage so as to avoid downside risks. As a result,  the reservation wage goes down.

\subsection{Firm Entry}
\label{ss:fe}

Consider a firm entry problem in the style of \cite{fajgelbaum2015uncertainty}.   Each period, an investment cost $f_t>0$ is observed, where 
$\{f_t\} \iidsim h$ with finite mean. The firm then decides whether to incur this cost and enter the market to win a stochastic dividend $x_t$ 
via production, or wait and reconsider next period. The firm aims to find a decision rule that maximizes the net returns. 
	
The dividend follows $x_t = \xi_t + \epsilon_t^{x}$, 
$\left\{ \epsilon_t^{x} \right\} 
		\iidsim 
  N(0,\gamma_{x})$, 
where $\xi_t$ and $\epsilon_t^x$ are respectively a persistent and a transient component, and 
$\xi_t = \rho \xi_{t-1} + \epsilon_t^{\xi}$, 
$\{ \epsilon_t^{\xi} \} \iidsim N(0, \gamma_{\xi})$.
A public signal $y_{t+1}$ is released at the end of each period $t$, where 
$y_t = \xi_t + \epsilon_t^{y}$, 
$\left\{ \epsilon_t^{y} \right\} 
		\iidsim 
	N(0,\gamma_{y})$. 
The firm has prior belief $\xi \sim N(\mu,\gamma)$ that is Bayesian updated after observing $y'$, so the posterior satisfies
$\xi | y' \sim N(\mu',\gamma')$, with 
\begin{equation}
\label{eq:pos_uncert}
	\gamma' 
		=\left[ 
				1 / \gamma + 
				\rho^2 / (\gamma_{\xi} + \gamma_{y})
		   \right]^{-1}
	\quad \mbox{and} \quad
	\mu' 
		= \gamma' 
			\left[
					\mu / \gamma + 
					\rho y' / (\gamma_{\xi} + \gamma_{y})
			\right].
\end{equation}
The firm has utility
$u(x)= \left(1-e^{-ax} \right) / a$, where $a>0$ is the coefficient of absolute risk aversion. The value function satisfies
\begin{align*}
	v^*(f,\mu, \gamma) 
		= \max \left\{ \EE_{\mu,\gamma} [u(x)] - f, \;
								\beta \int 
												v^*(f',\mu', \gamma') 
												p(f',y'|\mu, \gamma) 
										  \diff (f',y') 
					 \right\},
\end{align*}	
where $p(f',y'|\mu, \gamma) = h(f') l(y'|\mu,\gamma)$ with 
$l(y'|\mu,\gamma) 
		= N(\rho \mu, \rho^2 \gamma + \gamma_{\xi} + \gamma_y)$. 
The exit payoff is
$r(f,\mu,\gamma) 
	:= \EE_{\mu,\gamma} [u(x)] - f 
	= \left( 1 - 
				 e^{-a \mu + a^2 (\gamma + \gamma_x) /2} 
		\right) / a - f$. 
This is a threshold state problem, with threshold state $x := f \in \RR_{++} =: \XX$ 
and environment $y := (\mu, \gamma) \in \RR \times \RR_{++} =: \YY$. The Jovanovic operator is
\begin{equation}
\label{eq:rro_fe}
	Q \psi(\mu,\gamma) 
		= \beta \int \max 
								\left\{ 
									\EE_{\mu',\gamma'} [u(x')] - f', 
									\psi(\mu',\gamma') 
								\right\} 
							 p(f', y'|\mu,\gamma) 
				 	 \diff (f',y').
\end{equation}
Let $n:=1$, $g(\mu, \gamma) := e^{ -\mu + a^2 \gamma / 2}$, $m:=1$ and 
$d:=0$. Define $\ell$ according to \eqref{eq:ell_func}.
We use $\bar{f}: \YY \rightarrow \RR$ to denote the reservation cost. 
\begin{proposition}
\label{pr:unc_traps}
 The following statements
	are true:
	\begin{enumerate}
		\item[1.] $Q$ is a contraction mapping on $(b_{\ell} \YY, \| \cdot \|_{\ell})$ 
				with unique fixed point $\psi^*$.
		\item[2.] The value function 
				  $v^* (f, \mu, \gamma) 
				  		= r(f, \mu, \gamma) 
				  			\vee
				  			\psi^* (\mu, \gamma)$, 
				  reservation cost 
				  $\bar{f} (\mu, \gamma) 
				  		= \EE_{\mu,\gamma}[u(x)] 
							- \psi^*(\mu,\gamma)$
				  and optimal policy
				  $\sigma^*(f, \mu,\gamma) 
						= \1 \{ f \leq \bar{f}(\mu,\gamma)\}$	
				  for all $(f,\mu, \gamma) \in \ZZ$. 
		\item[3.] $\psi^*$, $v^*$ and $\bar{f}$ are continuous functions.
		\item[4.] $v^*$ is decreasing in $f$, and, if $\rho \geq 0$, then $\psi^*$ and 
				  $v^*$ are increasing in $\mu$.
	\end{enumerate}
\end{proposition}

\begin{remark}
Notably, the first three claims of proposition \ref{pr:unc_traps} have no restriction on the range of $\rho$ values, the autoregression coefficient of $\{ \xi_t\}$.
\end{remark}
\begin{figure}[h]
\begin{center}
\includegraphics[scale=0.5]{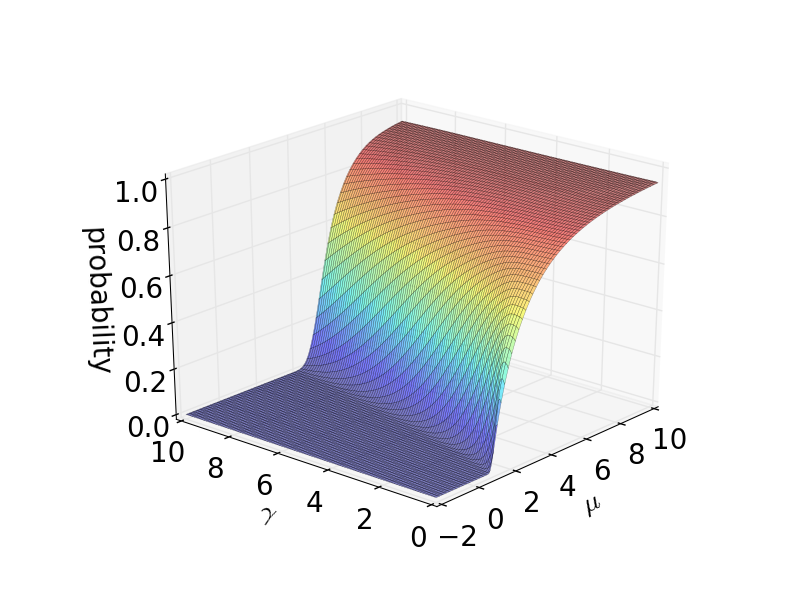}
\caption{The perceived probability of investment}
\label{fig:ppi}
\end{center}
\end{figure}
%


In simulation, we set $\beta=0.95$, $a=0.2$, $\gamma_{x}=0.1$, 
$\gamma_{y}=0.05$, and $h=LN(0, 0.01)$. Consider
$\rho = 1$, $\gamma_{\xi} = 0$, and let the grid points of 
$(\mu,\gamma)$ lie in $[-2,10] \times [10^{-4},10]$ with 
$200$ points for the $\mu$ grid and $100$ points for the $\gamma$ grid.  
The reservation cost function outside of the grid points is set to its value at the closest grid point. The integration in the operator is computed via Monte Carlo with 
1000 draws.\footnote{
		Changing the number of Monte Carlo samples, the grid range and grid 
		density produces almost the same results.} 
We plot the perceived probability of investment, i.e., 
$\PP \left\{f \leq \bar{f}(\mu, \gamma) \right\}$. 

As shown in figure \ref{fig:ppi}, the perceived probability of investment
is increasing in $\mu$ and decreasing in $\gamma$.
This parallels propositions 1 and 2 of \cite{fajgelbaum2015uncertainty}.
Intuitively, for given investment cost $f$ and variance $\gamma$, a more optimistic firm (higher $\mu$) is more likely to invest. Furthermore, higher $\gamma$ implies a higher level of uncertainty, thus a higher risk of low returns. As a result, the risk averse firm prefers to delay investment (gather more information to avoid downside risks), and will not enter the market unless the cost of investment is low enough.

\subsection{Job Search IV}
\label{ss:js_exog}

We consider another extension of \cite{mccall1970}. The setup is as in example \ref{eg:jsll}, except that the state process follows
\begin{align}
	\label{eq:w_t}
	w_{t}= & \mbox{ }\eta_{t}+\theta_{t}\xi_{t} \\
	\label{eq:theta_t}
	\ln\theta_{t}= & \mbox{ }\rho\ln\theta_{t-1}+\ln u_{t}
\end{align}
where $\rho \in [-1,1]$, 
$\{ \xi_t \} \iidsim h$ and $ \{ \eta_t \} \iidsim v$ with finite first moments, 
and $\{ u_t \} \iidsim LN(0, \gamma_u)$. Moreover, $\{ \xi_t \}$, $ \{ \eta_t \}$ 
and $\{ u_t \}$ are independent, and $\{ \theta_t \}$ is independent of 
$\{ \xi_t\}$ and $\{ \eta_t \}$. Similar settings as 
\eqref{eq:w_t}--\eqref{eq:theta_t} appear in many search-theoretic and real options studies (see e.g., \cite{gomes2001equilibrium}, \cite{low2010wage}, 
\cite{chatterjee2012maturity}, \cite{bagger2014tenure},
\cite{kellogg2014effect}).


We set $h = LN(0, \gamma_{\xi})$ and $v = LN(\mu_\eta, \gamma_{\eta})$. In this case, $\theta_t$ and $\xi_t$ are persistent and transitory components of income, respectively, and $u_t$ is treated as a shock to the persistent component. $\eta_t$ can be interpreted as social security, gifts, etc. Recall that the utility of the agent is defined by \eqref{eq:crra_utils}, $\tilde{c}_0 > 0$ is the unemployment compensation and 
$c_0 := u(\tilde{c}_0)$. The value function of the agent satisfies
\begin{align*}
	v^*(w,\theta) 
		= \max \left\{ 
							\frac{ u(w) }{ 1-\beta }, 
							 c_0 + \beta \int 
												  	 v^*(w',\theta') 
												  	 f(\theta'|\theta) 
													   h(\xi') v(\eta') 
												\diff (\theta', \xi', \eta') 						
					\right\},		
\end{align*}
where $w'=\eta' + \theta' \xi'$ and 
$f(\theta'|\theta) = LN(\rho \ln \theta, \gamma_u)$ is the density kernel of 
$\{ \theta_t \}$. 
The Jovanovic operator is
\begin{equation*}
	Q\psi(\theta) 
		= c_0 + \beta \int 
								 	 \max \left\{ 
													\frac{ u(w') }{ 1-\beta },
													\psi(\theta') 
											  \right\} 
 								  	  f(\theta' |\theta) h(\xi') v(\eta') 
 						   	    \diff (\theta', \xi', \eta').
\end{equation*}
This is another threshold state problem, with threshold state 
$x := w \in \RR_{++} =: \XX$ and environment 
$y := \theta \in \RR_{++} =: \YY$. Let $\bar{w}$ be the reservation wage.
Recall the relative risk aversion coefficient $\delta$ in \eqref{eq:crra_utils}
and the weight function $\ell$ defined by \eqref{eq:ell_func}.

\subsubsection{Case I: $\delta \geq 0$ and $\delta \neq 1$}
For $\rho \in (-1,1)$, choose $n \in \NN_0$ such that 
$\beta e^{\rho^{2n} \sigma} < 1$, where 
$\sigma := (1 - \delta)^2 \gamma_u$. 
Let $g(\theta) 
			:= \theta^{(1- \delta) \rho^n} + 
				 \theta^{-(1 - \delta) \rho^n}$ and 
$m := d:= e^{\rho^{2n} \sigma}$.
\begin{proposition}
\label{pr:js_exog}
	If $\rho \in (-1,1)$, then the following statements hold:
	
	\begin{enumerate}
	\item[1.] $Q$ is a contraction mapping on $(b_\ell \YY, \| \cdot \|_{\ell})$ with 
			unique fixed point $\psi^*$.	
	\item[2.] The value function 
			 $v^*(w,\theta)
				=  \frac{w^{1 - \delta}}{(1-\beta)(1 - \delta)} 
					 \vee \psi^*(\theta)$, 
			 reservation wage 
			 $\bar{w}(\theta) 
			 		= [(1-\beta) (1 - \delta) \psi^*(\theta)]^{\frac{1}{1 - \delta}}$, 
			 and optimal policy 
			 $\sigma^*(w, \theta) 
			 	= \1 \{ 
			 				w \geq \bar{w}(\theta) 
			 			\}$ 
			 for all $(w, \theta) \in \ZZ$.
	\item[3.] $\psi^*$ and $\bar{w}$ are continuously differentiable, and $v^*$ is 
		 	 continuous.	
	\item[4.] $v^*$ is increasing in $w$, and, if $\rho \geq 0$, then $\psi^*$, $v^*$ 
			and $\bar{w}$ are increasing in $\theta$. 
	\end{enumerate}
\end{proposition}

\begin{remark}
	If $\beta e^{(1-\delta)^2 \gamma_u / 2} < 1$, then claims 1--3 of 
	proposition \ref{pr:js_exog} remain true for $|\rho| = 1$, and claim 4 remains 
	true for $\rho = 1$.
\end{remark}

\subsubsection{Case II: $\delta=1$}

For $\rho \in (-1,1)$, choose $n \in \NN_0$ such that 
$\beta e^{\rho^{2n} \gamma_u} < 1$.
Let $g(\theta) 
			:= \theta^{\rho^n} + 
				 \theta^{-\rho^n}$ and 
$m := d:= e^{\rho^{2n} \gamma_u}$.

\begin{proposition}
\label{pr:js_exog_log}
	
	If $\rho \in (-1,1)$, then the following statements hold:
	
	\begin{enumerate}
	\item[1.] $Q$ is a contraction mapping on $(b_\ell \YY, \| \cdot \|_{\ell})$ with 
			 unique fixed point is $\psi^*$.		
	\item[2.] The value function 
			 $v^*(w,\theta)
				=  \frac{\ln w}{1-\beta} \vee \psi^*(\theta)$, 
			 reservation wage 
			 $\bar{w}(\theta) 
			 		= e^{ (1-\beta) \psi^*(\theta)}$, 
			 and optimal policy 
			 $\sigma^*(w, \theta) 
			 	= \1 \{ 
			 							w \geq \bar{w}(\theta) 
			 						  \}$ 
			 for all $(w, \theta) \in \ZZ$.	
	\item[3.] $\psi^*$ and $\bar{w}$ are continuously differentiable, and $v^*$ is 
		 	 continuous.	
	\item[4.] $v^*$ is increasing in $w$, and, if $\rho \geq 0$, then $\psi^*$, $v^*$ 
			 and $\bar{w}$ are increasing in $\theta$.
	\end{enumerate}
\end{proposition}

\begin{remark}
	If $\beta e^{ \gamma_u / 2} < 1$, then claims 1--3 of 
	proposition \ref{pr:js_exog_log} remain true for $|\rho| = 1$, and 
	claim 4 remains true for $\rho = 1$.
\end{remark}

We choose $\beta=0.95$ and $\tilde{c}_0=0.6$ as in 
\cite{ljungqvist2012recursive} (section 6.6). Further, $\mu_\eta = 0$, 
$\gamma_\eta = 10^{-6}$, $\gamma_\xi = 5 \times 10^{-4}$, 
$\gamma_u = 10^{-4}$ and $\delta = 2.5$. We consider parametric class 
problems with respect to $\rho$, where $\rho \in [0,1]$ and $\rho \in [-1,0]$ are 
treated separately, with $100$ grid points in each case. Moreover, the grid points 
of $\theta$ lie in $[10^{-4}, 10]$ with $200$ points, and the grid is scaled to be more dense when $\theta$ is smaller. The reservation wage outside the 
grid points is set to its value at the closest grid, and the integration is computed via 
Monte Carlo with 1000 draws.

\begin{figure}[h]
\centering
\begin{minipage}{.55\textwidth}
  \centering
  \includegraphics[width=1\linewidth]{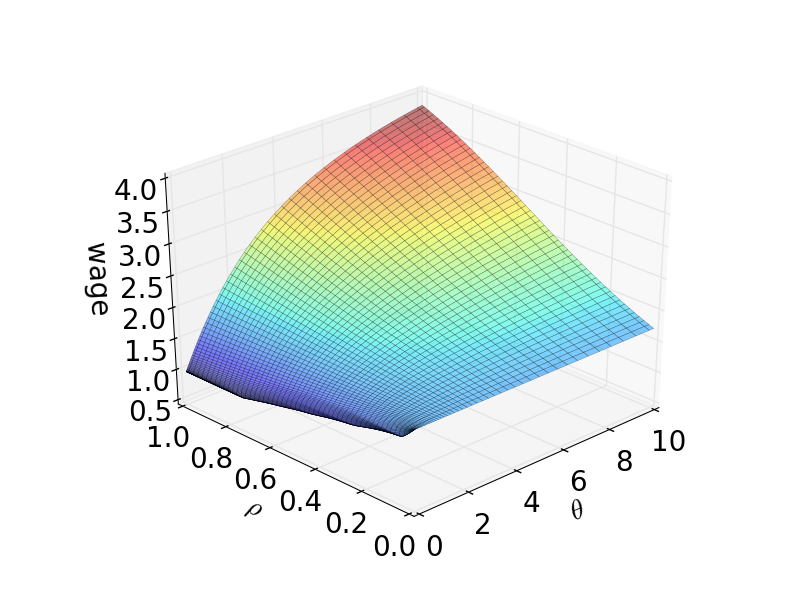}
\end{minipage}%
\begin{minipage}{.55\textwidth}
  \centering
  \includegraphics[width=1\linewidth]{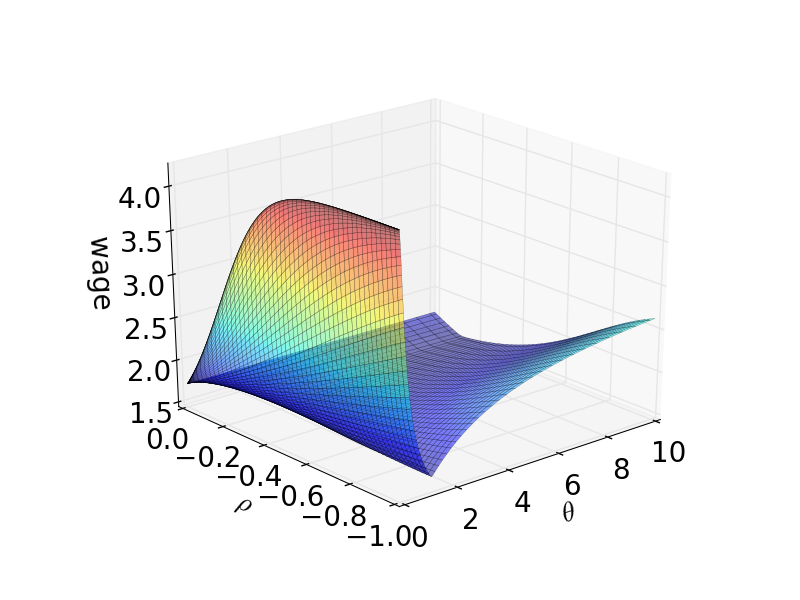}
\end{minipage}
\caption{The reservation wage}
\label{fig:res_wage_sv}
\end{figure}

When $\rho = 0$, the state process $\{\theta_t \}_{t \geq 0}$ is independent and identically distributed, in which case each realized persistent component will be forgotten in future stages. As a result, the continuation value is independent of $\theta$, yielding a reservation wage that is horizontal to the $\theta$-axis, as shown in figure \ref{fig:res_wage_sv}.

When $\rho >0$, the reservation wage is increasing in $\theta$, which 
parallels propositions \ref{pr:js_exog}--\ref{pr:js_exog_log}. Naturally, a higher $\theta$ puts the agent in a better situation, raising his desired wage level. 
Moreover, since $\rho$ measures the degree of income persistence, a higher 
$\rho$ prolongs the recovery from bad states (i.e., $\theta < 1$), and hinders the attenuation of good states (i.e., $\theta > 1$). 
As a result, the reservation wage tends to be decreasing in $\rho$ when $\theta<1$ and increasing in $\rho$ when $\theta > 1$.

When $\rho < 0$, the agent always has chance to arrive at a good state
in future stages. In this case, a very bad or a very good current state is favorable since a very bad state tends to evolve into a very good one next period, and 
a very good state tends to show up again in two periods.
If the current state is at a medium level (e.g., $\theta$ is close to $1$), however, 
the agent cannot take advantage of the countercyclical patterns. Hence, the reservation wage is decreasing in $\theta$ at the beginning and then starts to be increasing in $\theta$ after some point.

\section{Extensions}
\label{s:extension}

\subsection{Repeated Sequential Decisions}

In many economic models, the choice to stop is not permanent.  For example,
when a worker accepts a job offer, the resulting job might only be temporary
(see, e.g., \cite{rendon2006job}, \cite{ljungqvist2008two}, \cite{poschke2010regulation}, 
\cite{chatterjee2012spinoffs}, \cite{lise2012job}, \cite{moscarini2013stochastic},
\cite{bagger2014tenure}).
Another example is sovereign default (see, e.g.,  \cite{choi2003optimal}, \cite{albuquerque2004optimal}, \cite{arellano2008default},
\cite{alfaro2009optimal}, \cite{arellano2012default}, \cite{bai2012financial}, 
\cite{chatterjee2012maturity}, \cite{mendoza2012general},
\cite{hatchondo2016debt}), 
where default on international debt leads to a period of exclusion from
international financial markets.  The exclusion is not permanent, however.
With positive probability, the country exits autarky and begins borrowing from
international markets again.

To put this type of problem in a general setting, suppose that, at date $t$,
an agent is either \emph{active} or \emph{passive}.  When active, the agent
observes $Z_t$ and chooses whether to continue or exit.  Continuation results
in a current payoff $c(Z_t)$ and the agent remains active at $t+1$.  Exit
results in a current payoff $s(Z_t)$ and transition to the passive state.
From there the agent has no action available, but will return to the active
state at $t+1$ and all subsequent period with probability $\alpha$.  
\begin{assumption}
\label{a:unbdd_drift_ext1}
	There exist a $\zZ$-measurable function $g: \ZZ \rightarrow \RR_+$ and 	
	constants $n \in \NN_0$, $m,d \in \RR_+$ such that $\beta m <1$, and, 
	for all $z \in \ZZ$,
	\begin{enumerate}
		\item $\max \left\{ 
									\int |s(z')| P^n (z, \diff z'),
									\int |c(z')| P^n (z, \diff z')
				  			  \right\}
					\leq g(z)$; 
						
		\item $\int g(z') P(z, \diff z')
					\leq
			   				m g(z) + d.$
	\end{enumerate}
\end{assumption}
Let $v^*(z)$ and  $r^*(z)$ be the maximal discounted value starting at $z
\in \ZZ$ in the active and passive state respectively. One can show that, under 
assumption \ref{a:unbdd_drift_ext1}, $v^*$ and $r^*$ satisfy\footnote{	
	A formal proof of this statement is available from the authors upon request.}
\begin{equation}
    \label{eq:rst1}
    v^*(z) 
    = \max \left\{ r^*(z), 
    						c(z) + \beta \int v^*(z') P(z, \diff z') 
    			 \right\}
\end{equation}
and
\begin{equation}
    \label{eq:rst2}
    r^*(z) = s(z) + \alpha \beta \int v^*(z') P(z, \diff z') 
    			 + (1 - \alpha) \beta \int r^*(z') P(z, \diff z').
\end{equation}
With $\psi^* := c + \beta P v^*$ we can write $v^* = r^* \vee \psi^*$.
Using this notation, we can view $\psi^*$ and $r^*$ as solutions to the 
functional equations
\begin{equation}
    \label{eq:rst3}
    \psi = c + \beta P (r \vee \psi)
    \quad \text{and} \quad
    r = s + \alpha \beta P(r \vee \psi) + (1 - \alpha) \beta P r.
\end{equation}

Choose $m', d' >0$ such that $m+2m' >1$, $\beta (m+2m') <1$ and 
$d' \geq \frac{d}{m + 2m' - 1}$.
Consider the weight function $\kappa: \ZZ \rightarrow \RR_+$ defined by
\begin{equation}
\label{eq:kappa}
	\kappa (z) := 
			m' \sum_{t=0}^{n-1} 
				 \EE_z \left[ |s(Z_t)| + |c(Z_t)| \right] 
			+ g(z) + d'
\end{equation}
and the product space $(b_{\kappa} \ZZ \times b_{\kappa} \ZZ, \rho_{\kappa})$, where $\rho_{\kappa}$ is a metric on $b_{\kappa} \ZZ \times b_{\kappa} \ZZ$ defined by 
\begin{equation*}
    \rho_{\kappa} ((\psi, r), (\psi', r')) 
    = \| \psi - \psi' \|_{\kappa} \vee \| r - r'\|_{\kappa}.
\end{equation*}
With this metric, $(b_{\kappa} \ZZ \times b_{\kappa} \ZZ, \rho_{\kappa})$ inherits the completeness of
$(b_{\kappa} \ZZ, \| \cdot \|_{\kappa})$. Now define the operator $L$ on 
$b_{\kappa} \ZZ \times b_{\kappa} \ZZ$ by
\begin{equation*}
    L \begin{pmatrix}
            \psi \\
            r
        \end{pmatrix}
    = \begin{pmatrix}
        c + \beta P(r \vee \psi) \\
        s + \alpha \beta P(r \vee \psi) + (1 - \alpha) \beta P r 
        \end{pmatrix}.
\end{equation*}
\begin{theorem}
\label{thm:keythm_ext_1}
	Under assumption \ref{a:unbdd_drift_ext1}, the following statements hold:
	\begin{enumerate}
		\item[1.] $L$ is a contraction mapping on 
			$(b_{\kappa} \ZZ \times b_{\kappa} \ZZ, \rho_{\kappa})$ with modulus
			$\beta(m+2m')$.		
		\item[2.] The unique fixed point of $L$ in 
			$b_{\kappa} \ZZ \times b_{\kappa} \ZZ$ is $h^* := (\psi^*, r^*)$.
	\end{enumerate}
\end{theorem}

\subsection{Sequential Decision with More Choices}
\label{ss:more_choice}

In many economic problems, agents face multiple choices in the sequential decision process (see, e.g., \cite{crawford2005uncertainty}, \cite{cooper2007search},
\cite{vereshchagina2009risk}, \cite{low2010wage}, \cite{moscarini2013stochastic}). A standard example is on-the-job search, where an employee can choose from quitting the job market and taking the unemployment compensation, staying in the current job at a flow wage, or searching for a new job (see, e.g., \cite{jovanovic1987work},  \cite{bull1988mismatch}, \cite{gomes2001equilibrium}). A common characteristic of this type of problem is that different choices lead to different transition probabilities.

To treat this type of problem generally, suppose that in period $t$, the agent observes $Z_t$ and makes choices among $N$ alternatives. A selection of alternative $i$ results in a current payoff $r_i (Z_t)$ along with a stochastic kernel $P_i$. We assume the following.
\begin{assumption}
\label{a:unbdd_drift_ext}
	There exist a $\zZ$-measurable function $g: \ZZ \rightarrow \RR_+$ and 
	constants $m, d \in \RR_+$ such that $\beta m <1$, and, for all $z \in \ZZ$ and 
	$i, j = 1, ..., N$,
	\begin{enumerate}
		\item $\int |r_i (z')| P_j (z, \diff z') \leq g(z)$,	
		\item $\int g(z') P_i (z, \diff z') \leq m g(z) + d$.
	\end{enumerate}
\end{assumption}
Let $v^*$ be the value function and $\psi^*_i$ be the expected value of choosing alternative $i$. Under assumption \ref{a:unbdd_drift_ext}, we can show that $v^*$ and $(\psi_i^*)_{i=1}^N$ satisfy\footnote{
		A formal proof of this result is available from the authors upon request.}
\begin{equation}
\label{eq:vf_ext}
	v^*(z) = \max \{ \psi_1^* (z), ..., \psi_N^* (z)\},
\end{equation}
where
\begin{equation}
\label{eq:cvf_raw_ext}
	\psi_i^*(z) = r_i (z) 
					+ \beta \int v^*(z') P_i (z, \diff z'),
\end{equation}
for $i = 1, ..., N$. \eqref{eq:vf_ext}--\eqref{eq:cvf_raw_ext} imply that $\psi_i^*$ can be written as
\begin{equation}
\label{eq:cvf_ext}
	\psi_i^*(z) = r_i (z)
					+ \beta \int \max \{
													\psi_1^*(z'), ..., \psi_N^*(z')
												  \}
								  P_i (z, \diff z').
\end{equation}
for $i = 1, ..., N$. Define the continuation value function 
$\psi^*:= (\psi_1^*, ..., \psi_N^*)$. 

Choose $m', d' \in \RR_{++}$ such that $\beta(Nm' + m) <1$ and 
$d' \geq \frac{d}{Nm' + m -1}$. Consider the weight function 
$k: \ZZ \rightarrow \RR_+$ defined by 
\begin{equation}
	k (z) := m' \sum_{i=1}^N |r_i(z)| + g(z) + d'.
\end{equation}
One can show that the product space 
$\left(
		  \times _{i=1} ^N (b_k \ZZ), \rho_k 
\right)$ 
is a complete metric space, where $\rho_k$ is defined by 
$\rho_k (\psi, \tilde{\psi}) 
	= \vee_{i=1}^N
		\| \psi_i - \tilde{\psi}_i \|_k$
for all $\psi = (\psi_1, ..., \psi_N)$, 
$\tilde{\psi} = (\tilde{\psi}_1, ..., \tilde{\psi}_N) \in \times_{i=1}^N (b_k \ZZ)$. 
The Jovanovic operator on 
$\left( 
		  \times_{i=1}^N (b_k \ZZ), \rho_k
  \right)$
is defined by
\begin{equation}
\label{eq:cvo_ext}
	Q \psi 
	 = Q \begin{pmatrix}
	 		   \psi_1 \\
	 		   ... \\
	 		   \psi_N
	 	    \end{pmatrix}
	 = \begin{pmatrix}
	 		r_1 + \beta P_1 (\psi_1 \vee ... \vee \psi_N) \\
	 		... ... \\
	 		r_N + \beta P_N (\psi_1 \vee ... \vee \psi_N)
		 \end{pmatrix}.
\end{equation}

The next result is a simple extension of theorem \ref{thm:keythm_ext_1} and we omit its proof.

\begin{theorem}
\label{thm:keythm_ext2}
	Under assumption \ref{a:unbdd_drift_ext}, the following statements hold:
	\begin{enumerate}
		\item[1.] $Q$ is a contraction mapping on 
			  	$\left(
		 			 \times _{i=1} ^N (b_k \ZZ), \rho_k 
			  	\right)$
			  of modulus $\beta(Nm' + m)$.
		\item[2.] The unique fixed point of $Q$ in
				$\left(
		 			 \times _{i=1} ^N (b_k \ZZ), \rho_k 
			  	\right)$
			  	is $\psi^* = (\psi_1^*, ..., \psi_N^*)$. 
	\end{enumerate}
\end{theorem}

\begin{example}
Consider the on-the-job search model of \cite{bull1988mismatch}. Each 
period, an employee has three choices: quit the job market, stay in the 
current job, or search for a new job. Let $c_0$ be the value of leisure and $\theta$ 
be the worker's productivity at a given firm, with 
$(\theta_t)_{ t\geq 0} \iidsim G(\theta)$. Let $p$ be the current price. The price 
sequence $(p_t)_{t \geq 0}$ is Markov with transition probability $F(p'|p)$ and stationary distribution $F^*(p)$. It is assumed that there is 
no aggregate shock so that $F^*$ is the distribution of prices over firms. The 
current wage of the worker is $p \theta$. The value function satisfies
$v^* = \psi_1^* \vee \psi_2^* \vee \psi_3^*$, where 
\begin{equation*}
	\psi_1^* (p, \theta) 
		:= c_0 + 
			\beta \int v^* (p', \theta') \diff F^*(p') \diff G(\theta')
\end{equation*}
denotes the expected value of quitting the job, 
\begin{equation*}
	\psi_2^* (p, \theta) 
		:= p \theta + 
			\beta \int v^*(p', \theta) \diff F(p'|p)
\end{equation*}
is the expected value of staying in the current firm, and, 
\begin{equation*}
	\psi_3^* (p, \theta)
		:= p \theta +
			\beta \int v^*(p', \theta') \diff F^*(p') \diff G(\theta')
\end{equation*} 
represents the expected value of searching for a new job.
\cite{bull1988mismatch} assumes that there are compact supports $[\underline{\theta} , \bar{\theta}]$ and $[\underline{p}, \bar{p}]$ for the state processes $(\theta_t)_{t \geq 0}$ and $(p_t)_{t \geq 0}$, where $0 < \underline{\theta} < \bar{\theta} < \infty$ and $0 < \underline{p} < \bar{p} < \infty$. This assumption can be relaxed based on our theory. Let the state space be $ \ZZ := \RR_+^2$. Let $\mu_p := \int p \diff F^* (p)$ and
$\mu_{\theta} := \int \theta \diff G(\theta)$.
\begin{assumption}
\label{a:unbdd_bj}
	There exist a Borel measurable map $\tilde{g}: \RR_+ \rightarrow \RR_+$, and 
	constants	$\tilde{m}, \tilde{d} \in \RR_+$ such that $\beta \tilde{m} < 1$, 
	and, for all $p \in \RR_+$,
	\begin{enumerate}
		\item $\int p' \diff F(p'|p) \leq \tilde{g}(p)$,
		\item $\int \tilde{g}(p') \diff F(p'|p) 
						\leq \tilde{m} \tilde{g}(p) + \tilde{d}$,
		\item $\mu_p, \mu_{\theta} < \infty$ and
				  $\mu_{\tilde{g}} : = \int \tilde{g}(p) \diff F^*(p) < \infty$.
	\end{enumerate}
\end{assumption}
Let $\tilde{m} > 1$ and $\tilde{m}' \geq d / (\tilde{m}-1)$, then assumption 
\ref{a:unbdd_drift_ext} holds by letting 
$g(p, \theta) 
		:= \theta (\tilde{g}(p) + \tilde{m}' )$,
$m := \tilde{m}$ and $d := \mu_{\theta} (\mu_{\tilde{g}} + m')$.
By theorem \ref{thm:keythm_ext2}, $Q$ is a contraction mapping on 
$\left( \times_{i=1}^3 b_{\ell} \ZZ,
			\rho_{\ell} 
  \right)$. 
Obviously, assumption \ref{a:unbdd_bj} is weaker than the assumption of compact supports. 
\end{example}

\section{Conclusion}
\label{s:conclude}

A comprehensive theory of optimal timing of decisions was developed here. The theory successfully addresses a wide range of unbounded sequential decision problems that are hard to deal with via existing unbounded dynamic programming theory, including both the traditional weighted supremum norm theory and the local contraction theory. Moreover, this theory characterizes the continuation value function directly, and has obvious advantages over the traditional dynamic programming theory based on the value function and Bellman operator. First, since continuation value functions are typically smoother than value functions, this theory allows for shaper analysis of the optimal policies and more efficient computation.
Second, when there is conditional independence along the transition path (e.g., the class of threshold state problems), this theory mitigates the curse of dimensionality, a key stumbling block for numerical dynamic programming.

\section*{Appendix A}

\begin{lemma}
\label{lm:bd_vcv}
	Under assumption \ref{a:ubdd_drift_gel}, there exist $a_1, a_2 \in \RR_{+}$ 
	such that for all $z \in \ZZ$,  
	\begin{enumerate}
		\item $|v^*(z)| \leq \sum_{t=0}^{n-1} 
											\beta^t \EE_z [|r(Z_t)| + |c(Z_t)|]
										+ a_1 g(z) + a_2$.
		\item $|\psi^*(z)| 
						\leq 
						\sum_{t=1}^{n-1} 
								\beta^t \EE_z |r(Z_t)| + 
						\sum_{t=0}^{n-1}
								\beta^t \EE_z |c(Z_t)|
						+ a_1 g(z) + a_2$.
	\end{enumerate}
\end{lemma}

\begin{proof}
	Without loss of generality, we assume $m \neq 1$. By assumption 
	\ref{a:ubdd_drift_gel}, we have $\EE_z |r(Z_n)| \leq g(z)$,
	$\EE_z |c(Z_n)| \leq g(z)$ and 
	$\EE_z g(Z_1) \leq m g(z) + d$ for all $z \in \ZZ$. For all $t \geq 1$, by the 
	Markov property (see, e.g., \cite{meyn2012markov}, section 3.4.3),
	\begin{align*}
		\EE_z g(Z_t) 
		&= \EE_z 
				\left[ 
					\EE_z \left(
								 g(Z_t)| \fF_{t-1} 
							  \right) 
				\right]
		= \EE_z \left( 
							\EE_{Z_{t-1}} g(Z_1) 
						\right)		\\
		&\leq \EE_z \left( 
							m g(Z_{t-1} )+ d 
						 \right)
		= m \EE_z g(Z_{t-1}) + d.
	\end{align*}
	Induction shows that for all $t \geq 0$,
	\begin{equation}
	\label{eq:bdg}
		\EE_z g(Z_t) 
			\leq
				 m^t g(z) + \frac{1 - m^t}{1 - m} d.
	\end{equation}
	Moreover, for all $t \geq n$, apply the Markov property again shows that
	\begin{align*}
		\EE_z |r(Z_t)|
		&= \EE_z \left[ 
							\EE_z \left( 
										  |r(Z_t)| | \fF_{t-n}
									  \right) 
						\right] 
		= \EE_z \left( 
							\EE_{Z_{t-n}} |r(Z_n)| 
						\right)
		\leq \EE_z g(Z_{t-n}).			
	\end{align*}
	Based on \eqref{eq:bdg} we know that
	\begin{equation}
	\label{eq:bdr}
		\EE_z |r(Z_t)| 
		\leq 
			m^{t-n} g(z) + \frac{1-m^{t-n}}{1-m} d.
	\end{equation}
	Similarly, for all $t \geq n$, we have
	\begin{equation}
	\label{eq:bdc}
		\EE_z |c(Z_t)| \leq \EE_z g(Z_{t-n}) 
			\leq m^{t-n} g(z) + \frac{1-m^{t-n}}{1-m} d.
	\end{equation}
	
	Based on \eqref{eq:bdg} - \eqref{eq:bdc}, we can show that
	\begin{align}
	\label{eq:bdsum}
		S(z) &:= \sum_{t\geq 1} 
					  \beta^{t} \EE_z 
										\left[ 
											|r(Z_t)| + |c(Z_t)| 
										\right]		\nonumber	\\
		&\leq  
			  \sum_{t=1}^{n-1} \beta^t \EE_z [|r(Z_t)| + |c(Z_t)|]  		
			  + \frac{2 \beta^n}{1-\beta m} g(z) 
			  + \frac{2 \beta^{n+1} d}{(1-\beta m)(1-\beta)}. 
	\end{align}
	Since $|v^*| \leq |r| + |c| + S$ and $|\psi^*| \leq |c| + S$,
	the two claims hold by letting $a_1 := \frac{2 \beta^n}{1-\beta m}$
	and $a_2 := \frac{2 \beta^{n+1} d}{(1-\beta m)(1-\beta)}$. 
	This concludes the proof.
\end{proof}

Denote $(X,\mathcal{X})$ as a measurable space and $(Y,\mathcal{Y},u)$ as a measure space.

\begin{lemma}
\label{lm:cont}	
	Let $p: Y \times X \rightarrow \RR$ be a measurable map that is  
	continuous in $x$. If there exists a measurable map 
	$q: Y \times X \rightarrow \RR_+$ that is continuous in $x$ with 
	$q(y,x) \geq |p(y,x)|$ for all $(y,x) \in Y \times X$, and that 
	$x \mapsto \int q(y,x) u(\diff y)$ is continuous, then the 
	mapping $x \mapsto \int p(y,x) u(\diff y)$ is continuous.
\end{lemma}

\begin{proof}
	 Since $q(y,x) \geq |p(y,x)|$ for all $(y, x) \in Y \times X$, we know that $(y,x) 
	 \mapsto q(y,x) \pm p(y,x)$ are nonnegative measurable functions. Let $(x_n)$ be 
	 a sequence of $X$ with $x_n \rightarrow x$. By Fatou's	 lemma, we have
	\begin{align*}
		\int \liminf_{n \rightarrow \infty} [q(y,x_n) \pm p(y,x_n)] u(\diff y)	
		\leq \liminf_{n\rightarrow \infty} \int [q(y,x_n) \pm p(y,x_n)] u(\diff y).	
	\end{align*}
	From the given assumptions we know that $\underset{n\rightarrow \infty}{\lim} 
	\int q(y,x_n) u(\diff y) = q(y, x)$. Combine this result with the above inequality, 
	we have
	\begin{align*}
		\pm \int p(y,x) u(\diff y) \leq \liminf_{n \rightarrow \infty} \left( \pm \int 
														p(y,x_n)  u(\diff y) \right),
	\end{align*}
	where we have used the fact that for any two given sequences $(a_n)_{n\geq 0}$ 
	and $(b_n)_{n \geq 0}$ of $\RR$ with $\underset{n\rightarrow \infty}
	{\lim} a_n$ exists, we have: $\underset{n \rightarrow \infty}{\liminf}(a_n + b_n) = 
	\underset{n \rightarrow \infty}{\liminf} \mbox{ } a_n + \underset{n \rightarrow 
	\infty}{\liminf} \mbox{ } b_n$. So 
	\begin{align*}
	\label{eq:limsup}
		\limsup_{n \rightarrow \infty} \int p(y,x_n) u(\diff y)  \leq \int p(y,x) u(\diff y)   
		\leq \liminf_{n \rightarrow \infty} \int p(y,x_n) u(\diff y).   				
	\end{align*}
	Therefore, the mapping $x \mapsto \int p(y,x) u(\diff y)$ is continuous.
\end{proof}

\section*{Appendix B : Main Proofs}
\label{s:pr_js}

\subsection{Proof of Section \ref{s:opt_results} Results.}
\label{ss:cvals}

In this section, we prove examples \ref{eg:jsll}--\ref{eg:js_adap}. Note that in  example \ref{eg:jsll}, $P$ has a density 
representation $f(z'|z) = N(\rho z + b, \sigma^2)$.

\begin{proof}[Proof of example \ref{eg:jsll}]
	\textit{Case I:} $\delta \geq 0$ and $\delta \neq 1$. In this case, the exit payoff 
	is	$r(z) := e^{(1 - \delta) z} /((1- \beta) (1 - \delta))$.	
	Since $\int e^{(1- \delta) z'} f(z'|z) \diff z'
		= a_1 e^{\rho (1 - \delta) z}$
	for some constant $a_1 > 0$, induction shows that 
	\begin{equation}
	\label{eq:e_ntimes}
		\int e^{(1- \delta)z'} P^t (z, \diff z')
		= a_t e^{\rho^t (1 - \delta) z}
		\leq a_t 
				\left(
					 e^{\rho^t (1 - \delta) z} +
					 e^{\rho^t (\delta - 1) z}
				\right)
	\end{equation} 
	for some constant $a_t >0$ and all $t \in \NN$. Recall the definition of $\xi$
	in example \ref{eg:jsll}. Let $n \in \NN$ such that $\beta e^{|\rho^n| \xi} < 1$,
	$g(z) :=  e^{\rho^n (1 - \delta) z} +
				  e^{\rho^n (\delta - 1) z}$
	and $m := d:= e^{|\rho^n| \xi}$. By remark \ref{rm:suff_key_assu}, it remains to 
	show that	$g$ satisfies the geometric drift condition \eqref{eq:drift}.
	Let $\xi_1 := (1 - \delta) b$ and $\xi_2 := (1-\delta)^2 \sigma^2 /2$, then
	$\xi_1 + \xi_2 \leq \xi$, and we have\footnote{
		To obtain the second inequality of \eqref{eq:e_g}, note that either 
		$\rho^{n+1}(1-\delta) z \leq 0$
		or $\rho^{n+1} (\delta - 1) z \leq 0$. Assume without loss of generality that 
		the former holds, then $e^{\rho^{n+1} (1 - \delta) z} \leq 1$ and 
		$0 \leq \rho^{n+1} (\delta - 1) z
			 \leq \rho^n (\delta - 1) z \vee \rho^n (1 - \delta) z$.
		The latter implies that 
		$e^{\rho^{n+1} (\delta -1) z}
			\leq e^{\rho^n (1 - \delta) z} + e^{\rho^n (\delta - 1) z}$. Combine this 
		with 	$e^{\rho^{n+1} (1 - \delta) z} \leq 1$ yields the second inequality of 
		\eqref{eq:e_g}.}
	\begin{align}
	\label{eq:e_g}
		\int g(z') f(z'|z) \diff z' 
		&= e^{\rho^{n+1} (1 - \delta) z} 
			  e^{ 
			  		\rho^n \xi_1 + 
			        \rho^{2n} \xi_2
			  	  }		
			 + e^{\rho^{n+1} (\delta - 1)z} 
			 	 e^ {
			 	   	    -\rho^n \xi_1 + 
			  			\rho^{2n} \xi_2
			  		 	}		\nonumber \\
		& \leq \left( e^{\rho^{n+1} (1 - \delta) z}
							+ e^{\rho^{n+1} (\delta - 1) z}
					\right)
					e^{ |\rho^n| \xi} 			\nonumber \\		
		& \leq \left( e^{\rho^n (1 - \delta)z}
							+ e^{\rho^n (\delta - 1)z} + 1
					\right)
					e^{|\rho^n| \xi}			
		= m g(z) + d.
	\end{align}
	Since $\beta m = \beta e^{|\rho^n| \xi}<1$, $g$ satisfies the geometric drift 
	property, and assumption \ref{a:ubdd_drift_gel} holds. 
	
	\begin{remark}
	In fact, if $\rho \in [0,1)$,
	by \eqref{eq:e_ntimes}, we can also let $g(z) := e^{\rho^n (1 - \delta) z}$, then 
	\begin{align}
	\label{eq:eg_rhopos}
		\int g(z') f(z'|z) \diff z'
		&= e^{\rho^{n+1} (1 - \delta) z} 
			  e^{\rho^n \xi_1 + \rho^{2n} \xi_2} 
		\leq \left( e^{\rho^n (1 - \delta) z} + 1 \right)
				e^{\rho^n (\xi_1 + \xi_2)} 		\nonumber \\
		& \leq \left( e^{\rho^n (1 - \delta) z} + 1 \right)
					e^{\rho^n \xi}	= m g(z) + d. 	
	\end{align}
	In this way, we have a simpler $g$ with geometric drift property.
	\end{remark}
	
	\begin{remark}
	If $\rho \in [-1,1]$ and $\beta e^{\xi} < 1$, by 
	\eqref{eq:e_ntimes}--\eqref{eq:e_g}
	one can show that assumption \ref{a:ubdd_drift_gel} holds with 
	$n := 0$, $g(z) := e^{(1 - \delta) z} + e^{(\delta - 1) z}$ and 
	$m := d := e^{ \xi}$. In fact, if $\rho \in [0,1]$ and 
	$\beta e^{\xi_1 + \xi_2} <1$, by \eqref{eq:eg_rhopos} one can show that 
	assumption \ref{a:ubdd_drift_gel} holds
	with $n := 0$, $g(z) := e^{(1 - \delta) z}$ and $m := d:= e^{\xi_1 + \xi_2}$.
	In this way, we can treat nonstationary state process at the cost of some
	additional restrictions on parameter values.
	\end{remark}
	
	\textit{Case II:} $\delta = 1$. In this case, the exit payoff is 
	$r(z) = z / (1 - \beta)$. Let $n:=0$, $g(z) := |z|$, $m :=|\rho|$ and 
	$d := \sigma + |b|$.
	Since $(\epsilon_t)_{t \geq 0} \iidsim N(0, \sigma^2)$, 
	by Jensen's inequality,
	\begin{align*}
		\int g(z') f(z'|z) \diff z' 
		&= \EE_z |Z_1| \leq |\rho| |z| + |b| + \EE |\epsilon_1| 		\\
		&\leq |\rho| |z| + |b| + \sqrt{\EE (\epsilon_1^2)} 
		= |\rho| |z| + |b| + \sigma 
		= m g(z) + d.
	\end{align*}
	Since $\beta m = \beta |\rho| < 1$, assumption \ref{a:ubdd_drift_gel} holds. 
	This concludes the proof.
\end{proof}	
	
\begin{proof}[Proof of example \ref{eg:js_adap}]
\textit{Case I:} $\delta \geq 0$ and $\delta \neq 1$.
Recall the definition of $g$. We have
\begin{equation}
\label{eq:js_adap_er}
	\int 
		  w'^{1 - \delta} f(w' | \mu, \gamma) 
	\diff w' 
	= e^{ (1 - \delta)^2 \gamma_{\epsilon} / 2 }
			\cdot
		e^{ (1- \delta) \mu + (1 - \delta)^2 \gamma / 2 }
	= e^{ (1 - \delta)^2 \gamma_{\epsilon} / 2 }
		g(\mu, \gamma). 
\end{equation}
By remark \ref{rm:suff_key_assu}, it remains to verify the geometric drift condition 
\eqref{eq:drift}. This follows from \eqref{eq:pos_js_adap}--\eqref{eq:Ph}. Indeed, one can show that
\begin{equation}
\label{eq:js_adap_eg}
	\EE_{\mu, \gamma} g(\mu', \gamma')
	:=
	\int 
		  g(\mu', \gamma') f(w'| \mu, \gamma) 
	\diff w'
	= g(\mu, \gamma)
\end{equation}
\textit{Case II:} $\delta = 1$.
Since $|\ln a| \leq a + 1 / a$, $\forall a > 0$, we have: $|u(w')| \leq w' + w'^{-1}$, and
\begin{equation}
\label{eq:js_adap_eu}
	\int 
		  |u(w')| f(w'|\mu, \gamma) 
	\diff w'
	\leq e^{\mu + (\gamma + \gamma_\epsilon) / 2} +
			e^{-\mu + (\gamma + \gamma_\epsilon) / 2}
	= e^{\gamma_\epsilon} g(\mu, \gamma).
\end{equation}
Similarly as \textit{case I}, one can show that
$\EE_{\mu, \gamma} g(\mu', \gamma') = g(\mu, \gamma)$. 

Hence, assumption \ref{a:ubdd_drift_gel} holds in both cases. This concludes the proof.
\end{proof}

\begin{proof}[Proof of theorem~\ref{t:bk}]

	To prove claim 1, based on the weighted contraction mapping 
	theorem (see, e.g., \cite{boud1990recursive}, section 3), it suffices to verify: 
	(a) $Q$ is monotone, 
	i.e., $Q\psi	\leq Q\phi$ if $\psi, \phi \in b_{\ell} \ZZ$ and $\psi \leq \phi$; 
	(b) $Q0 \in b_{\ell} \ZZ$ and $Q \psi$ is $\zZ$-measurable for all 
	$\psi \in b_{\ell} \ZZ$; and 
	(c) $Q(\psi+a\ell)\leq Q\psi+a \beta (m + 2m') \ell$ for all $a\in\RR_{+}$ and 
	$\psi \in b_{\ell} \ZZ$. 
	Obviously, condition (a) holds. 
	By \eqref{eq:defq}--\eqref{eq:ell_func}, we have
	\begin{align*}
		\frac{|(Q0)(z)|}{\ell(z)} 
		\leq 
			  \frac{|c(z)|}{\ell(z)} 
			  + \beta \int 
			  					\frac{|r(z')|}{\ell(z)}
				   			P(z,dz') 
		\leq 
			  (1 + \beta) / m' < \infty
	\end{align*}
	for all $z \in \ZZ$, so $\| Q0 \|_{\ell} < \infty$. The measurability of  
	$Q \psi$ follows immediately from our primitive assumptions. Hence, condition 
	(b) holds. 	By the Markov property (see, e.g., \cite{meyn2012markov}, section 
	3.4.3), we have
	\begin{align*}
		\int \EE_{z'} |r(Z_t)| P(z, \diff z')
			= \EE_z |r(Z_{t+1})|
				\; \mbox{ and } 
		\int \EE_{z'} |c(Z_t)| P(z, \diff z')
			= \EE_z |c(Z_{t+1})|.
	\end{align*}
	Let $h(z) := 
					\sum_{t=1}^{n-1} \EE_z |r(Z_{t})| +
					\sum_{t=0}^{n-1} \EE_z |c(Z_{t})|$,
	then we have
	\begin{align}
	\label{eq:bd_h}
		& \int 
				h(z')
			P(z, \diff z')
		= \sum_{t=2}^{n} \EE_z |r(Z_{t})| +
			\sum_{t=1}^{n} \EE_z |c(Z_{t})|.
	\end{align}
	By the assumptions on $m'$ and $d'$, we have $m + 2m' > 1$ and 
	$(d + d') / (m + 2m') \leq d'$.
	Assumption \ref{a:ubdd_drift_gel} and \eqref{eq:bd_h} then imply that
	\begin{align*}
	\int \ell(z') P(z,dz')
	& = 
		m' \left(
					\sum_{t=2}^{n} \EE_z |r(Z_{t})| +
					\sum_{t=1}^{n} \EE_z |c(Z_{t})|
			\right)
			+ \int g(z') P(z,dz') + d'   		\\
	& \leq 
		m' \left(
					\sum_{t=2}^{n-1} \EE_z |r(Z_{t})| +
					\sum_{t=1}^{n-1} \EE_z |c(Z_{t})|
			\right)
			+ (m + 2m') g(z) + d + d'  			\\
	& \leq 
		(m + 2m') 
		\left(
			\frac{m'}{m+ 2m'} h(z) +
			g(z) + \frac{d + d'}{m + 2m'}
		\right)			
	\leq (m + 2m') \ell(z).
	\end{align*}		
	Hence, for all $\psi \in b_{\ell} \ZZ$, $a \in \RR_+$ and $z \in \ZZ$, we have
	\begin{align*}
	Q(\psi+a\ell)(z) 
	& = c(z) + \beta \int \max \left\{ r(z'), \psi(z') + a\ell(z')\right\} P(z,dz')  \\
	& \leq c(z) + \beta \int \max \left\{ r(z'),\psi(z')\right\} P(z,dz') + a \beta \int 
			\ell(z') P(z,dz')  \\
	& \leq Q\psi(z) + a \beta (m + 2m') \ell(z).
	\end{align*}
	So condition (c) holds. Claim 1 is verified.

	Regarding claim 2, substituting $v^* = r \vee \psi^*$ into \eqref{eq:cvf} 
	we get
	\begin{equation*}
		\psi^*(z) 
		= c(z) + 
			\beta \int 
						  \max \{ r(z'), \psi^*(z') \} 
					  P(z, \diff z').
	\end{equation*}
	This implies that $\psi^*$ is a fixed point of $Q$. Moreover, from lemma 
	\ref{lm:bd_vcv} we know that $\psi^* \in b_\ell \ZZ$. Hence, $\psi^*$ must 
	coincide with the unique fixed point of $Q$ under $b_\ell \ZZ$.
	
	Finally, by theorem 1.11 of \cite{peskir2006}, we can show 
	that $\tilde{\tau}:= \inf \{t \geq 0: v^*(Z_t) = r(Z_t)\}$ is an optimal stopping 
	time. Claim 3 then follows from the definition of the optimal policy and the 
	fact that $v^* = r \vee \psi^*$. 
\end{proof}

\subsection{Proof of Section \ref{s:properties_cv} Results.}

\begin{proof}[Proof of proposition~ \ref{pr:cont}]
	Let $b_{\ell} c \ZZ$ be the set of continuous functions in $b_{\ell} \ZZ$. Since 
	$\ell$ is continuous by assumption, $b_{\ell}c \ZZ$ is a closed subset of 
	$b_{\ell} \ZZ$ (see e.g., \cite{boud1990recursive}, section 3). To show the 
	continuity of $\psi^*$, it suffices to verify that 
	$Q (b_{\ell} c \ZZ) \subset b_{\ell} c \ZZ$ (see, e.g., \cite{stokey1989}, 
	corollary 1 of theorem 3.2).
	For all $\psi \in b_{\ell} c \ZZ$, there exists a constant $G \in \RR_+$ such that 
	$|\max \{ r(z), \psi(z) \}| \leq |r(z)| + G \ell(z)$. In particular,
	\begin{equation*}
		z \mapsto |r(z)| + G \ell(z) \pm \max \{ r(z), \psi(z)\}
	\end{equation*}
	are nonnegative and continuous. Let $h(z) := \max \{ r(z), \psi (z) \}$. Based on 
	the generalized Fatou's lemma of \cite{feinberg2014fatou} (theorem 1.1), we can 
	show that for all sequence $(z_m)_{m \geq 0}$ of $\ZZ$ such that 
	$z_m \rightarrow z \in \ZZ$, we have
	\begin{align*}
		\int 
		\left( |r(z')| + G \ell(z') \pm h(z') \right) 
		P(z, \diff z') 
		\leq 	
			  \liminf_{m \rightarrow \infty} 
			  \int 
			  \left( |r(z')| + G \ell(z') \pm h(z') \right) 
			  P(z_m, \diff z').
	\end{align*}	 
	Since assumptions \ref{a:payoff_cont}--\ref{a:l_cont} imply that 
	\begin{align*}
		\lim_{m \rightarrow \infty} 
			\int \left( |r(z')| + G \ell(z') \right)
			 P(z_m, \diff z')
		= \int \left( |r(z')| + G \ell(z') \right)
			 P(z, \diff z'),
	\end{align*}
	we have
	\begin{align*}
		\pm \int 
					h(z')
				P(z, \diff z')
		\leq 
				\liminf_{m \rightarrow \infty}
				\left(				
				\pm \int
							h(z')
						P(z_m, \diff z')
				\right),
	\end{align*}
	where we have used the fact that for given sequences $(a_m)_{m \geq 0}$ and 
	$(b_m)_{m \geq 0}$ of $\RR$ with $\underset{m\rightarrow \infty}{\lim} a_m$ 
	exists, we have: 
	$\underset{m \rightarrow \infty}{\liminf} (a_m + b_m)
	= 
		\underset{m \rightarrow \infty}{\lim} a_m 
		+ \underset{m \rightarrow \infty}{\liminf} \; b_m$. 
	Hence, 
	\begin{align}
	\label{eq:fatou_eq}
		\limsup_{m \rightarrow \infty} 
		\int
			  h(z') 
		P(z_m, \diff z')
		\leq
			\int
				  h(z')
			P(z, \diff z')
		\leq 
			\liminf_{m \rightarrow \infty} 
			\int
				  h(z') 
			P(z_m, \diff z'),
	\end{align}
	i.e., $z \mapsto \int h(z') P(z, \diff z')$ is continuous. Since $c$ is continuous by 
	assumption, $Q \psi \in b_{\ell} c \ZZ$. Hence, 
	$Q (b_{\ell} c \ZZ) \subset b_{\ell} c \ZZ$ and $\psi^*$ is 
	continuous. The continuity of $v^*$ follows from the continuity of $\psi^*$ and 
	$r$ and the fact that $v^*=r \vee \psi^*$.
\end{proof}

Recall $\mu$ and $\mu_i$ defined in the beginning of section \ref{ss:diff}.
The next lemma holds.

\begin{lemma}
\label{lm:diff_gel}
Suppose assumption \ref{a:ubdd_drift_gel} holds, and, 
for $i = 1, ..., m$ and $j = 1, 2$
\begin{enumerate}
	\item $P$ has a density representation $f$ such that
			$D_i f(z'|z)$ exists, $\forall (z, z') \in \interior(\ZZ) \times \ZZ$.
			
	\item For all $z_0 \in \interior(\ZZ)$, there exists $\delta>0$, such that 
			\begin{equation*}
				\int |k_j (z')|
						\underset{z^i \in \bar{B}_{\delta}(z_0^i)}{\sup} 
						\left| D_i f(z'|z) \right| 
				\diff z' 
				< \infty		
				\qquad{(z^{-i} = z_0^{-i})}.
			\end{equation*}
\end{enumerate}
Then: $D_i \mu (z) = \mu_i (z)$ for all $z \in \interior (\ZZ)$ and $i=1,...,m$.
\end{lemma}

\begin{proof}[Proof of lemma~ \ref{lm:diff_gel}]
	For all $z_0 \in \interior(\ZZ)$, let $\{ z_n\}$ be an arbitrary sequence of 
	$\interior (\ZZ)$ such that $z_n^i \rightarrow z_0^i$, $z_n^i \neq z_0^i$ and 
	$z_n^{-i} = z_0^{-i}$ for all $n \in \NN$.
	For the $\delta>0$ given by (2), there exists $N\in \NN$ such that 
	$z_n^i \in \bar{B}_{\delta}(z_0^i)$ for all $n\geq N$. Holding $z^{-i} = z_0^{-i}$, 
	by the mean value theorem, there exists 
	$\xi^i (z',z_n,z_0) \in \bar{B}_{\delta}(z_0^i)$ such that
	\begin{equation*}
		|\triangle^i (z',z_n,z_0)| 
		:= \left| 
					\frac{f(z'|z_n) - f(z'|z_0)}{z_n^i-z_0^i}
			\right| 
		= \left| 
					D_i f(z'|z)|_{z^i = \xi^i (z',z_n, z_0)} 
			\right|
		\leq 
			   \underset{z^i \in \bar{B}_{\delta}(z_0^i)}{\sup} 
			   \left| 
						D_i f(z'|z)
			   \right|
	\end{equation*}
	Since in addition $|\psi^*| \leq G \ell$ for some $G \in \RR_+$, we have: for all 
	$n\geq N$, 
	\begin{enumerate}
		\item[(a)] 
					$\left| \max \{ r(z'), \psi^*(z')\} 
							 \triangle^i (z',z_n, z_0) \right| 
					\leq 
					   	   \left( |r(z')|+ G \ell(z') \right) 
						   \underset{z^i \in \bar{B}_{\delta}(z_0^i)}{\sup}
						   										\left| D_i f(z'|z)\right|$,
		
		\item[(b)] 
					$\int \left( |r(z')|+ G \ell(z') \right) 
							\underset{z^i \in \bar{B}_{\delta}(z_0^i)}{\sup} 
																   \left| D_i f(z'|z)\right| 
					  dz' < \infty$, and 
		
		\item[(c)] 
					$\max \{ r(z'), \psi^*(z')\} \triangle^i (z',z_n,z_0) 
					\rightarrow 
								\max \{ r(z'), \psi^*(z')\} D_i f(z'|z_0)$ as $n\rightarrow \infty$,
	\end{enumerate}
	where (b) follows from condition (2). By the dominated convergence theorem, 
	\begin{align*}
		\frac{\mu(z_n)-\mu(z_0)}{z_n^i - z_0^i} 
		&= \int \max \{ r(z'), \psi^*(z')\} \triangle^i (z',z_n,z_0) \diff z'    \\
		&\rightarrow \int \max \{ r(z'), \psi^*(z')\} D_i f(z'|z_0) \diff z' 
		= \mu_i (z_0).
	\end{align*}
	Hence, $D_i \mu(z_0) = \mu_i(z_0)$, as was to be shown.
\end{proof}

\subsection{Proof of Section \ref{s:application} Results}

\begin{proof}[Proof of proposition \ref{pr:js_ls}]
	Assumption \ref{a:ubdd_drift_gel} holds due to bounded payoffs.
	By theorem \ref{t:bk}, claim 1 holds. 
	Let $\XX = [w_l, w_h] \subset \RR_+$. Since $c_0 \in \XX$, 
	$v^*(w, \pi) \in [w_l / (1 - \beta), w_h / (1 - \beta)]$, then
	$c_0 + \beta \int v^* (w', \pi') h_{\pi} (w') \diff w' 
			\in [w_l / (1 - \beta), w_h / (1 - \beta)]$.
	By the intermediate value theorem, assumption \ref{a:opt_pol} holds. 
	By theorem \ref{t:bk} and \eqref{eq:res_rule_pol}, claim 2 holds.
	$P$ satisfies the Feller property by lemma \ref{lm:cont}. Since payoff functions are 
	continuous, the continuity of $\psi^*$ and $v^*$ follows from proposition 
	\ref{pr:cont}
	(or remark \ref{rm:bdd_cont}). The continuity of $\bar{w}$ follows from 
	proposition \ref{pr:pol_cont}. Claim 3 is verified. 
	%
\end{proof}

\begin{proof}[Proof of proposition \ref{pr:unc_traps}]
	The exit payoff satisfies
	\begin{equation}
	\label{eq:bdr_uncert}
		\left| 
				r(f', \mu', \gamma') 
		\right| 
			\leq 1/a + 
					\left( 
							e^{a^2 \gamma_x / 2} / a
					\right) \cdot
					e^{-a \mu' + a^2 \gamma' / 2}
					+ f'.
	\end{equation}
	Using \eqref{eq:pos_uncert}, we can show that
	\begin{equation}
	\label{eq:exp_mugam_uncert}
		\int 
			  e^{-a \mu' + a^2 \gamma' / 2} 
		P(z, \diff z')
		= \int 
				  e^{-a \mu' + a^2 \gamma' / 2} 
				  l(y' | \mu, \gamma)
			\diff y'
		= e^{-a \mu + a^2 \gamma / 2}.
	\end{equation}
	Let $\mu_f$ denote the mean of $\{ f_t\}$. Combine 
	\eqref{eq:bdr_uncert}--\eqref{eq:exp_mugam_uncert}, we have 
	\begin{equation}
	\label{eq:int_r}
		\int 
			  \left| r(f',\mu',\gamma') \right|
		P(z, \diff z')
		\leq 
			  \left(1/a + \mu_f \right) + 
			  \left( 
					e^{a^2 \gamma_x / 2} / a
		  	  \right) \cdot
			  g(\mu, \gamma).
	\end{equation}
	Notice that \eqref{eq:exp_mugam_uncert} is equivalent to 
	\begin{equation}
	\label{eq:int_g}
		\int g(\mu', \gamma') P(z, \diff z')
		= g(\mu, \gamma).
	\end{equation} 
	Hence, assumption \ref{a:ubdd_drift_gel} holds with $n :=1$, $m :=1$ and $d:=0$. 
	The intermediate value theorem shows that assumption \ref{a:opt_pol} holds. By theorem \ref{t:bk} and the analysis of section \ref{s:opt_pol}, claims 1--2 hold.
	
	For all bounded continuous function $\tilde{f}: \ZZ \rightarrow \RR$, we have
	\begin{equation*}
		\int \tilde{f}(z') P(z, \diff z') 
		= \int \tilde{f} (f',\mu', \gamma') 		   
				   h(f') l(y' | \mu, \gamma)
			\diff (f', y').
	\end{equation*}
	By \eqref{eq:pos_uncert} and lemma \ref{lm:cont}, this function is bounded and continuous in $(\mu, \gamma)$. Hence,  
	assumption \ref{a:feller} holds.
	The exit payoff $r$ is continuous. By \eqref{eq:pos_uncert}, both sides of \eqref{eq:bdr_uncert} are continuous in $(\mu, \gamma)$.
	By \eqref{eq:exp_mugam_uncert}--\eqref{eq:int_r}, the conditional expectation of the right side of \eqref{eq:bdr_uncert} is continuous in 
	$(\mu, \gamma)$. Lemma \ref{lm:cont} then implies that 
	$(\mu, \gamma) \mapsto \EE_{\mu, \gamma} |r(Z_1)|$ is 
	continuous. Now we have shown that assumption \ref{a:payoff_cont} holds. 
	Assumption \ref{a:l_cont} holds since $g$ is continuous and \eqref{eq:int_g} holds.
	Proposition \ref{pr:cont} then implies that $\psi^*$ and $v^*$ are continuous. 
	By proposition \ref{pr:pol_cont}, $\bar{f}$ is continuous. Claim 3 is verified.
	
	Since $r$ is decreasing in $f$, $v^* = r \vee \psi^*$ is decreasing in $f$.
	If $\rho \geq 0$, then $l$ is stochastically increasing in $\mu$. 
	By \eqref{eq:pos_uncert},  $P(r \vee \psi)$ is increasing in 
	$\mu$ for all $\psi \in b_{\ell} \YY$ that is increasing in $\mu$, i.e., assumption
	\ref{a:mono_map} holds. Since $r$ is increasing in $\mu$, by proposition 
	\ref{pr:mono}, $\psi^*$ and $v^*$ are increasing in $\mu$. Hence, claim 4 
	holds. 
\end{proof}

\begin{proof}[Proof of proposition \ref{pr:js_exog}]
	\textit{Proof of claim 1.} 
	Since
	\begin{equation}
	\label{eq:w'_bd}
		w'^{1 - \delta} 
			= \left( \eta' + \theta' \xi' \right)^{1 - \delta}
			\leq 2 \left( 
							\eta'^{1 - \delta} +
							\theta'^{1 - \delta} \xi'^{1 - \delta}
					   \right),
	\end{equation}
	we have
	\begin{align}
	\label{eq:ew'}
		\int w'^{1 - \delta} P(z, \diff z')
		&\leq 2 \int 
						\eta'^{1 - \delta} v(\eta') 
				  	 \diff \eta' 
			  	 + 2 \int
			   			    \xi'^{1 - \delta} h(\xi') 
			   		   \diff \xi'	\cdot
					   \int
				  		    \theta'^{1 - \delta} f(\theta' | \theta)
				  	   \diff \theta'			\nonumber	\\
		&= 2 e^{ (1 - \delta) \mu_{\eta} + 
						(1 - \delta)^2 \gamma_{\eta} / 2 } 
				 + 
			  2 e^{ (1 - \delta)^2 (\gamma_{\xi} + \gamma_u) / 2 }
			  		\cdot
			  \theta^{(1 - \delta) \rho}.
	\end{align}
	Induction shows that 
	\begin{align}
	\label{eq:ew'_ttimes}
		\int 
			  w'^{1 - \delta} 
		P^t (z, \diff z') 
		\leq 
			a_1^{(t)} + a_2^{(t)} \theta^{(1 - \delta) \rho^t}
		\leq 
			a_1^{(t)} + a_2^{(t)} \left(
												\theta^{(1 - \delta) \rho^t}+
												\theta^{-(1 - \delta) \rho^t}
											\right)
	\end{align}
	for some $a_1^{(t)}$, $a_2^{(t)} > 0$ and all $t \in \NN$. Define
	$g$ as in the assumption, then
	\begin{align}
	\label{eq:eg_js_exog}
		\int 
			  g(\theta') 
			  f(\theta' | \theta) 
		\diff \theta'
		&= \left(
			 e^{
			 	  (1 - \delta) \rho^{n+1} \ln \theta	 	  
			 	 }   +  
			 e^{
			 	  -(1 - \delta) \rho^{n+1} \ln \theta
			 	 }
			 \right) 
			 e^{ (1 - \delta)^2 \rho^{2n} \gamma_u / 2 }		\\
		& \leq 	 
			 \left(
			 	e^{(1 - \delta) \rho^n \ln \theta}   +  
				e^{ -(1 - \delta) \rho^n \ln \theta} + 1
			 \right)
			 e^{ (1 - \delta)^2 \rho^{2n} \gamma_u / 2 }	\nonumber		\\
		& \leq 
			 \left( g(\theta) + 1 \right) e^{\rho^{2n} \sigma}
		= m g(\theta) + d. 		\nonumber
	\end{align}
	Hence, assumption \ref{a:ubdd_drift_gel} holds. Claim 1 then follows 
	from theorem \ref{t:bk}.
	
	\textit{Proof of claim 2.} 
	Assumption \ref{a:opt_pol} holds by the intermediate value theorem. Claim 2 
	then follows from theorem \ref{t:bk}, assumption \ref{a:opt_pol} and 
	\eqref{eq:res_rule_pol}.
	
	\textit{Proof of claim 3.} Note that the stochastic kernel $P$ has a density
	representation in the sense that for all $z \in \ZZ$ and $B \in \zZ$,
	\begin{equation*}
		P(z, B) 
			= \int \1 \left\{ 
			  					(\eta' + \xi' \theta', \theta') \in B 
			  			   \right\}
			  			v (\eta') h(\xi') f(\theta'|\theta)
			  	\diff (\eta', \xi', \theta').
	\end{equation*}
	%
	Moreover, it is straightforward (though tedious) to show that 
	$\theta \mapsto f(\theta' | \theta)$ is twice differentiable for all $\theta'$, 
	that
	$(\theta, \theta') 
		\mapsto \partial f(\theta'|\theta) / \partial \theta$
	is continuous, and that 
	\begin{equation}
		\partial^2 f(\theta'|\theta) / \partial \theta^2 = 0
			\quad \mbox{if and only if} \quad
		\theta = \theta^*(\theta') = \tilde{a}_i \; e^{\ln \theta' / \rho}, i = 1, 2
	\end{equation}
	where 
	$\tilde{a}_1, \tilde{a}_2 
		= \exp \left[ 
							\frac{\gamma_u}{\rho}
							\left( -\frac{1}{2 \rho} \pm
									 \sqrt{\frac{1}{4 \rho^2} + \frac{1}{\gamma_u}}
							\right)
				    \right]$.
	If $\rho>0$, then 
	$\theta^*(\theta') \rightarrow \infty$ as $\theta' \rightarrow \infty$
	and $\theta^*(\theta') \rightarrow 0$ as $\theta' \rightarrow 0$. 
	If $\rho < 0$, then
	$\theta^*(\theta') \rightarrow 0$ as $\theta' \rightarrow \infty$
	and $\theta^*(\theta') \rightarrow \infty$ as $\theta' \rightarrow 0$. 
	Hence, assumption \ref{a:2nd_diff} holds. 
	Based on \eqref{eq:w'_bd}--\eqref{eq:eg_js_exog} and lemma \ref{lm:cont}, we
	can show that assumption \ref{a:cont_diff_gel} holds.  
	By proposition \ref{pr:cont_diff_gel}, $\psi^*$ is continuously differentiable. 
	Since assumption \ref{a:opt_pol} 
	holds and $r$ is continuously differentiable, by proposition \ref{pr:pol_diff},  
	$\bar{w}$ is continuously differentiable. $v^*$ is continuous since
	$v^* = r \vee \psi^*$.
	 
	 \textit{Proof of claim 4.} Assumption \ref{a:c_incre} holds since 
	 $c \equiv c_0$. Note that 
	 \begin{equation*}
	 	r(w) 
	 	 = r(\eta + \xi \theta) 
	 	 = (\eta + \xi \theta)^{1 - \delta} / [(1 - \beta) (1 - \delta)]
	 \end{equation*}
	 is increasing in $\theta$, and, when $\rho >0$, $f(\theta' | \theta)$ is 
	 stochastically increasing in $\theta$. Hence, assumption \ref{a:mono_map} 
	 holds.
	 By propositions \ref{pr:mono} and \ref{pr:pol_mon}, $\psi^*$ and $\bar{w}$ 
	 are increasing in $\theta$. Moreover, $r$ is a function of $w$, $\psi^*$ is a 
	 function of $\theta$, both functions are increasing, and $v^* = r \vee \psi^*$.
	 Hence, $v^*$ is increasing in $w$ and $\theta$.
\end{proof}

\begin{proof}[Proof of proposition \ref{pr:js_exog_log}]

	Since $|\ln a| \leq 1 / a + a$ for all $a>0$, we have
	\begin{align*}
		|u(w')| = |\ln w' |
				   = |\ln (\eta' + \theta' \xi')| 
				  \leq 1/ \eta' + \eta' + \theta' \xi' .
	\end{align*}
	Hence,
	\begin{align*}
		\int |u(w')| P(z, \diff z')
		&\leq \int (1 / \eta' + \eta') v(\eta') \diff \eta' +
				  \int \xi' h(\xi') \diff \xi' 
				  	\cdot
			 	  \int \theta' f(\theta'| \theta) \diff \theta'		\\
		&= \left(
			  		e^{-\mu_{\eta} + \gamma_{\eta} / 2} +
			  		e^{\mu_{\eta} + \gamma_{\eta} / 2} 
			  \right) +		
			  e^{ (\gamma_{\xi} + \gamma_u) / 2}	  
			  		\cdot
			  \theta^{\rho}.
	\end{align*}
	Induction shows that
	\begin{equation}
		\int |u(w')| P^t (z, \diff z')
		\leq 
			a_1^{(t)} + a_2^{(t)} \; \theta^{\rho^t}
		\leq 
			a_1^{(t)} + a_2^{(t)} \left( 
													\theta^{\rho^t} + 
													\theta^{- \rho^t} 
											\right)
	\end{equation}
	for some $a_1^{(t)}$, $a_2^{(t)} > 0$ and all $t \in \NN$. Hence, we can define 
	$g$ as in the assumption. Similarly as in the proof 	of proposition \ref{pr:js_exog}, 
	we can show that
	\begin{align}
	\label{eq:eg_js_exog_log}
		\int g(\theta') 
			   f(\theta' | \theta) 
		\diff \theta'
		&= \left(
					 e^{ \rho^{n+1} \ln \theta } +  
				 	 e^{ -\rho^{n+1} \ln \theta }
			 \right) 
			 e^{\rho^{2n} \gamma_u / 2 }		\\
		& \leq 	 
			 \left(
			 		e^{\rho^n \ln \theta}   +  
					e^{ -\rho^n \ln \theta} + 1
			 \right)
			 e^{ \rho^{2n} \gamma_u / 2 }		\nonumber 	\\
		& = 
			 \left(
			 		g(\theta) + 1
			 \right)
			 e^{ \rho^{2n} \gamma_u / 2 }	
		\leq m g(\theta) + d. 		\nonumber	
	\end{align}
	Hence, assumption \ref{a:ubdd_drift_gel} holds. Claim 1 then follows from
	theorem \ref{t:bk}.	
	The remaining proof is similar to proposition \ref{pr:js_exog}.
\end{proof}

\subsection{Proof of Section \ref{s:extension} Results}

\begin{proof}[Proof of theorem \ref{thm:keythm_ext_1}]

Regarding claim 1, similar to the proof of theorem \ref{t:bk}, we can show that
\begin{align*}
	\int \kappa(z') P(z, \diff z')
	\leq (m + 2 m') \kappa (z)
\end{align*}
for all $z \in \ZZ$. We next show that 
$L \colon (b_{\kappa} \ZZ \times b_{\kappa} \ZZ, \rho_{\kappa})
				\rightarrow
				(b_{\kappa} \ZZ \times b_{\kappa} \ZZ, \rho_{\kappa})$.
For all $h := (\psi, r) \in b_{\kappa} \ZZ \times b_{\kappa} \ZZ$, define the 
functions
$p(z) := c(z) + \beta \int \max \{ r(z'), \psi(z')\} P(z, \diff z')$ and 
$q(z) := s(z) + \alpha \beta \int \max \{ r(z'), \psi(z') \} P(z, \diff z')
			  + (1 - \alpha) \beta \int r(z') P(z, \diff z')$.
Then there exists $G \in \RR_+$ such that for all $z \in \ZZ$,
\begin{align*}
	\frac{|p(z)|}{\kappa(z)}
	\leq \frac{|c(z)|}{\kappa(z)}
			+ \frac{\beta G \int \kappa(z') P(z, \diff z')}{\kappa(z)}
	\leq \frac{1}{m'} + \beta (m + 2m') G < \infty
\end{align*}
and 
\begin{align*}
	\frac{|q(z)|}{\kappa(z)} 
	\leq \frac{|s(z)|}{\kappa(z)}
			+ \frac{\beta G \int \kappa(z') P(z, \diff z')}{\kappa(z)}
	\leq \frac{1}{m'} + \beta (m + 2m') G < \infty.
\end{align*}	
This implies that $p \in b_{\kappa} \ZZ$ and $q \in b_{\kappa} \ZZ$. Hence,
$L h \in b_{\kappa} \ZZ \times b_{\kappa} \ZZ$. Next, we show that $L$ is indeed a contraction mapping on $(b_{\kappa} \ZZ \times b_{\kappa} \ZZ, \rho_{\kappa})$. For all fixed $h_1 := (\psi_1, r_1)$ and $h_2 := (\psi_2, r_2)$ in 
$b_{\kappa} \ZZ \times b_{\kappa} \ZZ$, we have 
$\rho_{\kappa}(Lh_1, Lh_2) = I \vee J$, where
\begin{equation*}
    I := \| \beta P (r_1 \vee \psi_1) - \beta P(r_2 \vee \psi_2) \|_{\kappa}
\end{equation*}
and
\begin{equation*}
    J := \| \alpha \beta [P(r_1 \vee \psi_1) - P(r_2 \vee \psi_2)] +  (1 -
        \alpha) \beta (P r_1 - Pr_2) \|_{\kappa}.
\end{equation*}
For all $z \in \ZZ$, we have
\begin{align*}
	& \left| \int 
					( r_1 \vee \psi_1 )(z') 
				  P(z, \diff z')
				 -
				 \int 
				 	( r_2 \vee \psi_2 )(z') 
				  P(z, \diff z')
		\right|				\\
	& \leq \int
					\left| r_1 \vee \psi_1 
							 - r_2 \vee \psi_2
					\right| (z')
				P(z, \diff z')
	\leq \int 
				( |\psi_1 - \psi_2| \vee |r_1 - r_2| ) (z')
			P(z, \diff z')			\\
	& \leq ( \| \psi_1 - \psi_2 \|_{\kappa} 
				 \vee 
				 \| r_1 - r_2 \|_{\kappa} )
			   \int \kappa(z') P(z, \diff z')
	\leq \rho_{\kappa} (h_1, h_2) 
			(m + 2m') \kappa(z),
\end{align*}
where the second inequality is due to the elementary fact 
$|a \vee b - a' \vee b'| \leq |a-a'| \vee |b-b'|$. This implies that 
$I \leq \beta (m+2m') \rho_{\kappa}(h_1, h_2)$.  Regarding $J$, similar arguments yield $J \leq \beta (m + 2m') \rho_{\kappa}(h_1, h_2)$. In conclusion, we have 
\begin{equation}
	\rho_{\kappa}(L h_1, L h_2) 
	= I \vee J
	\leq \beta (m + 2m')
			\rho_{\kappa} (h_1, h_2).
\end{equation}
Hence, $L$ is a contraction mapping on 
$(b_{\kappa} \ZZ \times b_{\kappa} \ZZ, \rho_{\kappa})$ 
with modulus $\beta (m + 2m')$, as was to be shown. Claim 1 is verified.

Since $v^*$ and $r^*$ satisfy \eqref{eq:rst1}--\eqref{eq:rst2}, by \eqref{eq:rst3},  $h^* := (\psi^*, r^*)$ is indeed a fixed point of $L$. To prove that claim 2 holds, it remains to show that $h^* \in b_{\kappa} \ZZ \times b_{\kappa} \ZZ$. Since
\begin{equation*}
	\max \{ |r^*(z)|, | \psi^*(z)| \} 
		\leq 
			  \sum_{t=0}^{\infty} 
			  \beta^t 
			  \EE_z [ |s(Z_t)| + g(Z_t) ],
\end{equation*}
this can be proved in a similar way as lemma \ref{lm:bd_vcv}. Hence, claim 2 is verified.
\end{proof}

\bibliographystyle{ecta}

\bibliography{localbib}

\begin{thebibliography}{67}
\newcommand{\enquote}[1]{``#1''}
\expandafter\ifx\csname natexlab\endcsname\relax\def\natexlab#1{#1}\fi

\bibitem[\protect\citeauthoryear{Albuquerque and Hopenhayn}{Albuquerque and
  Hopenhayn}{2004}]{albuquerque2004optimal}
\textsc{Albuquerque, R. and H.~A. Hopenhayn} (2004): \enquote{Optimal lending
  contracts and firm dynamics,} \emph{The Review of Economic Studies}, 71,
  285--315.

\bibitem[\protect\citeauthoryear{Alfaro and Kanczuk}{Alfaro and
  Kanczuk}{2009}]{alfaro2009optimal}
\textsc{Alfaro, L. and F.~Kanczuk} (2009): \enquote{Optimal reserve management
  and sovereign debt,} \emph{Journal of International Economics}, 77, 23--36.

\bibitem[\protect\citeauthoryear{Alvarez and Stokey}{Alvarez and
  Stokey}{1998}]{alvarez1998dynamic}
\textsc{Alvarez, F. and N.~L. Stokey} (1998): \enquote{Dynamic programming with
  homogeneous functions,} \emph{Journal of Economic Theory}, 82, 167--189.

\bibitem[\protect\citeauthoryear{Arellano}{Arellano}{2008}]{arellano2008default}
\textsc{Arellano, C.} (2008): \enquote{Default risk and income fluctuations in
  emerging economies,} \emph{The American Economic Review}, 98, 690--712.

\bibitem[\protect\citeauthoryear{Arellano and Ramanarayanan}{Arellano and
  Ramanarayanan}{2012}]{arellano2012default}
\textsc{Arellano, C. and A.~Ramanarayanan} (2012): \enquote{Default and the
  maturity structure in sovereign bonds,} \emph{Journal of Political Economy},
  120, 187--232.

\bibitem[\protect\citeauthoryear{Asplund and Nocke}{Asplund and
  Nocke}{2006}]{asplund2006firm}
\textsc{Asplund, M. and V.~Nocke} (2006): \enquote{Firm turnover in imperfectly
  competitive markets,} \emph{The Review of Economic Studies}, 73, 295--327.

\bibitem[\protect\citeauthoryear{Bagger, Fontaine, Postel-Vinay, and
  Robin}{Bagger et~al.}{2014}]{bagger2014tenure}
\textsc{Bagger, J., F.~Fontaine, F.~Postel-Vinay, and J.-M. Robin} (2014):
  \enquote{Tenure, experience, human capital, and wages: A tractable
  equilibrium search model of wage dynamics,} \emph{The American Economic
  Review}, 104, 1551--1596.

\bibitem[\protect\citeauthoryear{Bai and Zhang}{Bai and
  Zhang}{2012}]{bai2012financial}
\textsc{Bai, Y. and J.~Zhang} (2012): \enquote{Financial integration and
  international risk sharing,} \emph{Journal of International Economics}, 86,
  17--32.

\bibitem[\protect\citeauthoryear{Becker and Boyd}{Becker and
  Boyd}{1997}]{becker1997capital}
\textsc{Becker, R.~A. and J.~H. Boyd} (1997): \emph{Capital Theory, Equilibrium
  Analysis, and Recursive Utility}, Wiley-Blackwell.

\bibitem[\protect\citeauthoryear{Bellman}{Bellman}{1969}]{bellman1969new}
\textsc{Bellman, R.} (1969): \enquote{A new type of approximation leading to
  reduction of dimensionality in control processes,} \emph{Journal of
  Mathematical Analysis and Applications}, 27, 454--459.

\bibitem[\protect\citeauthoryear{Bental and Peled}{Bental and
  Peled}{1996}]{bental1996accumulation}
\textsc{Bental, B. and D.~Peled} (1996): \enquote{The accumulation of wealth
  and the cyclical generation of new technologies: A search theoretic
  approach,} \emph{International Economic Review}, 37, 687--718.

\bibitem[\protect\citeauthoryear{Boyd}{Boyd}{1990}]{boud1990recursive}
\textsc{Boyd, J.~H.} (1990): \enquote{Recursive utility and the Ramsey
  problem,} \emph{Journal of Economic Theory}, 50, 326--345.

\bibitem[\protect\citeauthoryear{Bull and Jovanovic}{Bull and
  Jovanovic}{1988}]{bull1988mismatch}
\textsc{Bull, C. and B.~Jovanovic} (1988): \enquote{Mismatch versus
  derived-demand shift as causes of labour mobility,} \emph{The Review of
  Economic Studies}, 55, 169--175.

\bibitem[\protect\citeauthoryear{Burdett and Vishwanath}{Burdett and
  Vishwanath}{1988}]{burdett1988declining}
\textsc{Burdett, K. and T.~Vishwanath} (1988): \enquote{Declining reservation
  wages and learning,} \emph{The Review of Economic Studies}, 55, 655--665.

\bibitem[\protect\citeauthoryear{Chalkley}{Chalkley}{1984}]{chalkley1984adaptive}
\textsc{Chalkley, M.} (1984): \enquote{Adaptive job search and null offers: A
  model of quantity constrained search,} \emph{The Economic Journal}, 94,
  148--157.

\bibitem[\protect\citeauthoryear{Chatterjee and Eyigungor}{Chatterjee and
  Eyigungor}{2012}]{chatterjee2012maturity}
\textsc{Chatterjee, S. and B.~Eyigungor} (2012): \enquote{Maturity,
  indebtedness, and default risk,} \emph{The American Economic Review}, 102,
  2674--2699.

\bibitem[\protect\citeauthoryear{Chatterjee and Rossi-Hansberg}{Chatterjee and
  Rossi-Hansberg}{2012}]{chatterjee2012spinoffs}
\textsc{Chatterjee, S. and E.~Rossi-Hansberg} (2012): \enquote{Spinoffs and the
  Market for Ideas,} \emph{International Economic Review}, 53, 53--93.

\bibitem[\protect\citeauthoryear{Choi, Laibson, Madrian, and Metrick}{Choi
  et~al.}{2003}]{choi2003optimal}
\textsc{Choi, J.~J., D.~Laibson, B.~C. Madrian, and A.~Metrick} (2003):
  \enquote{Optimal defaults,} \emph{The American Economic Review}, 93,
  180--185.

\bibitem[\protect\citeauthoryear{Cooper, Haltiwanger, and Willis}{Cooper
  et~al.}{2007}]{cooper2007search}
\textsc{Cooper, R., J.~Haltiwanger, and J.~L. Willis} (2007): \enquote{Search
  frictions: Matching aggregate and establishment observations,} \emph{Journal
  of Monetary Economics}, 54, 56--78.

\bibitem[\protect\citeauthoryear{Co{\c{s}}ar, Guner, and Tybout}{Co{\c{s}}ar
  et~al.}{2016}]{cocsar2016firm}
\textsc{Co{\c{s}}ar, A.~K., N.~Guner, and J.~Tybout} (2016): \enquote{Firm
  dynamics, job turnover, and wage distributions in an open economy,} \emph{The
  American Economic Review}, 106, 625--663.

\bibitem[\protect\citeauthoryear{Crawford and Shum}{Crawford and
  Shum}{2005}]{crawford2005uncertainty}
\textsc{Crawford, G.~S. and M.~Shum} (2005): \enquote{Uncertainty and learning
  in pharmaceutical demand,} \emph{Econometrica}, 73, 1137--1173.

\bibitem[\protect\citeauthoryear{Dinlersoz and Yorukoglu}{Dinlersoz and
  Yorukoglu}{2012}]{dinlersoz2012information}
\textsc{Dinlersoz, E.~M. and M.~Yorukoglu} (2012): \enquote{Information and
  industry dynamics,} \emph{The American Economic Review}, 102, 884--913.

\bibitem[\protect\citeauthoryear{Dixit and Pindyck}{Dixit and
  Pindyck}{1994}]{dixit1994investment}
\textsc{Dixit, A.~K. and R.~S. Pindyck} (1994): \emph{Investment Under
  Uncertainty}, Princeton University Press.

\bibitem[\protect\citeauthoryear{Duffie}{Duffie}{2010}]{duffie2010dynamic}
\textsc{Duffie, D.} (2010): \emph{Dynamic Asset Pricing Theory}, Princeton
  University Press.

\bibitem[\protect\citeauthoryear{Dunne, Klimek, Roberts, and Xu}{Dunne
  et~al.}{2013}]{dunne2013entry}
\textsc{Dunne, T., S.~D. Klimek, M.~J. Roberts, and D.~Y. Xu} (2013):
  \enquote{Entry, exit, and the determinants of market structure,} \emph{The
  RAND Journal of Economics}, 44, 462--487.

\bibitem[\protect\citeauthoryear{Dur{\'a}n}{Dur{\'a}n}{2000}]{duran2000dynamic}
\textsc{Dur{\'a}n, J.} (2000): \enquote{On dynamic programming with unbounded
  returns,} \emph{Economic Theory}, 15, 339--352.

\bibitem[\protect\citeauthoryear{Dur{\'a}n}{Dur{\'a}n}{2003}]{duran2003discounting}
---\hspace{-.1pt}---\hspace{-.1pt}--- (2003): \enquote{Discounting long run
  average growth in stochastic dynamic programs,} \emph{Economic Theory}, 22,
  395--413.

\bibitem[\protect\citeauthoryear{Ericson and Pakes}{Ericson and
  Pakes}{1995}]{ericson1995markov}
\textsc{Ericson, R. and A.~Pakes} (1995): \enquote{Markov-perfect industry
  dynamics: A framework for empirical work,} \emph{The Review of Economic
  Studies}, 62, 53--82.

\bibitem[\protect\citeauthoryear{Fajgelbaum, Schaal, and
  Taschereau-Dumouchel}{Fajgelbaum et~al.}{2015}]{fajgelbaum2015uncertainty}
\textsc{Fajgelbaum, P., E.~Schaal, and M.~Taschereau-Dumouchel} (2015):
  \enquote{Uncertainty traps,} Tech. rep., NBER Working Paper.

\bibitem[\protect\citeauthoryear{Feinberg, Kasyanov, and Zadoianchuk}{Feinberg
  et~al.}{2014}]{feinberg2014fatou}
\textsc{Feinberg, E.~A., P.~O. Kasyanov, and N.~V. Zadoianchuk} (2014):
  \enquote{Fatou's lemma for weakly converging probabilities,} \emph{Theory of
  Probability \& Its Applications}, 58, 683--689.

\bibitem[\protect\citeauthoryear{Gomes, Greenwood, and Rebelo}{Gomes
  et~al.}{2001}]{gomes2001equilibrium}
\textsc{Gomes, J., J.~Greenwood, and S.~Rebelo} (2001): \enquote{Equilibrium
  unemployment,} \emph{Journal of Monetary Economics}, 48, 109--152.

\bibitem[\protect\citeauthoryear{Hatchondo, Martinez, and
  Sosa-Padilla}{Hatchondo et~al.}{2016}]{hatchondo2016debt}
\textsc{Hatchondo, J.~C., L.~Martinez, and C.~Sosa-Padilla} (2016):
  \enquote{Debt dilution and sovereign default risk,} \emph{Journal of
  Political Economy}, 124, 1383--1422.

\bibitem[\protect\citeauthoryear{Hopenhayn}{Hopenhayn}{1992}]{hopenhayn1992entry}
\textsc{Hopenhayn, H.~A.} (1992): \enquote{Entry, exit, and firm dynamics in
  long run equilibrium,} \emph{Econometrica}, 1127--1150.

\bibitem[\protect\citeauthoryear{Jovanovic}{Jovanovic}{1982}]{jovanovic1982selection}
\textsc{Jovanovic, B.} (1982): \enquote{Selection and the evolution of
  industry,} \emph{Econometrica}, 649--670.

\bibitem[\protect\citeauthoryear{Jovanovic}{Jovanovic}{1987}]{jovanovic1987work}
---\hspace{-.1pt}---\hspace{-.1pt}--- (1987): \enquote{Work, rest, and search:
  unemployment, turnover, and the cycle,} \emph{Journal of Labor Economics},
  131--148.

\bibitem[\protect\citeauthoryear{Jovanovic and Rob}{Jovanovic and
  Rob}{1989}]{jovanovic1989growth}
\textsc{Jovanovic, B. and R.~Rob} (1989): \enquote{The growth and diffusion of
  knowledge,} \emph{The Review of Economic Studies}, 56, 569--582.

\bibitem[\protect\citeauthoryear{Kambourov and Manovskii}{Kambourov and
  Manovskii}{2009}]{kambourov2009occupational}
\textsc{Kambourov, G. and I.~Manovskii} (2009): \enquote{Occupational mobility
  and wage inequality,} \emph{The Review of Economic Studies}, 76, 731--759.

\bibitem[\protect\citeauthoryear{Kellogg}{Kellogg}{2014}]{kellogg2014effect}
\textsc{Kellogg, R.} (2014): \enquote{The effect of uncertainty on investment:
  evidence from Texas oil drilling,} \emph{The American Economic Review}, 104,
  1698--1734.

\bibitem[\protect\citeauthoryear{Le~Van and Vailakis}{Le~Van and
  Vailakis}{2005}]{le2005recursive}
\textsc{Le~Van, C. and Y.~Vailakis} (2005): \enquote{Recursive utility and
  optimal growth with bounded or unbounded returns,} \emph{Journal of Economic
  Theory}, 123, 187--209.

\bibitem[\protect\citeauthoryear{Lise}{Lise}{2013}]{lise2012job}
\textsc{Lise, J.} (2013): \enquote{On-the-job search and precautionary
  savings,} \emph{The Review of Economic Studies}, 80, 1086--1113.

\bibitem[\protect\citeauthoryear{Ljungqvist and Sargent}{Ljungqvist and
  Sargent}{2008}]{ljungqvist2008two}
\textsc{Ljungqvist, L. and T.~J. Sargent} (2008): \enquote{Two questions about
  European unemployment,} \emph{Econometrica}, 76, 1--29.

\bibitem[\protect\citeauthoryear{Ljungqvist and Sargent}{Ljungqvist and
  Sargent}{2012}]{ljungqvist2012recursive}
---\hspace{-.1pt}---\hspace{-.1pt}--- (2012): \emph{Recursive Macroeconomic
  Theory}, MIT Press.

\bibitem[\protect\citeauthoryear{Low, Meghir, and Pistaferri}{Low
  et~al.}{2010}]{low2010wage}
\textsc{Low, H., C.~Meghir, and L.~Pistaferri} (2010): \enquote{Wage risk and
  employment risk over the life cycle,} \emph{The American Economic Review},
  100, 1432--1467.

\bibitem[\protect\citeauthoryear{Lucas and Prescott}{Lucas and
  Prescott}{1974}]{lucas1974equilibrium}
\textsc{Lucas, R.~E. and E.~C. Prescott} (1974): \enquote{Equilibrium search
  and unemployment,} \emph{Journal of Economic Theory}, 7, 188--209.

\bibitem[\protect\citeauthoryear{Martins-da Rocha and Vailakis}{Martins-da
  Rocha and Vailakis}{2010}]{martins2010existence}
\textsc{Martins-da Rocha, V.~F. and Y.~Vailakis} (2010): \enquote{Existence and
  uniqueness of a fixed point for local contractions,} \emph{Econometrica}, 78,
  1127--1141.

\bibitem[\protect\citeauthoryear{Matkowski and Nowak}{Matkowski and
  Nowak}{2011}]{matkowski2011discounted}
\textsc{Matkowski, J. and A.~S. Nowak} (2011): \enquote{On discounted dynamic
  programming with unbounded returns,} \emph{Economic Theory}, 46, 455--474.

\bibitem[\protect\citeauthoryear{McCall}{McCall}{1970}]{mccall1970}
\textsc{McCall, J.~J.} (1970): \enquote{Economics of information and job
  search,} \emph{The Quarterly Journal of Economics}, 84, 113--126.

\bibitem[\protect\citeauthoryear{Mendoza and Yue}{Mendoza and
  Yue}{2012}]{mendoza2012general}
\textsc{Mendoza, E.~G. and V.~Z. Yue} (2012): \enquote{A general equilibrium
  model of sovereign default and business cycles,} \emph{The Quarterly Journal
  of Economics}, 127, 889--946.

\bibitem[\protect\citeauthoryear{Menzio and Trachter}{Menzio and
  Trachter}{2015}]{menzio2015equilibrium}
\textsc{Menzio, G. and N.~Trachter} (2015): \enquote{Equilibrium price
  dispersion with sequential search,} \emph{Journal of Economic Theory}, 160,
  188--215.

\bibitem[\protect\citeauthoryear{Meyn and Tweedie}{Meyn and
  Tweedie}{2012}]{meyn2012markov}
\textsc{Meyn, S.~P. and R.~L. Tweedie} (2012): \emph{Markov Chains and
  Stochastic Stability}, Springer Science \& Business Media.

\bibitem[\protect\citeauthoryear{Mitchell}{Mitchell}{2000}]{mitchell2000scope}
\textsc{Mitchell, M.~F.} (2000): \enquote{The scope and organization of
  production: firm dynamics over the learning curve,} \emph{The Rand Journal of
  Economics}, 180--205.

\bibitem[\protect\citeauthoryear{Moscarini and Postel-Vinay}{Moscarini and
  Postel-Vinay}{2013}]{moscarini2013stochastic}
\textsc{Moscarini, G. and F.~Postel-Vinay} (2013): \enquote{Stochastic search
  equilibrium,} \emph{The Review of Economic Studies}, 80, 1545--1581.

\bibitem[\protect\citeauthoryear{Nagyp{\'a}l}{Nagyp{\'a}l}{2007}]{nagypal2007learning}
\textsc{Nagyp{\'a}l, {\'E}.} (2007): \enquote{Learning by doing vs. learning
  about match quality: Can we tell them apart?} \emph{The Review of Economic
  Studies}, 74, 537--566.

\bibitem[\protect\citeauthoryear{Pakes and Ericson}{Pakes and
  Ericson}{1998}]{pakes1998empirical}
\textsc{Pakes, A. and R.~Ericson} (1998): \enquote{Empirical implications of
  alternative models of firm dynamics,} \emph{Journal of Economic Theory}, 79,
  1--45.

\bibitem[\protect\citeauthoryear{Perla and Tonetti}{Perla and
  Tonetti}{2014}]{perla2014equilibrium}
\textsc{Perla, J. and C.~Tonetti} (2014): \enquote{Equilibrium imitation and
  growth,} \emph{Journal of Political Economy}, 122, 52--76.

\bibitem[\protect\citeauthoryear{Peskir and Shiryaev}{Peskir and
  Shiryaev}{2006}]{peskir2006}
\textsc{Peskir, G. and A.~Shiryaev} (2006): \emph{Optimal Stopping and
  Free-boundary Problems}, Springer.

\bibitem[\protect\citeauthoryear{Poschke}{Poschke}{2010}]{poschke2010regulation}
\textsc{Poschke, M.} (2010): \enquote{The regulation of entry and aggregate
  productivity,} \emph{The Economic Journal}, 120, 1175--1200.

\bibitem[\protect\citeauthoryear{Pries and Rogerson}{Pries and
  Rogerson}{2005}]{pries2005hiring}
\textsc{Pries, M. and R.~Rogerson} (2005): \enquote{Hiring policies, labor
  market institutions, and labor market flows,} \emph{Journal of Political
  Economy}, 113, 811--839.

\bibitem[\protect\citeauthoryear{Rendon}{Rendon}{2006}]{rendon2006job}
\textsc{Rendon, S.} (2006): \enquote{Job search and asset accumulation under
  borrowing constraints,} \emph{International Economic Review}, 47, 233--263.

\bibitem[\protect\citeauthoryear{Rinc{\'o}n-Zapatero and
  Rodr{\'\i}guez-Palmero}{Rinc{\'o}n-Zapatero and
  Rodr{\'\i}guez-Palmero}{2003}]{rincon2003existence}
\textsc{Rinc{\'o}n-Zapatero, J.~P. and C.~Rodr{\'\i}guez-Palmero} (2003):
  \enquote{Existence and uniqueness of solutions to the Bellman equation in the
  unbounded case,} \emph{Econometrica}, 71, 1519--1555.

\bibitem[\protect\citeauthoryear{Rinc{\'o}n-Zapatero and
  Rodr{\'\i}guez-Palmero}{Rinc{\'o}n-Zapatero and
  Rodr{\'\i}guez-Palmero}{2009}]{rincon2009corrigendum}
---\hspace{-.1pt}---\hspace{-.1pt}--- (2009): \enquote{Corrigendum to
  “Existence and uniqueness of solutions to the Bellman equation in the
  unbounded case” Econometrica, Vol. 71, No. 5 (September, 2003),
  1519--1555,} \emph{Econometrica}, 77, 317--318.

\bibitem[\protect\citeauthoryear{Robin}{Robin}{2011}]{robin2011dynamics}
\textsc{Robin, J.-M.} (2011): \enquote{On the dynamics of unemployment and wage
  distributions,} \emph{Econometrica}, 79, 1327--1355.

\bibitem[\protect\citeauthoryear{Rust}{Rust}{1997}]{rust1997using}
\textsc{Rust, J.} (1997): \enquote{Using randomization to break the curse of
  dimensionality,} \emph{Econometrica}, 487--516.

\bibitem[\protect\citeauthoryear{Shiryaev}{Shiryaev}{1999}]{shiryaev1999essentials}
\textsc{Shiryaev, A.~N.} (1999): \emph{Essentials of Stochastic Finance: Facts,
  Models, Theory}, vol.~3, World scientific.

\bibitem[\protect\citeauthoryear{Stokey, Lucas, and Prescott}{Stokey
  et~al.}{1989}]{stokey1989}
\textsc{Stokey, N., R.~Lucas, and E.~Prescott} (1989): \emph{Recursive Methods
  in Economic Dynamics}, Harvard University Press.

\bibitem[\protect\citeauthoryear{Timoshenko}{Timoshenko}{2015}]{timoshenko2015product}
\textsc{Timoshenko, O.~A.} (2015): \enquote{Product switching in a model of
  learning,} \emph{Journal of International Economics}, 95, 233--249.

\bibitem[\protect\citeauthoryear{Vereshchagina and Hopenhayn}{Vereshchagina and
  Hopenhayn}{2009}]{vereshchagina2009risk}
\textsc{Vereshchagina, G. and H.~A. Hopenhayn} (2009): \enquote{Risk taking by
  entrepreneurs,} \emph{The American Economic Review}, 99, 1808--1830.

\end{thebibliography}

\end{document}